\documentclass[a4paper,11pt]{book}

	\usepackage[T1]{fontenc}
	\usepackage[utf8]{inputenc}

	\usepackage[british]{babel}

	\usepackage{gentium} 
	\usepackage{layaureo} 
	\usepackage{microtype}
	\usepackage{fancyhdr}
		\renewcommand{\headrulewidth}{0.4pt}
		\renewcommand{\footrulewidth}{0.4pt}
	\usepackage{emptypage} 
	
	\usepackage{enumitem}
		\setlist{labelindent=\parindent} 
		\setlist[enumerate]{label=\textit{(\roman*)},ref=(\roman*),align=left, leftmargin=*}

	\usepackage{tikz} 
	\usepackage{subfig} 

\usepackage{amsthm} 
	\newtheoremstyle{thm} 
  	{11pt}
  	{11pt}
 	{\itshape}
  	{ }
  	{\sc}
 	{\newline}
  	{ }
  	{ }

	\newtheoremstyle{remark} 
 	{6pt}
  	{6pt}
  	{\small}
  	{}
  	{\scshape}
  	{\newline}
  	{ }
  	{ }
	
	\theoremstyle{thm}
	\newtheorem{thm}{Theorem}[chapter]
	\newtheorem{prop}[thm]{Proposition}
	\newtheorem{cor}[thm]{Corollary}
	\newtheorem{lemma}[thm]{Lemma}
	\newtheorem{dfn}[thm]{Definition}
	
	\theoremstyle{remark}
	\newtheorem{rmk}[thm]{Remark}
	\newtheorem{exa}[thm]{Example}
	
\usepackage{amsmath,amssymb,mathrsfs,braket,mathtools,esint} 
	\numberwithin{equation}{chapter}
		\newcommand{\call}[1]{\mathscr{#1}}
		\newcommand{\E}{\call{E}}
		\newcommand{\F}{\call{F}}
		
		\newcommand{\M}{\call{M}}
		\newcommand{\eps}{\varepsilon}
		\renewcommand{\epsilon}{\eps}
		\renewcommand{\phi}{\varphi}
		\renewcommand{\theta}{\vartheta}
		\renewcommand{\set}[1]{\left\{\hspace{-0.1cm}~#1~\hspace{-0.1cm}\right\}}
		\newcommand{\N}{\mathbb{N}}
		\newcommand{\R}{\mathbb{R}}
		\newcommand{\Rd}{\R^d}
		\newcommand{\Rdmz}{\R^d\setminus\set{0}}
		\newcommand{\co}[1]{#1^\mathsf{c}}
		\newcommand{\symdif}{\!\bigtriangleup\!}
		\newcommand{\id}{\mathrm{id}}
		\newcommand{\va}[1]{\left\vert{#1}\right\vert}
		\DeclareMathOperator{\sign}{sign}
		\newcommand{\norm}[1]{\left\Vert#1\right\Vert}
		\newcommand{\dist}{\mathrm{dist}}
		\newcommand{\trasp}{\mathsf{t}} 
		\newcommand{\tr}[1]{\mathrm{tr}\left(#1\right)} 
		\newcommand{\Sdmu}{\mathbb{S}^{d-1}}
		\newcommand{\limseq}[1]{\lim_{#1 \to +\infty}}
		\newcommand{\D}{\mathrm{D}}
		\newcommand{\de}{\mathrm{d}}
		\newcommand{\der}[2]{\frac{\de #1}{\de #2}}
		\DeclareMathOperator{\dvg}{div}
		\renewcommand{\div}{\dvg}
		\newcommand{\Ld}{\call{L}^d}
		\newcommand{\Hdmu}{\call{H}^{d-1}}
		\newcommand{\supp}{\mathrm{supp}}
		\newcommand{\weakstar}{\stackrel{\ast}{\rightharpoonup}}
		\newcommand{\Cc}{C_\mathrm{c}}
		\newcommand{\Cuc}{C^1_\mathrm{c}}
		\newcommand{\loc}{\mathrm{loc}}
		\newcommand{\BV}{\mathrm{BV}}
		\DeclareMathOperator{\Per}{Per}
		\newcommand{\PerK}{\Per_K}

\usepackage[alphabetic]{amsrefs}

\usepackage{imakeidx} 
	\makeindex[intoc,columns=2]

\usepackage{comment}
\usepackage{hyperref}

\begin{document}

\frontmatter
		\fancypagestyle{empty}{
		\fancyhf{}
		\fancyfoot[C]{\vspace{2mm} A.Y. 2018-2019 \\ Department of Mathematics, University of Pisa}
		\renewcommand{\headrulewidth}{0pt}
		\renewcommand{\footrulewidth}{0.1mm}
	}

	\begin{titlepage}
	\thispagestyle{empty}
	\begin{center}
		\begin{figure}[h!]
			\begin{center}
			\includegraphics[scale=0.4]{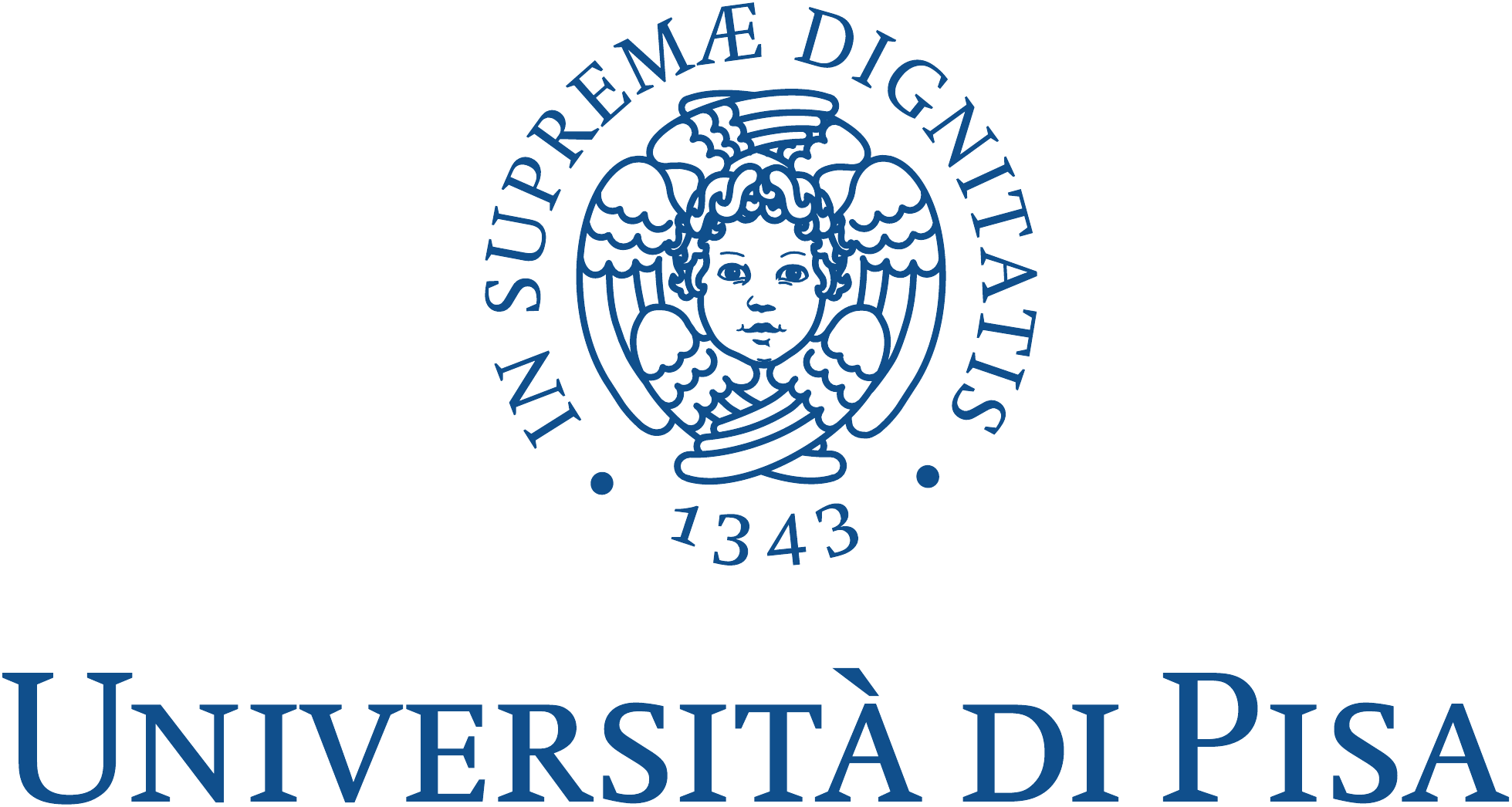}
			\end{center}
		\end{figure}
		\rule[4mm]{13.9cm}{0.1mm}
		\textbf{Ph.D. Course  in Mathematics}\\
		\vspace{1.5mm}
		\textsc{xxxii} cycle
	\end{center}

	\vspace{3cm}

	\begin{center}
	{\huge\textbf{Asymptotic behaviour}}
	
	\vspace{2mm}
	{\huge\textbf{of rescaled nonlocal functionals}}
	
	\vspace{2mm}
	{\huge\textbf{and evolutions}}
	\end{center}

	\vspace{4cm}
	\par\noindent

	\begin{minipage}[t]{0.47\textwidth}
	{\large Candidate:\\ \textbf {Valerio Pagliari}}
	\end{minipage}
	\hfill
	\begin{minipage}[t]{0.47\textwidth}\raggedleft
	{\large Supervisor:\\ \textbf {Matteo Novaga}}
	\end{minipage}
	\end{titlepage}

\fancypagestyle{empty}{
		\fancyhf{}
		\renewcommand{\headrulewidth}{0pt}
		\renewcommand{\footrulewidth}{0pt}
	}
\newpage\shipout\null
		\thispagestyle{empty}
\null\vspace{\stretch{1}}
\begin{flushright}
	<<Quello dell'articolo da citare \`e un grande classico>>
\end{flushright}
\vspace{\stretch{2}}
\newpage\shipout\null

\frontmatter
\pagestyle{plain}
		\section*{Abstract}
Taking up a variational viewpoint,
we present some nonlocal-to-local asymptotic results
for various kinds of integral functionals.
The content of the thesis comprises the contributions first appeared
in the research papers \cites{BP,CP,Pa,CNP}
in collaboration with J. Berendsen, A. Cesaroni, A. Chambolle, and M. Novaga.

After an initial summary of the basic technical tools,
the first original result is discussed in Chapter \ref{ch:nle}.
It is motivated by the work \cite{BBM} by J. Bourgain, H. Brezis, and P. Mironescu,
who proved that the $L^p$-norm of the gradient of a Sobolev function can be recovered
as a suitable limit of iterated integrals involving the difference quotients of the function.
A. Ponce later showed that the relation also holds in the sense of $\Gamma$-convergence \cite{P}.
Loosely speaking, we take into account the rate of this convergence
and we establish the $\Gamma$-converge of the rate functionals to a second order limit
w.r.t. the $H^1_\loc(\Rd)$-convergence.

Next, from Chapter \ref{ch:nlpnlc} on,
we move to a geometric context
and we consider the nonlocal perimeters
associated with a positive kernel $K$, which we allow to be singular in the origin.
Qualitatively, these functionals express a weighted interaction between a given set and its complement.
More broadly, we study a total-variation-type nonlocal functional $J_K(\,\cdot\,;\Omega)$,
where $\Omega\subset \Rd$ is a measurable set.
We establish existence of minimisers of such energy under prescribed boundary conditions,
and we prove a criterion for minimality based on calibrations.
Due to the nonlocal nature of the problem at stake,
the definition of calibration has to be properly chosen.
As an application of the criterion, we prove that
halfspaces are the unique minimisers of $J_K$ in a ball
subject to their own boundary conditions.

A second nonlocal-to-local $\Gamma$-convergence result is discussed in Chapter \ref{ch:judith}.
We rescale the kernel $K$ so that, when the scaling parameter approaches $0$,
the family of rescaled functions tends to the Dirac delta in $0$.
If $K$ has small tails at infinity,
we manage to show that the nonlocal total variations associated with the rescaled kernels
$\Gamma$-converge w.r.t. the $L^1_\loc(\Rd)$-convergence
to a local, anisotropic total variation.

Lastly, we consider the nonlocal curvature functional associated with $K$,
which is the geometric $L^2$-first variation of the nonlocal perimeter.
In the same asymptotic regime as above,
we retrieve a local, anisotropic mean curvature functional as the limit of rescaled nonlocal curvatures.
In particular, the limit is uniform for sets whose boundary is compact and smooth.
As a consequence, we  establish the locally uniform convergence
of the viscosity solutions of the rescaled  nonlocal  geometric flows  
to the viscosity solution of the anisotropic mean curvature motion.
This is obtained by combining
a compactness argument and a set-theoretic approach
that relies on the theory of De Giorgi's barriers for evolution equations.

\paragraph{Acknowledgements.}
I owe so much to the commitment of the people
with whom I have collaborated during the last three years.
I would like to thank Annalisa Cesaroni
for showing me how to do this job
and sharing her opinions with me.

Part of the work presented in this thesis was carried out
during a stay at the {\it Centre de Math\'ematiques Appliqu\'ees}
of the {\'Ecole Polytechnique} in Palaiseau.
I thank the Unione Matematica Italiana,
which partly funded the stay, and
Antonin Chambolle for the time he devoted to me.

Finally, I am deeply grateful to my friend Marco Pozzetta
for being so helpful, and patient, and enthusiastic, and\dots
		\tableofcontents
		\chapter*{Introduction}
		\addcontentsline{toc}{chapter}{Introduction}
		We present some results that describe localisation effects
for certain classes of nonlocal functionals
which are defined by means of integral kernels.
The underlying intuition is that,
if we rescale the latter in such a way that
they approximate the Dirac delta in the origin
as the scaling parameter tends to $0$,
then in the limit we retrieve some \emph{local} functionals. 
The precise sense in which the nonlocal-to-local convergence holds
will be specified case by case.

We devote these introductory pages to the overview of our main results,
for which we provide some motivation and background.
At the same time, we outline the structure of the thesis.
Our main concern here is conveying general ideas,
rather than stating results with full precision.
To give a wider perspective to the possibly interested readers,
we hint also at some of the literature which, for the sake of brevity,
though being related to the subjects we treat,
is not mentioned in the main body of the work.

The first chapter includes the basics
about functions of bounded variations, finite perimeter sets, and $\Gamma$-convergence
needed in the sequel.
Then, in Chapter \ref{ch:nle}, we report the results of \cite{CNP}.

We start from the analysis carried out 
by J. Bourgain, H. Brezis, and P. Mironescu in \cite{BBM},
which has attracted much attention in the years:
as an extremely brief selection of follow-ups,
we mention \cites{MS,Da,P,AK,LS}.
The main result of \cite{BBM} consists in the approximation
of the $L^p$-norm of the gradient of a Sobolev function
by means of iterated integrals involving difference quotients.
In \cite{P}, A. Ponce extended the result to the functionals
	\[
		\F_\epsilon(u) \coloneqq \int_{\Omega}\int_{\Omega}
			\rho_\epsilon(y-x) f\left(\frac{\va{u(y)-u(x)}}{\va{y-x}}\right)\de y \de x,
	\]
where $u\in L^p(\Omega)$ for some $p\geq 1$,
$\Set{\rho_\eps}$ is a suitable family of functions in $L^1(\Rd)$,
and $f\colon [0,+\infty) \to [0,+\infty)$ is convex.
As $\epsilon\to 0^+$,
he proved both pointwise and $\Gamma$-convergence of $\Set{\F_\eps}$
to a first order limit functional $\F_0$.
When $\rho_\epsilon(x) \coloneqq \epsilon^{-d}\rho(\epsilon^{-1}x)$
for some $\rho\in L^1(\Rd)$
and $u\in W^{1,p}(\Omega)$ is a Sobolev function, we find
	\[
		\F_0(u) = \int_{\Omega}\int_{\Rd} \rho(z) f(\va{\nabla u(x)\cdot \hat{z}}) \de z \de x,
	\]
where, for $z\in\Rdmz$, $\hat{z} \coloneqq z / \va{z}$.
We address the asymptotic behaviour of the differences $\F_0 - \F_\eps$,
which express the rate of convergence of the $\epsilon$-functionals to their limit.
We prove that, if $f$ is strongly convex, the rescaled rate functionals
	\[
		\E _\epsilon(u) \coloneqq \frac{\F_0(u) - \F_\eps(u)}{\epsilon^2}
	\]
$\Gamma$-converge w.r.t. to the $H^1_\loc(\Rd)$-convergence to the second order functional
	\[
		\E_0(u) \coloneqq \frac{1}{24}\int_{\Rd}\int_{\Rd}
				\rho(z)\va{z}^2 f''(\va{\nabla u(x)\cdot \hat{z}})\va{\nabla^2 u(x) \hat{z} \cdot \hat{z}}^2 \de z \de x.
	\]
The assumption of strong convexity on $f$ rules out several interesting scenarios,
notably the one of linear growth at infinity.
However, the analysis of the case at stake is already nontrivial,
and it leads to a conclusion which is analogous
to the one of M. Peletier, R. Planqu\'{e}, and M. R\"{o}ger in \cite{PPR}
for convolution-type functionals.

Similarly to the works of M. Gobbino and M. Mora \cites{G,GM},
we perform a proof based on a slicing procedure,
which reduces the problem to the study of some energies defined on functions of one variable,
see Section \ref{sec:rateBBM}.
Then, in Section \ref{sec:Gconv-BBM}, 
we establish the desired result relying on the $\Gamma$-convergence of the $1$-dimensional functionals.
This is not immediate, though, and is achieved
by resorting to a compactness result,
whose proof, in turn, requires some effort because of the structure of $\E_\epsilon$.

From Chapter \ref{ch:nlpnlc} on,
we leave Sobolev spaces to move closer to geometric measure theory,
and we introduce generalised notions of
perimeter, total variation, and mean curvature.

Let $\Omega \subset \Rd$ be measurable
and consider an even function $K\colon \Rd \to [0,+\infty)$.
Heuristically, $K$ plays the role of a weight
that encodes long range interactions between the points in the space.
For a measurable function $u\colon \Omega \to \R$
we define its \emph{nonlocal total variation} in $\Omega$ as
	\begin{equation}\label{eq:nltv}
	\begin{split}
	J_K(u;\Omega) & = \frac{1}{2}\int_\Omega\int_\Omega K(y-x)\va{u(y)-u(x)}\de y\de x \\
						&\quad + \int_{\Omega}\int_{\co{\Omega}}K(y-x)\va{u(y)-u(x)}\de y \de x.
	\end{split}
	\end{equation}
We recall that, for a set $E\subset \Rd$,
the perimeter in the sense of E. De Giorgi \cites{D,AFP,Ma} is defined as
the total variation of the distributional gradient of its characteristic function.
Accordingly, we define the \emph{nonlocal perimeter} in $\Omega$ of $E$ as
	\[
		\PerK(E;\Omega) \coloneqq J_K(\chi_E;\Omega),
	\]
$\chi_E$ being the characteristic function.

The first instance of nonlocal perimeter appeared in the celebrated paper \cite{CRS}
by L. Caffarelli, J. Roquejoffre, and O. Savin,
who focused on the case $K(x) = \va{x}^{-d-s}$ for some $s\in(0,1)$,
i.e. on \emph{fractional kernels}.
Over the years, their work has influenced vastly the research,
and in the sequel we suggest a brief list of developments.

Before going into the details of Chapter \ref{ch:nlpnlc},
we present some examples by E. Cinti, J. Serra, and E. Valdinoci \cite{CSV}
with the purpose of motivating the interest in the functionals under consideration,

The first example concerns digital imaging.
In order to be displayed on the screen of an electronic device,
images are sampled by means of $2$-dimensional grids
whose elementary cells correspond to the smallest units that the device can handle.
These cells are called {\it pixels}, a short for picture element,
and their whole provides the screen representation of the given image.
As a general principle, the more the pixels,
the more accurate is the digital reproduction;
in other words, pixels whose side lengths are small result in high resolution and fidelity.
The following discussion shows
how nonlocal perimeters are sensitive to changes of resolution,
differently from De Giorgi's one.

Suppose that we want to create a pixel representation
of a unit square whose sides are rotated by $45^\circ$ w.r.t. to the reference grid
(see Figure \ref{fig:csv}) .
\begin{figure}
	\begin{center}
		\includegraphics[scale=0.6]{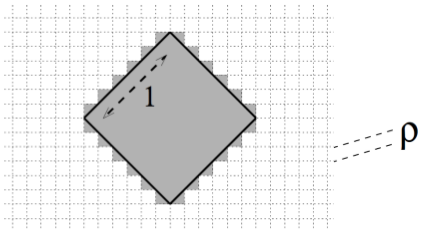}
		\caption{
			Representation of a unit square by pixels of side length equal to $\rho$.
			Source: \cite{CSV}.
			}\label{fig:csv}
	\end{center}
\end{figure}
An elementary calculation gets that the standard perimeter 
of the reproduction, which is displayed in grey in the figure,
equals $4\sqrt{2}$, no matter how we choose the side length of the pixels.
Hence, we see that
the discrepancy between the perimeter of the source image and
the one of its pixel approximation is left unchanged 
under increase of resolution.

In this respect, nonlocal perimeters are far more accurate.
For instance, let $s\in(0,1)$ and consider the fractional perimeter
	\[
		\Per_s(E)\coloneqq \Per_s(E;\Rd) = \int_{\co{E}}\int_E \frac{\de x \de y}{\va{y-x}^{2+s}}.
	\]
Denoting by $Q$ and $\tilde Q$ respectively the given square and its pixel representation,
we have that
	\[\begin{split}
	\Per_s(\tilde Q) & =  \int_{\co{\tilde Q}}\int_{\tilde Q\cap \co{Q}} \frac{\de x \de y}{\va{y-x}^{2+s}}
								+ \int_{\co{\tilde Q}}\int_{Q} \frac{\de x \de y}{\va{y-x}^{2+s}} \\
							& \quad = \int_{\co{\tilde Q}}\int_{\tilde Q\cap \co{Q}} \frac{\de x \de y}{\va{y-x}^{2+s}}
											+ \Per_s(Q) - \int_{\tilde Q\cap \co{Q}}\int_{Q} \frac{\de x \de y}{\va{y-x}^{2+s}} \\
							&\quad \leq \int_{\co{\tilde Q}}\int_{\tilde Q\cap \co{Q}} \frac{\de x \de y}{\va{y-x}^{2+s}}
											+ \Per_s(Q).
	\end{split}\]
The set $\tilde Q\cap \co{Q}$ is the ``frame'' of pixels that surrounds $Q$,
whence
	\[
		\Per_s(\tilde Q)\leq \Per_s(Q) + \sum_{i} \Per_s(P_i),
	\]
where $P_i$ is a boundary pixel.
The number of such pixels is of the order $\rho^{-1}$
and, by a change of variables, it is easy to see that
$\Per_s(P_i) = \rho^{2-s} \Per_s(Q)$.
On the whole, we find
	\[
		\Per_s(\tilde Q) \lesssim (1+\rho^{1-s})\Per_s(Q),
	\]
so that the discrepancy is bounded above by a quantity that vanishes as $\rho$ tends to $0$,
that is, it vanishes in the high resolution limit.

It is worth mentioning that
nonlocal functionals akin to \eqref{eq:nltv} have also been studied
in connection with image denoising.
In a nutshell, this amounts to recover a image of higher quality
starting from a distorted datum $u_0\colon \Omega \to \R$.
The result of the process should be a function
that is more regular than $u_0$ and
that is, at the same time, not too different from it.
To achieve the goal, the idea is to look for solutions to the problem
	\[
		\min_u \Set{J(u)+ \alpha \int_{\Omega} \va{u - u_0}^2},
	\]
where $J$ acts as a {\it filter} and $\alpha$ is the so-called {\it fidelity parameter},
to be suitably tuned. 
The differences between the models proposed in the literature lie in the choice of the filter $J$.
The classical one is $J(u)=\int_{\Omega}\va{\D u}$ \cite{ROF},
but, more recently, several nonlocal filters have been  introduced
\cites{GO2,AK,LS,BN2,BN1}.

The second motivational example comes from physics.
In \cite{CH}, J. Cahn and J. Hilliard introduced a model for phase transitions
in certain binary metallic alloys,
and they derived the following expression for the free energy:
	\[
		\F_\epsilon^{\mathrm{CH}} (u;\Omega) \coloneqq
			\int_{\Omega} \left[ \frac{\epsilon^2}{2} \va{\nabla u(x)}^2 + W(u(x))\right]\de x.
	\]
Here, $\Omega$ is the container where the system sits,
$u\colon \Omega \to [-1,1]$ represents the concentration of a reference phase,
$\epsilon>0$ is a (small) scale parameter, and
$W\colon \R \to [0,+\infty)$ is a double-well potential which vanishes only in $-1$ and $1$.

When the phases are mixed, that is, $u$ takes also values in $(-1,1)$, and
the system is in suitable physical conditions,
experiments show the evolution is steered
by the minimisation of $J_\epsilon^{\mathrm{CH}}$:
thus, owing to the pontential $W$, the two phases separate,
and the transition occurs in a thin interface,
because the gradient term penalises wide interfaces.

In the description of some systems,
it is reasonable to assume
that particles interact also if they are far apart;
for instance, this is observed 
in continuum limits of discrete Ising's spins.
Such circumstances suggest to modify the free energy of Cahn and Hilliard
as to include nonlocal functionals.
The option considered by G. Alberti and G. Bellettini in \cites{AB1,AB2} is
	\[
		\F_\epsilon^{\mathrm{AB}}(u;\Omega) \coloneqq
			\frac{1}{4} \int_{\Omega} \int_{\Omega} K_\epsilon(y-x)\va{u(y)-u(x)}^2 \de y \de x
					+ \int_{\Omega}W(u(x)) \de x,
	\]
where
	\begin{equation}\label{eq:intro-rescK}
	K_\epsilon(x) \coloneqq \frac{1}{\eps^d} K \left(\frac{x}{\eps}\right)
	\end{equation}
and $K$ is a positive, long range interaction kernel;
see also the paper \cite{BBP} by G. Bellettini, P. Butt\`a, and E. Presutti.
Analogous energies of fractional type have been investigated as well;
the reader interested in this topic may consult the survey \cite{V} by E. Valdinoci and references therein.

Observe that, when the phases are separated,
the functional $\F_\epsilon^{\mathrm{AB}}$
equals the first summand in \eqref{eq:nltv}:
indeed, if $u$ takes only the values $-1$ and $1$,
then $\va{u(y)-u(x)}^2 = 2\va{u(y)-u(x)}$.
Therefore, we see that $J_K$ is related to the free energy of the system at rest.

Let us now overview our analysis of nonlocal perimeters and total variations.
Relying on \cite{BP}, we explore some basic properties of the nonlocal energies 
in Sections \ref{sec:elem-nlp} and \ref{sec:ca-pl}.
We allow for a wide family of kernels,
which encompasses, for instance, the fractional ones,
and, as a first result, we show that the nonlocal total variation of a function $u$
is finite if $u$ is of bounded variation.

Secondly, we show that there exist local minimisers of $J_K$
under prescribed boundary conditions.
This corresponds to the existence of solutions to the well-studied Plateau's problem,
which amounts to find a hypersurface of least area
among the ones spanned by a given boundary.

A simple application of Fatou's Lemma yields $L^1_\loc(\Rd)$-lower semicontinuity of $J_K$,
but, in striking contrast with the local case,
sequences that have equibounded nonlocal variations are in general not precompact in $L^1$
(see the lines preceding Remark \ref{stm:coercivity} for a counterexample).
This hinders a straightforward transliteration
of the geometric measure theory approach to the classical Plateau's problem.
Nonetheless, it is possible to deduce existence of minimisers
{\it via} the direct method of calculus of variations.
The key ingredients in this regard are the convexity of $J_K$
and the validity of a \emph{generalised Coarea Formula}:
	\[
	J_K(u;\Omega) = \int_0^1 \PerK(\Set{u>t};\Omega) \de t.
	\]
The importance of the latter was firstly pointed out by A. Visintin \cites{Vi,Vi2},
who was concerned with energies for multiphase systems
as the one we mentioned above.

The existence of solutions to the nonlocal Plateau's problem raises
 at least a couple of questions:
\begin{itemize}
	\item Can we give a \emph{necessary} condition for optimality 
		in terms of the first variation of the nonlocal perimeter?
	\item Can we devise some \emph{sufficient} condition for optimality?
\end{itemize}

As in the local case,
tackling the first issue leads to the notion of curvature.
The geometric first variation of $\PerK$ was computed
as soon as nonlocal perimeters were introduced.
Indeed, in \cite{CRS},
Caffarelli, Roquejoffre, and Savin identified
	\begin{equation}\label{eq:intro-HK}
	H_K(E,x) \coloneqq - \mathrm{p.v.} \int_{\Rd} K(y-x) \big(\chi_{E}(y) - \chi_{\co{E}}(y)\big) \de y
	\end{equation}
as the first variation of $\PerK$, $K$ being a fractional kernel.
The relationship holds true in wider generality, and,
due to the possible singularity of $K$,
it is not evident in which sense a solution $E$ to Plateau's problem satisfies
the Euler-Lagrange equation $H_K(E,x) = 0$ for $x\in \partial E$.
It is not even clear, at first sight,
for which sets the integral defining the nonlocal curvature is convergent.

In Section \ref{sec:nlc}, we give sufficient conditions on $K$
so that $H_K(E,x)$ is finite for all sets $E$
with $C^{1,1}$ boundary and for all $x\in \partial E$. 
This is not a novelty \cites{CMP,MRT},
but we also establish an upper bound on $H_K$
that is to come in handy in Chapter \ref{ch:nlc}.

As for the Euler-Lagrange equation $H_K(E,x) = 0$,
J. Maz\'on, J. Rossi, and J. Toledo proved that
when $K\in L^1(\Rd)$ and $x$ belongs to the reduced boundary of $E$,
the equality holds pointwise.
More broadly, it has to be interpreted in a viscosity sense, see \cite{CRS}.
In a very recent paper \cite{C}, X. Cabr\'e has re-proved this fact
by utilising an argument that is completely different from the original one (see later on).

On the side of sufficient conditions for optimality,
fewer literature concerning our problem is available.
{\it Vice versa}, for De Giorgi's perimeter
the well developed theory of \emph{calibrations} is at our disposal \cites{Mo1,HL,Mo2,LM,CCP}.
In Section \ref{sec:calib}, which grounds on \cite{Pa},
we propose a nonlocal notion of calibration
that allows for a minimality criterion that parallels the classical one.
A similar task was indepentently accomplished by Cabr\'e \cite{C},
who also exploited calibrations to retrieve that
nonlocal minimal surface satisfy the equation $H_K = 0$ is the viscosity sense.
Another possible application will be discussed in Chapter \ref{ch:judith}.

We mould nonlocal calibrations on the classical ones.
Yet, while the latter are vector fields,
following the approach to nonlocal Cheeger's energies in \cite{MRT},
our calibrations are scalar functions on the product space $\Rd \times \Rd$.
Then, we prescribe that they satisfy
a bound on the norm, a vanishing divergence condition,
and agreement with the normal of the calibrated set.
Of course, these requirements need to be formulated appropriately
as to fit in the nonlocal framework.
In particular, inspired again by \cite{MRT},
we resort to the notion of nonlocal divergence introduced by G. Gilboa and S. Osher in \cite{GO}.

While, in general, it is not effortless to construct a calibration,
it is simple to exhibit one for graphs \cite{C}.
In the very special case of halfspaces,
alongside with optimality, we are also able to prove uniqueness \cite{Pa}.
This property had been established previously
only for fractional perimeters \cites{CRS,ADM}.

A common feature of the phase field models 
that we sketched above is that, under suitable rescalings,
the energies approach asymptotically a surface functional.
In the local setting, 
a classical result by L. Modica and S. Mortola \cites{MM} states that,
when $\epsilon\to 0^+$,
the family $\Set{\epsilon^{-1}J_\epsilon^{\mathrm{CH}}}$ $\Gamma$-converges
to De Giorgi's perimeter, up to a constant encoding surface tension.
In \cite{AB2} a local, anisotropic surface functional was retrieved
as the $\Gamma$-limit of $\Set{\epsilon^{-1}J_\epsilon^{\mathrm{AB}}}$,
and O. Savin and E. Valdinoci drew similar conclusions
for fractional free energies in \cite{SV}.
In a more abstract context,
also the limiting behaviour of the functional $J_K$ in \eqref{eq:nltv} was studied:
for $K(x) = \va{x}^{-d-s}$ with $s\in(0,1)$,
L. Ambrosio, G. De Philippis, and L. Martinazzi proved in \cite{ADM} that,
when $s \to 1^-$, De Giorgi's perimeter is obtained in the sense of $\Gamma$-convergence.

In Chapter \ref{ch:judith}, which is based on \cites{BP,Pa},
we discuss a nonlocal-to-local convergence result which is reminiscent of the previous ones.
Indeed, provided that
	\[
		\int_{\Rd} K(x)\va{x} \de x <+\infty,
	\]
we recover a local, anisotropic total variation $J_0$
as the $\Gamma$-limit of the family $\Set{\epsilon^{-1}J_\epsilon}$,
where $J_\epsilon(u;\Omega) \coloneqq J_{K_\eps}(u;\Omega)$
and $K_\epsilon$ is as in \eqref{eq:intro-rescK}.
Under the current assumption on the kernel,
it is non difficult to show that, when $u$ is a function of bounded variation,
$J_\epsilon(u;\Omega)$ is of order $\epsilon$.
Consequently, to get a nontrivial limit,
we put the factor $\epsilon^{-1}$ in front of the rescaled energies.

Our convergence result is closely related to the one in \cite{P},
and we rely on that work for what concerns the pointwise limit.
Nevertheless, the interactions between the reference set $\Omega$ and its complement
are not taken into consideration there,
so our statement does not follow immediately from the ones of Ponce.
In any case, aside from pointwise convergence, our proof is self-contained.

We remark that, paralleling the analysis of Chapter \ref{ch:nle},
it would be interesting to investigate whether the rate functionals
$\eps^{-\alpha}(J_0 - J_\eps)$ $\Gamma$-converge for some $\alpha > 0$
to a geometric, second order limit functional.

The first step
to prove the $\Gamma$-convergence of the rescaled functionals $J_\epsilon$
is reducing to nonlocal perimeters.
This is made possible by a result of A. Chambolle, A. Giacomini, and L. Lussardi \cite{CGL},
which, in turn, grounds on the validity of Coarea Formulas in the sense of Visintin.
Then, in Section \ref{sec:PerK-Glimsup},
we can easily establish the upper limit inequality,
exploiting the density of polyhedra in the family of finite perimeter sets.

The proof of the $\Gamma$-lower limit inequality,
which is addressed in Section \ref{sec:PerK-Gliminf}, is more elaborated.
As \cites{AB2,ADM}, we adopt the blow-up technique introduced by I. Fonseca and S. M\"uller \cite{FM},
and we deduce a lower bound
where a functional $\tilde J_0$, {\it a priori} distinct from $J_0$, appears.
By construction, the anisotropy of $\tilde J_0$ is given only implicitly:
indeed, it is defined as the value of a certain cell problem
which involves sequences of sets converging in $L^1$ to a given halfspace.
To achieve the conclusion, in Section \ref{sec:sigmaK}
we identify the anisotropy occurring in $\tilde J_0$ as the one of $J_0$.
Much in the spirit of \cite{ADM},
the minimality of halfspaces proves to be crucial at this stage,
and this provides an interesting application of the calibration criterion in Chapter \ref{ch:nlpnlc}.

We recalled that nonlocal curvatures can be derived as first variations of nonlocal perimeters.
With the result of Chapter \ref{ch:judith} on the table,
it is spontaneous to wonder if an analogue asymptotic results holds for curvatures.
In the affirmative case, it is also natural to ask
whether the convergence is inherited by the associated geometric motions,
that is, by the $L^2$-gradient flows of the perimeters.
We deal with these questions in the concluding chapter,
which includes the contributions in \cite{CP},
and we provide a positive answer to both.

When $K$ is a fractional kernel,
N. Abatangelo and E. Valdinoci \cite{AV} proved
that suitable rescalings of the nonlocal curvatures pointwise converge
to a multiple of the standard mean curvature when the fractional exponent approaches $1$.
For radial, $L^1$ kernels, the same conclusion was achieved in \cite{MRT}.
We establish a congruent theorem in Section \ref{sec:convcurv}
for a larger class of kernels.
We show that, when $E$ is a set  with $C^2$ boundary,
in the limit $\epsilon\to 0^+$,
the functionals $\epsilon^{-1}H_\epsilon(E,x) \coloneqq \epsilon^{-1}H_{K_\eps}(E,x)$,
with $K_\epsilon$ as above, tend
for all $x\in \partial E$ to an anisotropic mean curvature functional $H_0(\partial E,x)$.
If the boundary of $E$ is also compact,
we are able to conclude that the convergence is uniform.

The next thing we do is checking that
the limit curvature agrees with the limit perimeter found in Chapter \ref{ch:judith},
namely we notice that $H_0$ is the first variation of the restriction of $J_0$ to sets.
So, in a sense, the first variation commutes with the limiting procedures.

The last result that we present in this thesis concerns the motions by curvature
	\begin{gather*}
		\partial_t x(t)\cdot \hat{n}(t) = -\frac{1}{\eps}H_\eps(E(t), x(t))
		\quad\text{and}\quad
		\partial_t x(t)\cdot \hat{n}(t) = -H_0(\Sigma(t), x(t)),
	\end{gather*} 
where $t\mapsto E(t)\subset \Rd$ is an evolution of sets,
$\Sigma(t)\coloneqq\partial E(t)$, $x(t)\in\Sigma(t)$,
and $\hat{n}(t)$ is the outer unit normal to $\Sigma(t)$ at $x(t)$.
We study them utilizing the level-set method,
which is to say that we suppose that
the evolving set $E(t)$ and its boundary $\Sigma(t)$
are respectively the $0$ superlevel set and the $0$ level set
of some function $\phi(t,\,\cdot\,)$.
Rewriting the previous equations in terms of $\phi$ gets formally
the following nonlocal and local parabolic partial differential equations:
	\begin{gather}
		\partial_t  \phi(t,x) + 
		\frac{\va{\nabla \phi(t,x)}}{\eps} H_\epsilon(\{y:\phi(t,y)\geq \phi (t,x)\},x) = 0, \label{eq:intro-epspb}
		\\
		\partial_t  \phi(t,x) + 
		\va{\nabla \phi(t,x)}H_0(\{y:\phi(t,y)= \phi (t,x)\},x) = 0. \label{eq:intro-0pb}
	\end{gather}
Already in the local case,
it is well know that a weak notion of solution is needed
to cope with the onset of singularities observed in this sort of evolutions.
We employ the viscosity theory by L. Evans and J. Spruck \cite{ES}
and by Y. Chen, Y. Giga, and S. Goto \cite{CGG},
who, in turn, relied on results for Hamilton-Jacobi equations (see for instance \cite{CL}).
In this context, existence and uniqueness of a global-in-time solution are ensured.
On the nonlocal side,
well posedness for the Cauchy's problem for the motion
by fractional curvature was settled by C. Imbert in \cite{I}
(see also the paper \cite{LV} by D. La Manna and J. Vesa
for short time existence of the classical solutions).
Lately, existence and uniqueness for the level-set formulation
of a broad range of local and nonlocal curvature flows
have been established in a unified framework
by A. Chambolle, M. Morini, and M. Ponsiglione \cite{CMP}.

Several aspects
of motions by nonlocal curvatures
have recently been subjects of investigation,
mostly in the fractional case: for instance,
conservation of convexity \cite{CNR}, formation of neckpinch singularities \cite{CSV2},
fattening phenomena \cite{CDNV}, self-shrinkers \cite{CN2}.
Instead, going back in time, we see that
nonlocal geometric evolutions emerged
from physical models that strive to explain plastic behaviour of metallic crystals.
According to these theories,
the macroscopic phenomenon of plasticity arises as a consequence
of microscopic irregularities in the crystalline lattice, known as defects.
Among them, we find dislocations,
which correspond to linear misalignments in the microscopic structure.
Each dislocation produces an elastic field,
which in turn exerts the so called Peach-Koehler force on all other defects in the alloy.
As a consequence, dislocation lines evolve,
and each point on them moves
with normal velocity determined by the Peach-Koehler force.
In \cite{AHLM},
O. Alvarez, P. Hoch, Y. Le Bouar, and R. Monneau proposed
a mathematical description of dislocation dynamics
in terms of a nonlocal eikonal equation, for which they
were able to prove short time existence and uniqueness.
In this model, the Peach-Koelher force is encoded by a convolution kernel $c_0$, typically singular,
whence the nonlocal nature of the problem.
Imbert highlights in \cite{I} that 
the mathematical expression of the Peach-Koehler force may be understood
as a nonlocal curvature of the dislocation line,
thus, dislocation dynamics turns out to be a motion by nonlocal curvature.

The explicit expression of the kernel $c_0$ might be complicated,
because it has to capture the physical features of the system.
Further, $c_0$ might change sign
and this poses serious technical difficulties,
because it opposes the validity of comparison principles.
The original model has by then been simplified by various authors,
who obtained well-posedness for the problem
\cites{ACM,BL,IMR,FIM,DFM}.

After the excursus on the connection between nonlocal curvatures and dislocation dynamics,
we discuss now the main result of Chapter \ref{ch:nlc},
which shows that the convergence of the curvatures is sufficiently robust
to entail convergence of the solutions to the corresponding geometric flows.
Precisely, we fix $u_0\colon \Rd \to \R$
and we let $u_\epsilon$ and $u$ be the viscosity solutions with initial datum $u_0$
to \eqref{eq:intro-epspb} and \eqref{eq:intro-0pb}, respectively.
We prove that the family $\Set{u_\epsilon}$ locally uniformly converges to $u$ as $\eps\to 0^+$.

A similar result already appeared in the paper \cite{DFM}
by F. Da Lio, N. Forcadel, and R. Monneau,
who also proved convergence of the nonlocal curvature functionals.
However, they made assumptions on the kernel that are different from ours,
because they focused on a precise dislocation dynamics model.
Accordingly, their scaling does not match ours.
We also stress that our proof and theirs are based
on substantially distinct arguments.  
Indeed, in \cite{DFM}, the authors exploit viscosity semilimits and the perturbed test function method,
whilst we propose an approach by \emph{geometric barriers}.
These were introduced by De Giorgi in \cite{D2}
as weak solutions to a wide class of evolution problems.
G. Bellettini, M. Novaga, and M. Paolini considered solution in this class
for geometric parabolic PDEs such as \eqref{eq:intro-0pb},
and they found out that viscosity theory and barriers can be compared \cites{BP2,Be,BNo,BNo2}.
In the same spirit, in the recent paper \cite{CDNV},
A. Cesaroni, S. Dipierro, M. Novaga, and E. Valdinoci put in relation
barriers and viscosity solutions for nonlocal curvature motions.
These comparison results are the cornerstone of the analysis in Section \ref{sec:convflows}. 
  
As a concluding bibliographical note,
we point out that approximation results for mean curvature motions,
either by local or nonlocal operators, have been intensively studied.
One of the most renowned is due to J. Bence, B. Merriman, and S. Osher.
In \cite{BMO}, they devised a formal strategy to construct fronts evolving by mean curvature,
the so called threshold dynamics algorithm.
It amounts to a time-discretization, which, at each
step, takes the characteristic functions of a suitable set as datum,
and yields its evolution according to the heat equation.
The convergence of this scheme was rigorously settled in \cites{BG,E}
and subsequently more general diffusion operators were considered
\cites{Is,IPS,ChN}. 
In \cite{CaS} L. Caffarelli and P. Souganidis applied a threshold dynamics scheme
to motions by fractional mean curvature.
A generalisation encompassing anisotropies and driving forces was proved in \cite{CNR}. 
		\cleardoublepage
		\section*{Notation}
		\phantomsection
		\addcontentsline{toc}{section}{Notation}
			\begin{center}
		\begin{tabular}{cl}
			$a\wedge b$, $a \vee b$	& resp., minimum and maximum between $a,b\in\R$ \\
			$d$						& dimension of the ambient space $\Rd$ \\
			$\id$					& identity map in $\Rd$ \\
			$\co{E}$			& complement of the set $E$ in $\Rd$, i.e. $\Rd\setminus E$ \\
			$E\symdif F$		& symmetric difference between the sets $E$ and $F$ \\
			$E=F$				& the sets $E$ and $F$ coincide up to a Lebesgue negligible set \\
			$E \subset F$	& the sets $E$ is contained in the set $F$, up to a Lebesgue negligible set \\
			$\chi_E$			& characteristic function of the set $E$ \\
			$\tilde \chi_E$		& signed characteristic function of the set $E$,
											i.e. $\tilde \chi_E=\chi_E - \chi_{\co{E}}$ \\
			$x\cdot y$			& Euclidean inner product between $x,y\in\Rd$ \\
			$\va{x}$			& Euclidean norm of $x\in\Rd$ \\
			$\hat{x}$			& unit vector parallel to $x\in\Rdmz$, i.e. $x/\va{x}$ \\
			$x^\perp$			& subspace of the vectors that are orthogonal to $x\in\Rdmz$ \\
			$\pi_{\hat{x}^\perp}$ & the matrix that represents the orthogonal projection on $\hat{x}^\perp$,
										i.e $\id - \hat{x}\otimes \hat x$ \\
			$e_i$					& $i$-th element of the canonical orthonormal basis of $\Rd$ \\
			$B(x,r)$			& open ball of centre $x$ and radius $r>0$ in $\Rd$ \\
			$\Sdmu$				& topological boundary of $B(0,1)$ \\
			$\dist(x,E)$		& Euclidean distance between the point $x$ and the set $E$ \\
			$\dist(E,F)$		& Euclidean distance between the sets $E$ and $F$ \\
			$\va{\mu}$			& total variation measure of the measure $\mu$ \\
			$\norm{\mu}$ 	& total variation norm of the measure $\mu$ \\
			$\mu\llcorner\Omega$ & restriction of the measure $\mu$ to the set $\Omega$ \\
			$\Ld$					& $d$-dimensional Lebesgue measure \\
			$\Hdmu$				& $(d-1)$-dimensional Hausdorff measure \\
			$\omega_{d-1}$ & $(d-1)$-dimensional Lebesgue measure of the unit ball in $\R^{d-1}$ \\
			a.e.					& $\Ld$-almost everywhere in $\Rd$ or
											$\Ld\otimes \Ld$-almost everywhere in $\Rd\times\Rd$ \\
			$u=v$				& the functions $u,v\colon \Rd \to \R$ coincide $\Ld$-a.e. \\
			$\D u$					& distributional gradient of the function $u$ \\
			$\BV(\Omega)$,
			$\BV_{\loc}(\Omega)$ & functions of (locally) bounded variation in $\Omega$ \\
			$M^\trasp$				& transpose of the matrix $M$ \\
			$\tr{M}$			& trace of the matrix $M$
			
		\end{tabular}
	\end{center}

\mainmatter
\pagestyle{fancy}

\chapter{Basic ingredients}\label{ch:pre}
	In this preliminary chapter we collect some standard notions and results
that make up the background of the sequel.
We include here just the topics that are of general relevance for the thesis,
while we shall provide reminders that pertain to specific aspects
when they are needed.
Notably, the theories of viscosity solutions and of barriers
for geometric evolution equations are presented in Chapter \ref{ch:nlc}.
Moreover, we omit proofs, which may be found in the suggested references.

We start with the essentials of geometric measure theory.
We recall some facts about measures on the Euclidean space,
then we introduce the space of functions of bounded variation
and the class of finite perimeter sets.

One of the leitmotifs of our analysis is
recovering local functionals as limits of their rescaled nonlocal counterparts.
In Chapters \ref{ch:nle} and \ref{ch:judith},
the asymptotics relationships are expressed in terms of $\Gamma$-convergence,
whose definition we recall in Section \ref{sec:def-Gconv}.
Typically, the limit functionals that we obtain are not isotropic.
For this reason, we sum up in Section \ref{sec:aniso} some properties of anisotropic surface energies.

Before getting to the heart of the matter,
let us premise some notations concerning the ambient space of our analysis.

For us, $d$ is always a natural number greater than $1$.
We work in the $d$-dimensional Euclidean space,
that is, the vector space $\Rd$ endowed with the Euclidean inner product $\cdot$
and the associated norm $\va{\,\cdot\,}$. 
We let $e_1,\dots,e_d$ be the elements of the canonical orthonormal basis,
and, when $x\in \Rdmz$, we set
	\[
		\hat{x} \coloneqq \frac{x}{\va{x}}
			\quad\text{and}\quad
		x^\perp \coloneqq\Set{ y\in\Rd : x\cdot y = 0}.
	\]
We denote the open ball of centre $x$ and radius $r$ in $\Rd$ by $B(x,r)$
and we let $\Sdmu$ be the topological boundary of $B(0,1)$.
If $E,F\subset \Rd$ and $x\in\Rd$, we put
	\begin{gather*}
		\dist(x,E)\coloneqq \inf \Set{ \va{y-x} : y\in E}, \\
		\dist(E,F)\coloneqq \inf \Set{ \va{y-x} : x\in E, y\in F}.
	\end{gather*}

When $E\subset \Rd$ is a set,
we write $\co{E}$ for its complement in $\Rd$ and
$\chi_E$ for its \emph{characteristic function}, i.e. 
	\[
		\chi_E(x) \coloneqq
			\begin{cases}
				1 & \text{if } x\in E,\\
				0 & \text{if } x\in \co{E}.
			\end{cases}
	\]
If $F$ is another subset of $\Rd$, 
we let $E \symdif F$ be the symmetric difference between $E$ and $F$,
that is $E \symdif F \coloneqq (E \cap \co{F}) \cup (\co{E} \cap F)$.

\section{Measures and $\BV$ functions}
\label{sec:measures}
Following the monographs
\cite{AFP} by L. Ambrosio, N. Fusco, and D. Pallara and
\cite{Ma} by F. Maggi,
we introduce the by now classical tools
to tackle geometric variational problems.

Let $X$ be a set and let $\call{A}$ be a $\sigma$-algebra on $X$.
A map $\mu\colon \call{A} \to [-\infty,+\infty]$ is a (real) \emph{measure} on $(X,\call{A})$
if it is $\sigma$-additive and $\mu(\emptyset) = 0$.
We say that a measure $\mu$ is \emph{positive}
when $\mu(E)\in[0,+\infty]$ for all $E\in\call{A}$,
and that it is \emph{finite}
if it is positive and $\mu(E)< +\infty$ for all $E\in\call{A}$.
Similarly, we may consider \emph{vector measures},
that is, $\sigma$-additive maps $\mu\colon X \to \R^n$ such that $\mu(\emptyset) = 0$
with $n\in\N$ larger than $1$.

When $\mu$ is either a real- or vector-valued measure on $(X,\call{A})$,
we define its \emph{total variation} as
	\[
		\va{\mu}\!(E) \coloneqq \sup\Set{ \sum_{\ell=1}^{+\infty}  \va{\mu(E_\ell)}
															: E_\ell \in \call{A},
															E_\ell \cap E_m = \emptyset \text{ if } \ell\neq m, \text{ and }
															E = \bigcup_{\ell=1}^{+\infty} E_\ell
														},
	\]
which is in turn a finite measure.
We dub \emph{norm} of the measure $\mu$ the quantity
$\norm{\mu}\coloneqq \va{\mu}(X)$.

In this work, we only consider measures on the Euclidean space $\Rd$.
Let $\call{B}$ the Borel $\sigma$-algebra,
i.e. the $\sigma$-algebra generated by open sets.
Measures that are defined on each set in $\call{B}$ are called Borel measures.
The \emph{support} of a positive measure is the set
	\[
	\supp\mu \coloneqq \Set{ \bigcap C : C \text { is closed and } \mu(\co{C}) = 0 }. 
	\]
	
We say that a positive Borel measure $\mu$ on $\Rd$ is a \emph{Radon measure}
if it is
	\begin{itemize}
		\item locally finite: $\mu(C)<+\infty$ for all compact sets $C\subset \Rd$;
		\item regular: for all $F\subset \Rd$ there exists $E\in\call{B}$
			such that $F\subset E$ and $\mu(F) = \mu(E)$.
	\end{itemize}
This concept may also be extended to the vectorial case:
a Borel measure $\mu$ on $\Rd$ with values in $\R^n$ is a vector-valued Radon measure
if there exist a positive measure $\nu$ on $(\Rd,\call{B})$ which is Radon
and a $L^1(\nu)$ function $f\colon \Rd \to \R^n$ such that
$\va{f} = 1$ $\nu$-a.e. and $\mu=f\nu$.
\textit{A posteriori}, one finds $\nu = \va{\mu}$.

Radon measures satistisfy a bunch of useful properties.
Firstly, when equipped with the total variation norm,
they form a normed space which coincides
with the dual of compactly supported continuous functions:
	\begin{thm}[Riesz]\label{stm:riesz}
		Let $\Cc(\Rd;\R^n)$ be the space
		of $\R^n$-valued, compactly supported, continuous functions on $\Rd$.
		Then, for all bounded linear functionals $\Lambda\colon \Cc(\Rd;\R^n) \to \R$,
		there exists a unique $\R^n$-valued Radon measure $\mu$ on $\Rd$ such that
			\[
				\Lambda(\phi) = \int_{\Rd} \phi \cdot \de\mu				
				\quad\text{for all } \phi\in\Cc(\Rd;\R^n).
			\]
		Moreover, for such measure and for all open $V\subset\Rd$,
		it holds
			\[
			\va{\mu}(U) = \sup\Set{ \Lambda(\phi) : \phi\in \Cc(V;\R^n), \va{\phi}\leq 1}
			\]
	\end{thm}
In view of the previous representation result,
we say that the sequence of Radon measures
$\Set{\mu_\ell}$ \emph{weakly-$\ast$} converges to the Radon measure $\mu$,
and we write $\mu_\ell \weakstar \mu$, if
	\[
		\limseq{\ell} \int_{\Rd} \phi\cdot\de \mu_\ell = \int_{\Rd} \phi\cdot\de \mu
		\quad\text{for all } \phi\in\Cc(\Rd;\R^n).
	\]

We stated above that if $\mu$ is a vector valued Radon measure,
then $\mu = f\va{\mu}$ for some $L^1(\va{\mu})$.
Relations of this kind hold in wider generality:
	\begin{thm}[Besicovitch's Differentiation]
		Let $\mu$ and $\nu$ be respectively a positive and an $\R^n$-valued Radon measure
		on some open $\Omega\subset \Rd$.
		If for all Borel sets $E\subset \Omega$ $\mu(E) = 0$ implies $\va{\nu}(E)  = 0$,
		then for $\mu$-a.e. $x\in\supp \mu$ the following limit exists:
			\[
				f(x) \coloneqq \lim_{r\to 0^+} \frac{\nu(B(x,r))}{\mu(B(x,r))}.
			\]
		Moreover, $f\in L^1(\Omega,\mu;\R^n)$ and $\nu = f\mu$.
		\end{thm}
We refer to the function $f$ above as the \emph{Radon-Nikodym derivative} of $\nu$ w.r.t. to $\mu$
and we denote it by ${\de\nu}/{\de\mu}$.
 
Another useful property of Radon measures concerns the maximum numbers
of leaves in a foliation which can have positive mass for a Radon measure:
\begin{prop}\label{stm:foliations}
	Let $\mu$ be a Radon measure on $\Rd$
	and let $\set{E_\alpha}_{\alpha\in A}$ be a family of disjoint Borel sets.
	Then, there exist at most countable $\alpha\in A$ such that $\mu(E_\alpha)>0$.
\end{prop}

Finally, we mention a simple criterion to prove that
the restriction of a Borel measure is a Radon measure.
We recall that the \emph{restriction} of the measure $\mu$ to the set $\Omega$
is the measure given by $\mu \llcorner \Omega(E) \coloneqq \mu(E\cap\Omega)$.

	\begin{prop}\label{stm:restriction}
		If $\mu$ is a Borel regular measure on $\Rd$
		and $E$ is a $\mu$-measurable set such that $\mu \llcorner E$ is locally finite,
		then $\mu\llcorner E$ is a Radon measure.
	\end{prop}

The measures that occur most frequently in the sequel
are the the $d$-dimensional Lebesgue and the $(d-1)$-dimensional Hausdorff measure on $\Rd$,
which we denote respectively by $\Ld$ and $\Hdmu$.
The Lebesgue measure is an example of Radon measure,
while $\Hdmu$ is Borel regular, but not locally finite;
however, by Proposition \ref{stm:restriction}, the restriction $\Hdmu\llcorner E$
is Radon whenever $\Hdmu(E)$ is finite.

We shall henceforth neglect to specify the measure
w.r.t. which a set or a function is measurable,
when the measure is $\Ld$ or the product $\Ld \times \Ld$ on $\Rd \times \Rd$;
similarly, we shall adopt the expression ``a.e.'' in place of
``$\Ld$-a.e.'' and of ``$\Ld \times \Ld$-a.e.''.
If $u$ and $v$ are measurable functions,
we shall use ``$u=v$'' as a shorthand for ``$u(x)=v(x)$ for a.e. $x\in\Rd$'';
in the same spirit, if $E$ and $F$ are measurable sets,
``$E=F$'' means that $E$ and $F$ coincide up to a negligible set
and ``$E\subset F$'' that $\Ld(E\cap\co{F}) = 0$.

We now introduce the class of functions whose distributional derivative is a Radon measure.
If $u$ is a function,
we  denote by $\D u$ its distributional gradient,
while we use $\nabla u$ for the pointwise one.

Let $\Omega \subset \Rd$ be an open set.
A function $u\in L^1(\Omega)$ is
\emph{of bounded variation} in $\Omega$
if $\D u$ is a $\Rd$-valued Radon measure on $\Omega$ such that $\norm{\D u} <+\infty$.
In this case, we write $u\in\BV(\Omega)$.
Since
	\begin{equation}\label{eq:TV}
		\norm{\D u} =
			\sup\Set{ \int_{\Omega} u \div\zeta : \zeta\in C^1_{\mathrm{c}}(\Omega;\Rd), \va{\zeta} \leq 1 },
	\end{equation}
$\BV(\Omega)$ may be equivalently defined as the subspace of all $u\in L^1(\Omega)$
such that the right-hand side in \eqref{eq:TV} is finite.
Observe that, in the light of \eqref{eq:TV},
the map $u\mapsto \norm{\D u}$ is lower semicontinuous w.r.t. the $L^1_\loc(\Omega)$ convergence.

The total variation of a $\BV$ functions bounds the $L^1$-norm of its difference quotients.
Actually, this property provides a characterisation
that we shall exploit several times in our analysis:

\begin{prop}[Characterisation \emph{via} difference quotients]\label{stm:char-BV}
	Let $\Omega\subset \Rd$ be open.
	Then, $u\colon\Omega \to \R$ is a function of bounded variation in $\Omega$
	if and only if 	there exists a constant $c\geq 0$ such that
	for any open set $V$ that is compactly contained in $\Omega$ and
	for any $z\in\Rd$ with $\va{z}<\dist(V,\co{\Omega})$ it holds
	\[
	\norm{\tau_z u -u}_{L^1(V)}\leq c\va{z},
	\]
	where $\tau_z u(x)\coloneqq u(x+z)$.
	In particular, it is possible to choose $c=\va{\D u}(\Omega)$.
\end{prop}

Finally, we remind a compactness criterion for functions of bounded variation.

	\begin{thm}\label{stm:BVcpt}
		Let $\Omega\subset \Rd$ be an open set and
		let $\Set{u_\ell}$ be a sequence in $\BV_\loc(\Omega)$.
		If for all open sets $V$ that are compactly contained in $\Omega$ it holds
		$\norm{u_\ell}_{L^1(V)} + \va{\D u_\ell}(V) \leq c$ for some $c\geq 0$,			
		then there exist a subsequence $\Set{u_{\ell_m}}$ and a  function $u\in\BV_\loc(\Omega)$
		such that $\Set{u_{\ell_m}}$ converges to $u$ in $L^1_\loc(\Omega)$.
	\end{thm}

\section{Finite perimeter sets and anisotropic surface energies}\label{sec:aniso}
The material in this section lays the groundwork for Chapters \ref{ch:nlpnlc} and \ref{ch:judith}.
There, we deal with a class of functionals that we name nonlocal perimeters;
here, instead, we introduce the concept of distributional perimeter,
which we may consider as the classical one.

The definition of perimeter dates back to the seminal works of R. Caccioppoli and E. De Giorgi.
In a nutshell, we identify a set with its characteristic function,
we consider the distributional gradient of the latter,
and we define its total variation as perimeter of the given set.
So, if $E\subset \Rd$ is measurable and $\Omega\subset\Rd$ is an open set,
we define the \emph{perimeter} of $E$ in $\Omega$ the quantity
	\begin{equation}\label{eq:dfnPer}
		\Per(E;\Omega)\coloneqq \sup\Set{ \int_{E} \div\zeta : \zeta\in \Cuc(\Omega;\Rd), \va{\zeta}\leq 1 }.
	\end{equation}
We say that $E$ is a \emph{set of finite perimeter} in $\Omega$
when $\Per(E;\Omega)<+\infty$.
Sometimes, we shall use the expression \emph{Caccioppoli set} as a synonym,
and we shall refer to $\Per$ as De Giorgi's perimeter,
when we want to distinguish it from the nonlocal counterparts
that occur in the next chapters.

In what follows, we shall frequently identify sets with their characteristic functions.
In particular, by saying that
a sequence $\Set{E_\ell}$ converges in $L^1_\loc(\Omega)$ to $E$
we mean that 
	\[
		\limseq{\ell} \Ld\big( (E_\ell \symdif E) \cap C \big) = 0
		\quad\text{for all compact sets $C$ such that $C\subset \Omega$,}
	\]
As a consequence of its definition,
the perimeter is $L^1_\loc(\Omega)$-lower semicontinuous.

Using the language of the previous section,
we may say that $E$ has finite perimeter in $\Omega$
if its characteristic function $\chi_E$ belongs to $\BV_\loc(\Omega)$.
Accordingly, if $E$ has finite perimeter in $\Omega$,
then $\D \chi_E$ is an $\Rd$-valued Radon measure in $\Omega$ and
	\[
	 \Per(E;\Omega) \coloneqq \va{\D \chi_E}(\Omega).
	\]
Moreover, if we denote by $\hat{n}$ the Radon-Nikodym derivative
of $\D \chi_E$ w.r.t. $\va{\D \chi_E}$, it holds
	\begin{equation}\label{eq:divthm}
		\int_{E} \div\zeta \de x = - \int_{\Omega} \zeta\cdot \hat{n}\de \va{\D \chi_E}
		\quad\text{for all } \zeta\in \Cuc(\Omega;\Rd).
	\end{equation}
The unit vector field
	\begin{equation}\label{eq:in-norm}
	\hat n(x)\coloneqq  \lim_{r\to 0^+}\frac{\D\chi_E(B(x,r))}{\va{\D\chi_E}(B(x,r))},
	\end{equation}
which is well defined $\va{\D \chi_E}$-a.e. $x\in\supp \va{\D \chi_E}$, may be understood as
the \emph{inner normal} of $E$.

Further, let
	\[
	\partial^\ast E \coloneqq \Set{x\in \Rd : \hat n(x) \text{ exists and } \va{\hat n(x)}=1}.
	\]
Fundamental results by De Giorgi and H. Federer \cites{D,F} show that
$\partial^\ast E$ is $(d-1)$-rectifiable for all measurable $E$
and that $\D \chi_E = \hat{n} \Hdmu \llcorner \partial^\ast E$,
whence
	\begin{equation*}\label{eq:per-Haus}
		\Per(E;\Omega)=\Hdmu(\partial^\ast E\cap\Omega),
	\end{equation*}
We call the set $\partial^\ast E$ \emph{reduced boundary} of $E$.
For any $x\in\partial^\ast E$, it holds
	\begin{equation}\label{eq:blowup}
		\frac{E-x}{r} \to \Set{y\in\Rd : y\cdot \hat n(x) > 0} \quad\mbox{as } r\to 0^+ \mbox{ in } L^1_\mathrm{loc}(\Rd) .
	\end{equation}

Sets of finite perimeter may be exploited to describe various physical models.
We are interested in anisotropic surface energies,
but we also introduce some more general concepts.

Let $\sigma\colon \R^n \to [0,+\infty]$ be a positevely $1$-homogeneous measurable function
and let $\Omega \subset \Rd$ be measurable.
For any $\R^n$-valued Radon measure on $\Rd$,
we define the \emph{anisotropic total variation}\index{total variation!anisotropic --} of $\mu$ as
	\begin{equation}\label{eq:aTV}
		J_\sigma(\mu;\Omega) \coloneqq \int_{\Omega} \sigma\left(\der{\mu}{\va{\mu}}\right) \de\va{\mu}.
	\end{equation}
The anisotropy is encoded by the function $\sigma$;
when it is constant, the usual total variation is recovered.

The following results sums up the continuity properties of $J_\sigma$:

	\begin{thm}[Reshetnyak]\label{stm:reshetnyak}
		Let $J_\sigma$ be the functional in \eqref{eq:aTV}.
		\begin{enumerate}
			\item If $\sigma\colon \R^n \to [0,+\infty]$ is $1$-homogeneous, lower semicontinuous, and convex,
				then for all open $\Omega\subset \Rd$ it holds
					\[
						J_\sigma(\mu;\Omega) \leq \liminf_{k\to+\infty} J_\sigma(\mu_k;\Omega)
					\]
				whenever $\mu$ and $\mu_k$, for all $k\in\N$, are $\R^n$-valued Radon measures on $\Rd$
				such that $\mu_k \weakstar \mu$.
			\item If $\sigma\colon \mathbb{S}^{n-1} \to [0,+\infty)$ is bounded and continuous,
				then for all open $\Omega\subset \Rd$ it holds
					\[
						J_\sigma(\mu;\Omega) = \limseq{k} J_\sigma(\mu_k;\Omega)
					\]
				whenever $\mu$ and $\mu_k$, for all $k\in\N$, are $\R^n$-valued Radon measures on $\Omega$
				such that $\mu_k \weakstar \mu$, $\norm{\mu_k} \to \norm{\mu}$,
				and $\norm{\mu}<+\infty$.
		\end{enumerate}
	\end{thm}

For $E \subset \Rd$, we may introduce an \emph{anisotropic perimeter}\index{perimeter!anisotropic --} functional
by setting
	\begin{equation}\label{eq:aPer}
		\Per_\sigma(E;\Omega) \coloneqq J_\sigma(\D \chi_E;\Omega)
												= \int_{\partial^\ast E \cap \Omega} \sigma(\hat{n}(x)) \de\Hdmu(x).
	\end{equation}
It is well-known that the total variation of a $\BV$ function
equals the integral of the perimeter of its level set.
We recall an anisotropic variant of such identity:
 
	\begin{prop}[Anisotropic Coarea Formula,\cites{CCCNP,Gr}]\label{stm:a-coarea}\index{Coarea Formula!anisotropic --}
		Suppose that $\sigma\colon \Rd\to [0,+\infty]$ is a norm
		and that $\Omega$ is an open, bounded set with Lipschitz boundary.
		Then, for all $u\in\BV(\Omega)$,
		\begin{equation*}
			J_\sigma (\D u;\Omega) = \int_{-\infty}^{+\infty} \Per_\sigma (\Set{u>t};\Omega) \de t
		\end{equation*}
	\end{prop}

Minimising a surface energy under prescribed boundary conditions is a classical problem.
In precise terms, given the reference set $\Omega$ and the boundary datum $E_0$,
we look for a set $E$ such that $\Per_\sigma(E;\Omega)$ attains
\begin{equation*}
\inf\Set{J_\sigma(\D u;\Omega) :
	u\colon \Rd \to [0,1] \text{ is measurable and }
	u = \chi_{E_0} \text{ in } \co{\Omega}.}
\end{equation*}
When $\sigma$ is a norm,
we can show that the infimum is achieved:
indeed, in this case, $J_\sigma(\D u;\Omega)\geq c \va{\D u}(\Omega)$ for some constant $c>0$,
and standard arguments yield the conclusion.

\section{$\Gamma$-convergence}\label{sec:def-Gconv}
We include in this section the definition
of a celebrated notion of convergence proposed by De Giorgi in the '70s \cite{DF}.
For a thorough treatment of the subject,
we refer to the monographs \cite{dM} by G. Dal Maso and \cite{B} by A. Braides.

\begin{dfn}\label{stm:defGammac}
	Let $(X,\de)$ be a metric space and,
	for all $\epsilon>0$, let $f_\epsilon\colon X\to[-\infty,+\infty]$ be a function.
	We say that the family $\Set{f_\eps}_{\epsilon>0}$
	\emph{$\Gamma$-converges}
	as $\epsilon\to0^+$
	w.r.t. the metric $\de$ to the function $f_0\colon X\to [-\infty,+\infty]$
	if for any $x\in X$ and for any sequence $\Set{\epsilon_\ell}_{\ell\in\N}$
	such that $\epsilon_\ell\to 0^+$
	the following hold:
	\begin{enumerate}
		\item for any sequence $\Set{x_{\epsilon_\ell}}_{\ell\in\N}\subset X$
			such that $x_{\epsilon_\ell}\to x$, we have
				\begin{equation}\label{eq:dfnGliminf}
					f_0(x)\leq\liminf_{\ell\to +\infty}f_{\epsilon_\ell}(x_{\epsilon_\ell});
				\end{equation}
		\item there exists a sequence $\Set{x_{\epsilon_\ell}}_{\ell\in\N}\subset X$
			such that $x_{\epsilon_\ell}\to x$ and
				\begin{equation}\label{eq:dfnGlimsup}
					\limsup_{\ell\to +\infty}f_{\epsilon_\ell}(x_{\epsilon_\ell}) \leq f_0(x).
				\end{equation}
	\end{enumerate}

	If $\Set{f_\eps}$ $\Gamma$-converges to $f_0$,
	the we say that $f_0$ is the \emph{$\Gamma$-limit} of $\Set{f_\eps}$.
\end{dfn}
Formula \eqref{eq:dfnGliminf} is referred to as
\emph{$\Gamma$-lower limit inequality},
while \eqref{eq:dfnGlimsup} is known as
\emph{$\Gamma$-upper limit inequality}.
A sequence such that the latter holds is called \emph{recovery sequence}.

We highlight in a separate statement a couple of properties
that we shall use in the sequel.
	
	\begin{lemma}\label{stm:prop-Gconv}
		Let $(X,\de)$ and $\Set{f_\epsilon}$ be as in the previous definition.
		Let also $f_0\colon X\to [-\infty,+\infty]$ be a function.
			\begin{enumerate}
				\item If $f_0$ is the $\Gamma$-limit of $\Set{f_\epsilon}$,
					then it is lower semicontinuous w.r.t. the metric $\de$.
				\item\label{stm:density-Gconv} Let $Y\subset X$ be a subset such that
					for any $x \in X$ there exists a sequence $\Set{y_\ell} \subset Y$
					such that $y_\ell \to x$ and $f_0(y_\ell) \to f_0(x)$.
					If for all $y\in Y$ there exists a recovery sequence contained in $Y$,
					then \eqref{eq:dfnGlimsup} holds for all $x\in X$.
			\end{enumerate}
	\end{lemma}

A subset $Y$
that fulfils the hypotheses of Lemma \ref{stm:prop-Gconv}\ref{stm:density-Gconv}
is said to be \emph{dense in energy} for $f_0$.
	
\chapter{Rate of convergence of the Bourgain-Brezis-Mironescu~formula}\label{ch:nle}
	Let $\Omega\subset \Rd$ be a smooth, bounded domain.
In the renown paper \cite{BBM},
J. Bourgain, H. Brezis, and P. Mironescu
considered iterated integrals of the form
	\begin{equation}\label{eq:BBM0}
		\int_{\Omega}\int_{\Omega} \rho(y-x)\frac{\va{u(y)-u(x)}^p}{\va{y-x}^{p}} \de y \de x,
	\end{equation}
where $\rho\in L^1(\Rd)$ is a radial mollifier and $u\in L^p(\Omega)$,
with $p \geq 1$.
They proved that, if $u$ belongs to the Sobolev space $W^{1,p}(\Omega)$,
then the $L^p$-norm of the weak gradient is achieved
as the limit of the quantity above
when $\rho$ approaches the Dirac delta in $0$.
The archetypical case is the study as $s\to 1^-$
of the \emph{Gagliardo seminorm},
which, for $s\in(0,1)$ and $u\in L^p(\Omega)$, is defined as
	\[
		\va{u}_{W^{s,p}(\Omega)}
				\coloneqq \left[
					\int_{\Omega}\int_{\Omega} \frac{\va{u(y)-u(x)}^p}{\va{y-x}^{d+sp}} \de y \de x
				\right]^{\frac{1}{p}}.
	\]
In this respect, the analysis in \cite{BBM} yields
	\[
	\lim_{s\to 1^-} (1-s)^{\frac{1}{p}} \va{u}_{W^{s,p}(\Omega)} = c\norm{\nabla u}_{L^p(\Omega)}
	\qquad\text{if } u\in W^{1,p}(\Omega),
	\]
with $c\coloneqq c(d,p)>0$.

In the same paper, when $p=1$, the authors also managed
to characterise the space of functions of bounded variations $\BV(\Omega)$
w.r.t. the asymptotic behaviour of \eqref{eq:BBM0}.
However, whether the total variation of the distributional gradient is attained in the limit
was an issue that they left unsolved.
Positive answers were soon provided independently
by J. D\'avila in \cite{Da} and by A. Ponce in \cite{P}.
Notably, the latter extended the analysis to more general integral functionals and
studied their limits also in sense of $\Gamma$-convergence.
This is the starting point of the current chapter,
which is based on \cite{CNP}
and is devoted to the study of the rate of convergence
of a class of functionals appeared in the paper by Ponce.
In Section \ref{sec:rateBBM}, we recall a result of his 
(see also Theorem \ref{stm:ponce} in the next chapter),
we state the main result of this chapter, Theorem \ref{stm:Gconv-BBM} below,
and we compare it to the existing literature.
The proof is presented in the last section.
It grounds on a slicing argument,
which reduces the problem to the study of suitable functionals
defined on functions of one real variable.
The analysis of these auxiliary $1$-dimensional functionals is developed in Section \ref{sec:1D}.
	
	\section{Main result and slicing}\label{sec:rateBBM}
		Let $\Omega$ be an open set with compact Lipschitz boundary,
let $G\colon \Rd \to [0,+\infty)$ be an $L^1(\Rd)$-function,
and let $f\colon [0,+\infty) \to [0,+\infty)$ be a convex function such that $f(0)=0$.
For $\epsilon>0$, we consider the $L^1$-norm preserving rescalings of $G$
	\[
		G_\eps(x)\coloneqq \frac{1}{\epsilon^d}G\left( \frac{x}{\eps}\right)
	\]
and, for $u\in L^1(\Omega)$, the nonlocal functionals
	\[
		\F_\eps(u)\coloneqq \int_{\Rd}\int_{\Rd}
			G_\eps(z) f \left(\frac{\left| u(x+z)-u(x) \right|}{\left| z\right|}\right) \de z \de x.
	\]
By \cite{P}, we know that, as $\eps\to 0^+$,
the family $\Set{\F_\eps}$ converges to the limit functional 
	\[
		\F_0(u)\coloneqq \int_{\Rd}\int_{\Rd} G(z) f( \va{\nabla u(x)\cdot\hat{z}} ) \de z \de x,
	\]
where $\hat z=z/\va{z}$ for all $z\in\Rdmz$.
Approximations by finite difference functionals have also been studied 
in relation to free discontinuity problems by M. Gobbino and M. Mora \cites{G,GM}.
We shall see that their method, the so called slicing,
can be adapted to our problem as well.

In this chapter,
we study the asymptotics of the functionals
that express the rate of convergence
of $\F_\epsilon$ to the limit $\F_0$, that is,
we let
	\begin{equation}\label{rate}
		\begin{split}
		\E_\eps (u) & \coloneqq \frac{\F_0(u)-\F_\eps(u)}{\eps^2} \\
					& =	\frac{1}{\eps^2} \int_{\Rd}\int_{\Rd}
								\left[
									G(z) f( \left|\nabla u(x)\cdot\hat{z}\right| )
										- G_\eps(z) f \left(\frac{\left| u(x+z)-u(x) \right|}{\left| z\right|}\right)
								\right] \de z \de x
		\end{split}
	\end{equation}
and we consider the limit of $\Set{\E_\eps}$ as $\epsilon$ tends to $0$.
We prove that the limit is a {\it second order} nonlocal functional (see \eqref{eq:E0-d}),
under the following set of assumptions:
	\begin{description}
		\item[A1:] $G(x) = G(-x)$ for a.e. $x\in \Rd$ and
			\begin{equation}\label{eq:intG}
				\int_{\Rd} G(x)\left( 1 + \left| x \right|^2 \right) \de x < +\infty.
			\end{equation}
		\item[A2:] there exist $r_0\geq 0$ and $r_1>0$ such that
			$r_0 < \beta_d\, r_1$, where
				\begin{equation}\label{eq:sigmad}
					\beta_d \coloneqq
						\begin{cases}
						1												& \text{when } d=2, \\
						\displaystyle{\frac{d-2}{d-1}}	& \text{when } d>2,
						\end{cases}
				\end{equation}
			and
				\begin{equation}\label{eq:Kpos}
					\mathrm{ess} \inf \Set{ G(x) : x \in B(0,r_1)\setminus B(0,r_0) } > 0.
				\end{equation}
		\item[A3:] $f$ is of class $C^2$, $f(0) = f'(0) = 0$, and 
				\begin{equation*}
				\text{there exists $\alpha>0$
					such that $2f(t)- \alpha t^2$ is convex.}
				\end{equation*}		
		\end{description}
We establish the following:
	\begin{thm}[Compactness and $\Gamma$-convergence of the rate functionals]
		\label{stm:Gconv-BBM}
		Let $\Omega$ be a bounded open set with Lipschitz boundary
		and let $G$ and $f$ fulfil {\bf A1}, {\bf A2}, and {\bf A3}.
		Let also
			\begin{equation*}\label{eq:X}
			X \coloneqq \Set{ u\in H^1_	\loc(\Rd) : \nabla u=0 \text{ a.e. in } \co{\Omega} }
			\end{equation*}
		and
			\begin{equation}\label{eq:E0-d}
				\E_0(u) \coloneqq\left\{
					\begin{aligned}
						\displaystyle{\frac{1}{24}\int_{\Rd} \int_{\Rd}
							G(z) \left| z \right|^2	f''( \left| \nabla u(x) \cdot \hat z \right| ) 
								\left| \nabla^2 u(x)\hat{z}\cdot \hat{z} \right|^2
							\de z \de x} \hspace{6Em} \\
						 \text{if } u \in X\cap H^2_\loc(\Rd), \\
					{+\infty \hspace{28.3Em} \text{otherwise}}.
					\end{aligned}\right.
			\end{equation}
		Then, there hold:
		\begin{enumerate}
			\item\label{stm:cptgen} \index{compactness!for the rate functionals for the BBM formula}
				For any family $\Set{u_\eps}\subset X$
				such that $\E_\eps(u_\eps)\le M$ for some $M>0$,
				there exists a sequence  $\Set{\eps_\ell}$ and
				a function $u\in H^2(\Rd)$ vanishing outside $\Omega$
				such that $\nabla u_{\eps_\ell}\to \nabla u$ in $L^2(\Rd)$ as $\ell\to +\infty$.
			\item\label{stm:Gliminf}\index{$\Gamma$-convergence!of the rate functionals for the BBM formula}
				For any family $\Set{u_\eps}\subset X$
				that converges to $u\in X$ in $H^1_\loc(\Rd)$,
					\[
						\E_0(u) \leq \liminf_{\eps \to 0^+} \E_\eps(u_\eps).
					\]
		\end{enumerate}
		\begin{enumerate}[label=(\it{iii}.\alph*),ref={\rm(iii.\alph*)}]
			\item\label{stm:Glimsup-a} For any $u\in X$ such that $\nabla u \in L^\infty(\Omega)$
			there exists a family $\Set{u_\eps}\subset X$
			that approaches $u$ in $H^1_\loc(\Rd)$ and satisfies
			\[	\limsup_{\epsilon\to 0^+} \E_\eps(u_\eps) \leq \E_0(u). \]
			\item\label{stm:Glimsup-b} If $f''$ is bounded,
			for any $u\in X$ there exists a family $\Set{u_\eps}\subset X$
			that converges to $u$ in $H^1_\loc(\Rd)$ with the property that
			\[	\limsup_{\epsilon\to 0^+} \E_\eps(u_\eps) \leq \E_0(u). \]
		\end{enumerate}
	\end{thm}	
	
	\begin{rmk}
		If we replace $X$ with $H^1_\loc(\Rd)$,
		it can be seen that for $u\in H^2_\loc(\Rd)$
		statements \ref{stm:Gliminf}, \ref{stm:Glimsup-a}, and \ref{stm:Glimsup-b}
		in Theorem \ref{stm:Gconv-BBM} still hold true,
		the proof remaining essentially unchanged.
		Conversely, the passage to the unbounded setting invalidates statement \ref{stm:cptgen}.
	\end{rmk}

The problem we tackle is not a higher order $\Gamma$-limit of $\F_\eps$,
which corresponds instead to studying the asymptotics of
	\[
		\frac{\F_\eps-\min\F_0}{\eps^\beta} \qquad \text{for some }\beta>0.
	\]
We refer to~\cite{AnB} for more details about developments by $\Gamma$-convergence.
On the other hand, our conclusion is reminiscent of the one obtained
in \cite{PPR} by M. Peletier, P. Planqu\'e and M. R\"oger,
who dealt with a model for bilayer membranes.
There, the authors considered the convolution functionals
	\[
		\tilde \F_\eps (u) \coloneqq \int_{\Rd} f \left( G_\eps \ast u\right) \de x,
	\]
which converge, as $\eps\to 0^+$, to
	\[
	\tilde \F_0(u) = c \int_{\Rd} f (u) \de x, \quad \text{with } c\coloneqq c(d,G)>0.
	\]
They showed that the rate functionals $\tilde E_\eps \coloneqq \eps^{-2}(\tilde \F_0-\tilde \F_\eps)$ 
converge pointwise for $u\in H^1(\R^d)$ to the functional
	\[
		\frac{1}{2} \int_{\Rd} \int_{\Rd} K(z) |z|^2 f''(u(x)) |\nabla u(x)\cdot \hat z|^2 \de z \de x,
	\]
and they proved as well that
the family $\Set{\tilde E_\eps(u)}$ is uniformly bounded if and only if $u \in H^1(\R^d)$.
We note that our analysis yields a similar regularity criterion:
indeed, if $\F_\eps(u)$ tends to $\F_0(u)$ sufficiently fast,
i.e. if $\Set{\E_\eps (u)}$ is uniformly bounded for all small $\epsilon$,
then $u\in H^2(\R^d)$,
see Remark \ref{stm:carattH2} below.

When the kernel $G$ is radially symmetric and $f$ is the modulus, 
our study is related to a geometric problem considered in \cite{MuS} 
to treat a physical model for liquid drops with dipolar repulsion.
However, the uniform convexity assumption on $f$,
which we utilise to prove the $\Gamma$-inferior limit inequality,
excludes this scenario from our analysis,
as well as the choice $f(t)=\va{t}^p$ with $p\ge 1$, $p\ne 2$.
Extensions of our result aimed at including these cases are a possible subject of investigation.

As for the assumptions on the kernel,
by {\bf A2} we require that the support of $G$
contains a sufficiently large annulus centred at the origin.
This might be regarded as a very weak isotropy hypothesis on $G$,
in a sense that will become clearer in the proof of Lemma \ref{stm:cpt-BBM}
The simplest case for which \eqref{eq:Kpos} holds is  
when there exists $\gamma >0$ such that $G(x) \geq \gamma$ for all $x\in B(0,r_1)$.

	\begin{rmk}[Radial case]
		When $G$ is radial, that is,
		$G(z) = \bar G( \left| z \right| )$ for some $\bar G\colon [0,+\infty) \to [0,+\infty)$,
		we find
		\[\begin{gathered}
		\F_0 (u) = \lVert G \rVert_{L^1(\Rd)} \int_{\Rd} \fint_{\Sdmu}
							f( \left| \nabla u(x)\cdot e \right| ) \de\Hdmu(e) \de x, \\
		\E_0 (u) = \frac{1}{24} \left( \int_{\Rd} G(z) \left| z \right|^2 \de z \right)
							\int_{\Rd} \fint_{\Sdmu}
							f''( \left| \nabla u(x) \cdot e \right| ) 
								\left| \nabla^2 u(x)e\cdot e \right|^2
							\de\Hdmu(e) \de x.
		\end{gathered}\]
	\end{rmk}

We prove Theorem \ref{stm:Gconv-BBM}
by a \emph{slicing} procedure,
which amounts to express the $d$-di\-men\-sional  energies $\E_\eps$ 
as superpositions of suitable $1$-dimensional energies $E_\eps$,
regarded as functionals on each line of $\Rd$.
Then, we establish the convergence of the functionals $E_\eps$
(see Section \ref{sec:1D}),
and we recover from it the result concerning $\E_\epsilon$
(see Section \ref{sec:Gconv-BBM}).

The slicing approach was introduced in \cites{G,GM}.
In our case, the next lemma shows that,
for $u\in L^2(\R)$ vanishing outside a given bounded interval and $\eps>0$,
the $1$-dimensional functionals retrieved by slicing are
	\begin{gather*}
		E_\eps(u) \coloneqq 
			\frac{1}{\eps^2} \int_\R \left[ f(u(x)) - f\left( \fint_{x}^{x+\eps}u(y)dy\right) \right] \de x, \\
		E_0(u)\coloneqq 
			\begin{cases}
				\begin{displaystyle}
					\frac{1}{24} \int_\R f''(u(x)) \left| u'(x) \right|^2 \de x
				\end{displaystyle} 		&\text{if } 
														u\in H^1(\R) \cap \Cc(\R),  \\
				+\infty 						&\text{otherwise}.
			\end{cases}
	\end{gather*}

We recall that, for $z\in \Rd\setminus\Set{0}$,
we set $\hat{z} \coloneqq z / \va{z}$ and
	\[
		\hat{z}^\perp \coloneqq \Set{ \xi \in \Rd : \xi \cdot z = 0}.
	\]

\begin{lemma}[Slicing] \label{stm:slicing}\index{slicing}
	For $u\in X$, $z\in \Rd\setminus\Set{0}$, and $\xi \in \hat{z}^\perp$,
	let
		\[\begin{array}{lccc}
			w_{\hat z,\xi}\colon 	& \R & \longrightarrow & \R \\
												& t & \longmapsto & u(\xi + t\hat{z}).
		\end{array}\]
	Then, $w'_{\hat z,\xi}(t) = \nabla u(\xi + t\hat{z}) \cdot \hat{z}$ and
		\begin{gather}
			\E_\eps(u) = \int_{\Rd}\int_{z^\perp} G(z) \left| z \right|^2 
					E_{ \eps\left| z \right| } (w'_{\hat z,\xi}) \de \Hdmu(\xi) \de z, \label{eq:slicing} \\
			\E_0(u) = \int_{\Rd}\int_{z^\perp} G(z) \left| z \right|^2 E_0(w'_{\hat z,\xi}) \de \Hdmu(\xi)\de z. \nonumber
		\end{gather}
	\end{lemma}	
	\begin{proof}
	We prove only \eqref{eq:slicing}, the case of $\E_0$ being identical.
	
	The formula is a simple application of Fubini's Theorem.
	Indeed, given the direction $\hat{z}\in\Sdmu$,
	for all $x\in\Rd$, we have $x = \xi + t\hat{z}$
	for some $\xi \in \Rd$ such that $\xi\cdot z = 0$ and $t\in\R$.
	Using this decomposition, we write
		\[\begin{split}
			\F_\eps(u)
				& = \int_{\Rd}\int_{\Rd} G(z) f\left(  \frac{ \left| u(x+\eps z)-u(x) \right| }{ \eps \left| z \right|}\right) \de z \de x \\
				& = \int_{\Rd}\int_{z^\perp} \int_{\R}
							G(z) f\left( 
									\frac{ \left| w_{\hat z,\xi}(t + \eps \left| z \right| ) - w_{\hat z,\xi}(t) \right| }{ \eps \left| z \right|}
									\right) dt  \de\Hdmu(\xi) \de z,
		\end{split}\]
	whence $\E_\eps(u)$ equals
		\[
			 \frac{1}{\eps^2}\int_{\Rd}\int_{z^\perp}\int_{\R} G(z) \left[
								f \big( \lvert w'_{\hat z,\xi}(t) \rvert \big)
								- f \left(
										\frac{ \left| w_{\hat z,\xi}(t+\epsilon\lvert z \rvert) - w_{\hat z,\xi}(t) \right| }{\eps \left| z \right|}
									\right)
							\right] dt \de\Hdmu(\xi) \de z.			
		\]
	We obtain \eqref{eq:slicing}
	by multiplying and dividing the integrands by $\left| z \right|^2$.
\end{proof}
	\section{$\Gamma$-convergence of the auxiliary functionals}\label{sec:1D}
		In the current section we analyse the limiting properties
of the family of auxiliary functionals $E_\epsilon$
obtained by slicing the $d$-dimensional energy $\E_\epsilon$.
So, for  $u\in L^2(\R)$ and $\eps>0$, we let
	\begin{equation}\label{eq:U}
		U(x) \coloneqq \int_{0}^x u(y) \de y,
		\qquad		
		\D_\eps U(x) \coloneqq \frac{U(x+\eps)-U(x)}{\eps} = \fint_{x}^{x+\eps}u(y) \de y,
	\end{equation}
and, as above, we define
	\begin{equation}\label{eq:Eh}
		E_\eps(u) \coloneqq  \frac{1}{\eps^2} \int_\R \left[ f(u(x)) - f\left(\D_\eps U(x)\right) \right] \de x.
	\end{equation}
We fix an open interval $I\coloneqq(a,b)\subset \R$ and
we consider the closed subspace of $L^2(\R)$
	\begin{equation}\label{eq:constr}
		Y \coloneqq \Set{ u \in L^2(\R) : u = 0 \text{ a.e. in } \co{I}}.
	\end{equation}
By positivity and convexity of $f$,
one can check that $E_\epsilon(u)$ belongs to $[0,+\infty]$
when $u\in Y$.
We want to compute the  $\Gamma$-limit of $\Set{E_\eps}$
thought as a collection of functionals on $Y$ endowed with the $L^2$-topology.
The candidate limit is
	\begin{equation}\label{eq:E0}
	E_0(u)\coloneqq 
	\begin{cases}
	\begin{displaystyle}
	\frac{1}{24} \int_\R f''(u(x)) \left| u'(x) \right|^2 \de x
	\end{displaystyle} 		&\text{if } 
	u\in Y\cap H^1(\R),  \\
	+\infty 						&\text{otherwise}.
	\end{cases} 
	\end{equation}	
Our result reads as follows:
	\begin{thm}[$\Gamma$-convergence of the auxiliary functionals]\label{stm:1D-Gconv} 
	Assume that $f$ is as in {\bf A3}.
	Then, as $\eps \to 0^+$,
	the restriction to $Y$ of the family $\Set{E_\eps}$ $\Gamma$-converges w.r.t. the $L^2(\R)$-topology
	to $E_0$, that is, for any $u\in Y$ the following hold:
	\begin{enumerate}
		\item \label{stm:1D-liminf}
		Whenever $\Set{u_\eps}\subset Y$ is a family that converges to $u$ in $L^2(\R)$,
		we have
			\[
				E_0(u) \leq \liminf_{\eps \to 0^+} E_\eps(u_\eps).
			\]
		\item \label{stm:1D-limsup}
		There exists $\Set{u_\eps} \subset Y$ converging to $u$ in $L^2(\R)$ such that
			\begin{equation}\label{eq:Glimsup-Eeps}
				\limsup_{\eps \to 0^+} E_\eps(u_\eps) \leq E_0(u).
			\end{equation}
	\end{enumerate}	
\end{thm}

	\begin{rmk}[Lower order limits] \label{remscala1}
		Taking into account Proposition \ref{stm:1D-pointlim} below, {\it a posteriori}
		$\Set{\epsilon^\beta E_\epsilon}$ $\Gamma$-converges to $0$ for all $\beta>0$.
		However, when $\beta = 1,2$, this can be proved straightforwardly.
		Indeed, if for $\eps>0$ we let
			\[
			F_0(u) \coloneqq \int_\R f(u(x))\de x
			\quad\text{and}\quad
			F_\eps(u) \coloneqq \int_\R f(\D_\eps U(x)) \de x
			\]
		(recall position \eqref{eq:U}),
		then it is easy to see that $\lim_{\eps \to 0^+}F_\eps(u)= F_0(u)$,
		and, when $u\in Y \cap C^2(\R)$, we also find
			\begin{equation*}
			\lim_{\eps \to 0^+}\frac{F_0(u)-F_\eps(u)}{\eps}
					=\frac{1}{2}\left[f(u(a))+f(u(b))\right]
					= 0.
			\end{equation*}		
	\end{rmk}

We separate the proofs of the two statements in Theorem \ref{stm:1D-Gconv}.
We focus on the lower limit in Proposition \ref{stm:1D-Gliminf},
which is based on a estimate from below of the energy and on a compactness criterion,
established respectively in Lemma \ref{stm:1D-enbound} and in Lemma \ref{stm:1D-cpt}.
We deal with the $\Gamma$-upper limit in the next proposition,
where we compute the pointwise limit, as $\eps \to 0^+$, of $E_\eps(u)$ for $u$ smooth.

\begin{prop}\label{stm:1D-pointlim}
	For all $u\in Y\cap C^2(\R)$,
		\begin{equation*}
		\lim_{\eps \to 0^+} E_\eps(u) = E_0(u).
		\end{equation*}
	More precisely,
	there exists a continuous, bounded, and increasing function
	$m\colon [0,+\infty) \to [0,+\infty)$ such that $m(0)=0$ and
		\begin{equation}\label{eq:1D-point-est}
			\left| E_\eps(u) - E_0(u) \right| \leq c\, m(\eps),
		\end{equation}
	where
	$c \coloneqq
	c( b-a,\lVert u \rVert_{C^2(\R)},\lVert f \rVert_{C^2([-\lVert u \rVert_{C^2(\R)},\lVert u \rVert_{C^2(\R)}])} ) >0$
	is a constant.
	
	Furthermore, for every $u\in Y$, there exists a family $\Set{u_\eps} \subset Y$
	that converges to $u$ in $L^2(\R)$ and that satisfies \eqref{eq:Glimsup-Eeps}.
	\end{prop}
	\begin{proof}
	Since $u\in Y\cap C^2(\R)$ and $f\in C^2(\R)$,
	$\epsilon^2 E_\epsilon(u)$ and $E_0(u)$ are uniformly bounded in $\eps$.
	Consequently, 
		\begin{equation}\label{eq:h>1}
		\left| E_\eps(u) - E_0(u) \right| \leq c_\infty \qquad \text{for } \eps>1
		\end{equation}
	for some constant $c_\infty>0$.
	
	Next, we assume $\eps\in(0,1]$.
	If $x\notin (a-\eps,b)$, then $\D_\eps U(x) = 0$, and hence
		\[
			\epsilon^2 E_\epsilon(u) =
				\int_{a}^{b} \left[ f(u(x)) - f\left(\D_\eps U(x)\right) \right]\de x - \int_{a-\eps}^a f(\D_\eps U(x)) \de x.
		\]
	By the regularity of $u$, for any $x\in(a-\eps,b)$
	we have the Taylor's expansion 
		\begin{equation*}
			\D_\eps U(x)=u(x)+\frac{\eps}{2}u'(x)+\frac{\eps^2}{6}u''(x_\eps), \qquad\text{with $x_\eps\in(x,x+\eps)$},
		\end{equation*}
	which we rewrite as
		\begin{equation}\label{eq:Taylor-DhU}
			\D_\eps U(x)=u(x)+\eps v_\eps(x), \qquad\text{with } v_\eps(x)\coloneqq \frac{u'(x)}{2}+\frac{\eps}{6}u''(x_\eps);
		\end{equation}
	note that $v_\eps$ converges uniformly to $u'/2$ as $\eps \to 0^+$.
	
	Plugging \eqref{eq:Taylor-DhU} into the definition of $E_\eps$, we find
	\[\begin{split}
		\eps^2 E_\epsilon(u)
			& = - \int_a^b \left[f\big( u(x) + \eps v_\eps(x) \big) - f(u(x)) \right]\de x
					- \int_{a-\eps}^a f\left( \frac{\eps^2}{6} u''(x_\eps) \right) \de x \\
			& = -\eps \int_a^b f'(u(x))v_\eps(x) \de x - \frac{\eps^2}{2}\int_a^b f''(w_\eps(x))v_\eps(x)^2 \de x \\
			& \quad - \int_{a-\eps}^a f\left( \frac{\eps^2}{6} u''(x_\eps) \right) \de x,
	\end{split}\]
	where $w_\eps$ fulfils $w_\eps(x)\in(u(x),u(x)+\eps v_\eps(x_\eps))$ for all $x\in(a,b)$.
	
	It is not difficult to see that
		\[
			\left| \int_{a-\eps}^a f\left( \frac{\eps^2}{6} u''(x_\eps) \right) \de x \right| \leq c_1 \eps^5,
		\]
	for a constant $c_1>0$
	that depends only on $N\coloneqq \lVert u \rVert_{C^2(\R)}$
	and on $\lVert f'' \rVert_{L^\infty( [-N,N] )}$.
	Also, recalling the definition of $v_\eps$, we have
		\[
			\int_a^b f'(u(x))v_\eps(x) \de x
				=  \frac{\eps}{6}\int_a^b f'(u(x)) u''(x_\eps) \de x,
		\]
	and, therefore,
	\begin{equation}\label{eq:Eh-E0}
		\begin{split}
		\left| E_\eps(u) - E_0(u) \right| 
			& \leq \frac{1}{6}\left| - \int_a^b f'(u(x)) u''(x_\eps) \de x
				- \int_a^b f''(u(x)) u'(x)^2 \de x \right|
		\\ & \quad + \frac{1}{2} \left| \frac{1}{4} \int_a^b f''(u(x)) u'(x)^2 \de x
				- \int_a^b f''(w_\eps(x))v_\eps(x)^2 \de x \right| 
				+ c_1 \eps^3.
	\end{split}		
	\end{equation}
	Since $u\in Y\cap C^2(\R)$,
	$u''$ admits a uniform modulus of continuity $m_{u''}\colon [0,+\infty) \to [0,\infty)$.
	An integration by parts yields
		\[\begin{split}
			\left| -\int_a^b f'(u(x)) \right. & \left.\! u''(x_\eps) \de x  - \int_a^b f''(u(x)) u'(x)^2 \de x \right| \\
			&\leq   \int_a^b \left| f'(u(x)) \right|  \left| u''(x) - u''(x_\eps) \right| \de x \\
			&\leq  c_2 m_{u''}(\eps),
	\end{split}\]
	where $c_2 \coloneqq (b-a) \lVert f' \rVert_{L^\infty( [-N, N] )}$.
	
	Analogously,
	letting $m_{f''}$ be the modulus of continuity of the restriction of $f''$ to the interval $[-N,N]$,
	we find as well
		\[\begin{split}
			& \left| \frac{1}{4} \int_a^b f''(u(x)) u'(x)^2 \de x - \int_a^b f''(w_\eps(x))v_\eps(x)^2 \de x \right| 
			\\ &\quad  \leq  \int_a^b \left| f''(u(x)) \right|  \left| \frac{1}{4} u'(x)^2 -  v_\eps(x)^2\right| \de x 
				+   \int_a^b \left| f''(u(x)) - f''(w_\eps(x)) \right|  v_\eps(x)^2 \de x
			\\ &\quad  \leq c_3( \eps + m_{f''}(\eps) ),
	\end{split}\]
	with $c_3$ depending on $b-a$, $N$, and $\lVert f'' \rVert_{L^\infty( [-N, N] )}$.
	
	By combining \eqref{eq:Eh-E0} with the inequalities above,
	we obtain
	\begin{equation}\label{eq:h<1}
	\left| E_\eps(u) - E_0(u) \right| \leq c_0 \big( m_{u''}(\eps) + m_{f''}(\eps) +\eps + \eps^3 \big)
	\qquad \text{for } \eps\in(0,1],
	\end{equation}
	for a suitable constant $c_0>0$.
	Now, \eqref{eq:h>1} and \eqref{eq:h<1} get \eqref{eq:1D-point-est}.
	
	For what concerns \eqref{eq:Glimsup-Eeps},
	we construct a subset of $Y$ that is dense in energy,
	and we invoke Lemma \ref{stm:prop-Gconv}.
	The argument is standard,
	so we just sketch it.	
	If $u\in Y \setminus H^1(\R)$,
	\eqref{eq:Glimsup-Eeps} holds trivially.
	Else, by rescaling the domain and mollifying,
	we can produce a sequence of smooth functions $\Set{u_\ell} \subset Y$
	that approximate $u$ both uniformly and in $H^1(\R)$.
	Thanks to the continuity of $f''$, we see that
	$\lim_{\ell \to +\infty} E_0(u_\ell) = E_0(u)$.
	Also, $\lim_{\eps \to 0^+} E_\eps(u_\ell) = E_0(u_\ell)$
	for any $\ell\in \N$, because $u_\ell$ is smooth.
	Then, we are in position to apply Lemma \ref{stm:prop-Gconv}.
\end{proof}


With Proposition \ref{stm:1D-pointlim} on hand,
to achieve the proof of Theorem \ref{stm:1D-Gconv},
it suffices to validate statement \ref{stm:1D-liminf},
namely, for each $u\in Y$ and for each family $\Set{u_\eps}\subset Y$ converging to $u$ in $L^2(\R)$
it holds
	\[	E_0(u)\leq \liminf_{\eps \to 0^+} E_\eps(u_\eps). 	\]
It is in the proof of this inequality that
the assumption of strong convexity of $f$ comes into play:
it enables us to provide a lower bound on the energy $E_\eps$,
which, in turn, yields that
sequences with equibounded energy are precompact w.r.t. the $L^2$-topology.

\begin{lemma}[Lower bound on the energy]\label{stm:1D-enbound}
	Assume that {\bf A3} is satisfied.
	Then, for any $u\in Y$, there holds
		\begin{equation}\label{eq:1D-lbound}
			E_\eps(u)\geq \sup_{\phi \in \Cc^\infty\left(\R^2\right)}
			\left\{ \int_\R\fint_x^{x+\eps}
			\left(
			\frac{u(y) - \D_\eps U(x)}{\eps} \phi(x,y) - \frac{\phi(x,y)^2}{4\lambda_\eps(x,y)}
			\right) \de y \de x
			\right\},
		\end{equation}
	with		
	\begin{equation} \label{eq:lambda}
	\lambda_\eps(x,y)\coloneqq \int_0^1 (1-\theta)f''\big((1-\theta)\D_\eps U(x)+\theta u(y)\big)d\theta.
	\end{equation}
	Moreover, 
	\begin{equation}\label{eq:1D-lbound2}
	E_\eps(u) \geq \frac{\alpha}{4} \int_\R\int_{-\eps}^{\eps}
		H_\eps(r)\left(\frac{u(y+r) - u(y)}{\eps}\right)^2 \de r\de y,
	\end{equation}
	where
	\begin{equation*}
		H(r)\coloneqq (1-\va{r})^+
		\quad\text{and}\quad
		H_\eps(r)\coloneqq \frac{1}{\eps}H\left(\frac{r}{\eps}\right).
	\end{equation*}
\end{lemma}

\begin{proof}	
	Given $\eps>0$, let $u\in Y$ be such that $E_\eps(u)$ is finite.
	We write
		\[
		E_\eps(u) = \frac{1}{\eps^2} \int_\R e_\eps(x) \de x,
		\quad\text{where } e_\eps(x) \coloneqq \fint_x^{x+\eps} [f(u(y))-f(\D_\eps U(x))] \de y.
		\]
	The identity
		\[
			f(s)-f(t) = f'(t)(s-t)+(s-t)^2 \int_0^1 (1-\theta) f''((1-\theta)t+\theta s))d\theta,
		\]
	gets
		\begin{equation}\label{eq:Ch}
		e_\eps(x) = \fint_x^{x+\eps} \lambda_\eps(x,y) \left(u(y) - \D_\eps U(x)\right)^2 \de y,
		\end{equation}
	with $\lambda_\eps(x,y)$ as in \eqref{eq:lambda}.
	We now make use of the identity $a^2 \geq 2ab - b^2$ for $a,b\in\R$
	to recover the bound
		\[
		e_\eps(x) \geq
			\fint_x^{x+\eps} \left( \frac{u(y) - \D_\eps U(x)}{\eps} \phi(x,y) - \frac{\phi(x,y)^2}{4\lambda_\eps(x,y)} \right) \de y,
		\]
	for any $\phi \in C^\infty_c\left(\R^2\right)$.
	Hence, \eqref{eq:1D-lbound} is proved.
		
	Thanks to the strong convexity of $f$,
	$\lambda_\eps(x,y) \ge \alpha / 2$ for all $(x,y)\in \R^2$ and $\eps>0$.
	Thus, from \eqref{eq:Ch} we infer
		\[
			e_h(x)\geq \frac{\alpha}{2}\fint_x^{x+\eps}\left(u(y) - \D_\eps U(x)\right)^2,
		\]
	and, therefore,
		\begin{equation*}
		\begin{split}
			E_\eps(u)
				& \geq \frac{\alpha}{2} \int_\R\fint_x^{x+\eps}\left(\frac{u(y) - \D_\eps U(x)}{\eps}\right)^2\de y\de x \\
				& \geq \frac{\alpha}{4} \int_\R\fint_x^{x+\eps}\fint_x^{x+\eps}\left(\frac{u(z) - u(y)}{\eps}\right)^2
						\de z\de y\de x.
		\end{split}
		\end{equation*}
	The last inequality is a consequence of the identity
		\[ 
			\int |\phi(y)|^2 \de\mu(y) = \left|\int \phi(y) d\mu(y)\right|^2 + \frac{1}{2}\int\int |\phi(z)-\phi(y)|^2\de\mu(z)\de\mu(y),
		\]
	which holds whenever $\mu$ is a probability measure and $\phi\in L^2(\mu)$.
	By Fubini's Theorem and neglecting contributions near the boundary,
	we find:
		\begin{equation*}
		\begin{split}
			E_\eps(u) & \geq \frac{\alpha}{4} \int_\R\fint_{y-\eps}^y\fint_x^{x+\eps}
				\left(\frac{u(z) - u(y)}{\eps}\right)^2 \de z\de x\de y \\
				& = \frac{\alpha}{4\eps} \int_\R\int_{y-\eps}^{y+\eps}
					\left(1-\frac{\left|z-y\right|}{\eps}\right)\left(\frac{u(z) - u(y)}{\eps}\right)^2 \de z\de y.
	\end{split}
	\end{equation*}
	The change of variables $r=z-y$ yields \eqref{eq:1D-lbound2}.
\end{proof}

\begin{rmk}\label{stm:f''bounded}
	Let $u\in Y \cap H^1(\R)$.
	Then, the family $\Set{E_\eps(u)}$ is bounded above
	if there exists $c>0$ such that  $f''\leq c$.
	To see this,
	we observe that \eqref{eq:Ch} and the definition of $\lambda_\eps$ yield
	\begin{equation*}
	E_\eps(u) \leq \frac{c}{2 \eps^2} \int_{\R} \fint_x^{x+\eps} \left(u(y) - \D_\eps U(x)\right)^2 \de y \de x.
	\end{equation*}
	Being $u$ in $H^1(\R)$, for $y\in (x,x+\eps)$ we have
	\[
	\left| u(y) - \D_\eps U(x) \right|^2
	\leq \eps \int_{x}^{x+\eps} \left| u'(z)\right|^2 \de z
	= \eps^2 \int_{0}^{1} \left| u'(x + \eps z)\right|^2 \de z,
	\]
	and we conclude
	\begin{equation}\label{eq:Ehbounded}
	E_\eps(u) \leq \frac{c}{2} \int_{\R} \left| u'(x )\right|^2 \de x.
	\end{equation}
	
	The current remark will come in handy to prove a characterisation of $H^2_\loc(\Rd)$,
	see Remark \ref{stm:carattH2} below.
\end{rmk}

\begin{lemma}[Compactness]\label{stm:1D-cpt}
	Assume that $f$ fulfils {\bf A3}.
	If $\Set{u_\eps}\subset Y$ is a family such that
	$E_\eps(u_\eps)\leq M $ for some $M\geq 0$,
	then, there exist a sequence $\Set{\eps_\ell}$ and a function $u\in Y \cap H^1(\R)$
	such that $\Set{u_{\eps_\ell}}$ converges to $u$ in $L^2(\R)$, as $\ell\to+\infty$.		
\end{lemma}
\begin{proof}	
	We adapt the proof of \cite{AB2}*{Theorem 3.1}.
	
	By Lemma \ref{stm:1D-enbound}, we infer that
	\begin{equation}\label{eq:1D-lbound-bis}
	\frac{\alpha}{4} \int_\R\int_{-\eps}^{\eps}
	H_\eps(r)\left(\frac{ u_\eps(y+r) -  u_\eps(y)}{h}\right)^2 \de r \de y \leq M.
	\end{equation}
	Observe that $H_\eps(r)\de r$ is a probability measure on $[-\eps,\eps]$.
	
	Let $\rho\in \Cc^{\infty}(\R)$ be a mollifier
	such that its support is contained in $[-1,1]$,
	$0\leq \rho \leq H$, and $\left|\rho'\right| \leq H$.
	For all $\epsilon>0$, we define
	 	\[
			 \rho_\eps(r)\coloneqq \frac{1}{c\eps}\rho\left(\frac{r}{\eps}\right),
			 \qquad \text{with }c\coloneqq \int_\R\rho(r)\de r,
	 	\]
	and we introduce the regularised functions $v_\eps \coloneqq \rho_\eps \ast  u_\eps$,
	Each $v_\eps\colon \R\to\R$ is a smooth function
	whose support is a subset of $(a-\eps,b+\eps)$,
	and the family of derivatives $\Set{v'_\eps}_{\eps \in(0,1)}$ is uniformly bounded in $L^2(\R)$.
	Indeed, since $\int_\R \rho'(r)\de r =0$, we have
		\[\begin{split}
			\int_\R \left| v_\eps'(y) \right|^2 \de y 
				& = \int_\R \left| \int_{-\eps}^{\eps} \rho'_\eps(r)[ u_\eps(y+r) - u_\eps(y)]\de r \right|^2 \de y \\
				& \le \int_\R\left(\int_{-\eps}^{\eps}
							\left|\rho'_\eps(r)\right|\left| u_\eps(y+r) - u_\eps(y) \right| \de r \right)^2 \de y \\
				& \le \frac{1}{c^2}\int_\R 
							\left(\int_{-\eps}^{\eps} H_\eps(r) \left|
							\frac{ u_\eps(y+r) - u_\eps(y)}{\eps}\right| \de r \right)^2 \de y \\
				& \le \frac{1}{c^2}\int_\R \int_{-\eps}^{\eps} H_\eps(r) \left|
							\frac{ u_\eps(y+r) -  u_\eps(y)}{\eps}\right|^2 \de r \de y,
		\end{split}\]
	whence
		\begin{equation}\label{eq:vh-H^1}	
			\int_\R \left|v_\eps'(y)\right|^2 \de y \leq \frac{4M}{c^2\alpha}.
		\end{equation}
	
	For all $\eps\in(0,1)$, let $\tilde v_\eps$ be the restriction of $v_\eps$ to the interval $(a-1,b+1)$.
	By virtue of Poincar\'e's inequality,
	\eqref{eq:vh-H^1} entails boundedness in $H^1_0((a-1,b+1))$
	of the family $\Set{\tilde v_\eps}_{\eps\in(0,1)}$,
	and, by Sobolev's Embedding Theorem,
	there exist a sequence $\Set{\eps_\ell}$ and a function $\tilde u\in H^1_0([a-1,b+1])$ such that
	$\Set{\tilde v_{\eps_\ell}}$ uniformly converges to $\tilde u$.
	Since each $\tilde v_{\eps_\ell}$ is supported in $(a-\eps_\ell,b+\eps_\ell)$,
	we see that $\tilde u\in H^1_0(\bar I)$.
	It follows that $\Set{v_{\eps_\ell}}$ converges uniformly to $u \in Y\cap H^1(\R)$
	if we set
		\[
			u(x)\coloneqq
				\begin{cases}
					\tilde u(x) 	&\text{if } x\in \bar I, \\
					0					& \text{otherwise}.
				\end{cases}
		\]
	
	Eventually,
	we focus on the $L^2$-distance between $u_\eps$ and $v_\eps$.
	As above, we have 
		\[\begin{split}
			\int_\R \left|v_\eps(y)- u_\eps(y)\right|^2 \de y 
				& =  \int_\R\left|\int_{-\eps}^{\eps} \rho_\eps(r)[ u_\eps(y+r) - u_\eps(y)]\de r \right|^2 \de y \\
				& \leq  \int_\R \int_{-\eps}^{\eps} \rho_\eps(r)\left| u_\eps(y+r) - u_\eps(y)\right|^2 \de r \de y \\
				& \leq  \frac{1}{c}\int_\R \int_{-\eps}^{\eps} H_\eps(r) \left| u_\eps(y+r) - u_\eps(y) \right|^2 \de r \de y,
		\end{split}	\]
	and, by \eqref{eq:1D-lbound-bis}, we get
		\begin{equation*}
			\int_\R \left|v_\eps(y) - u_\eps(y)\right|^2 \de y  \leq \frac{4M}{c\alpha} \eps^2.
		\end{equation*}
	The thesis is now proved, because
	there exists a subsequence $\Set{v_{\eps_\ell}}$ that
	converges uniformly to $u \in Y\cap H^1(\R)$.
\end{proof}

At this stage, we are in position to establish
statement \ref{stm:1D-liminf} in Theorem \ref{stm:1D-Gconv},
thus concluding the proof of the latter.

	\begin{prop}\label{stm:1D-Gliminf}
	Let $f$ satisfy {\bf A3}. Then, for any $u\in Y$ and for any family $\Set{u_\eps}\subset Y$
	that converges to $u$ in $L^2(\R)$, it holds
		\begin{equation}\label{eq:1D-Gliminf}
			E_0(u)\leq \liminf_{\eps \to 0^+} E_\eps(u_\eps).
		\end{equation}
	\end{prop}

\begin{proof}
	Let us pick $u,u_\eps\in Y$ in such a way that $u_\eps\to u$ in $L^2(\R)$.
	We may suppose that the right-hand side in \eqref{eq:1D-Gliminf} is finite,
	otherwise the conclusion holds trivially.
	Then, up to extracting a subsequence, which we do not relabel,
	there exists $\lim_{\eps \to 0^+} E_\eps(u_\eps)$ and it is finite.
	In particular, for all $\eps>0$, $E_\eps(u_\eps)\leq M$ for some $M\geq 0$,
	so that, by Lemma \ref{stm:1D-cpt}, $u\in Y \cap H^1(\R)$.
	
	We apply formula \eqref{eq:1D-lbound} for each $u_\eps$,
	choosing, for $(x,y) \in \R^2$, 
	\[
	\phi(x,y) = \psi\left(x,\frac{y-x}{\eps}\right),
	\qquad \text{with } \psi\in \Cc^\infty(\R^2).
	\]
	We get
		\begin{equation}\label{eq:psi-estimate}
			\begin{split}
			E_\eps(u_\eps) \geq	
				& \int_\R \fint_x^{x+\eps}\frac{u_\eps(y) - \fint_{x}^{x+\eps} u_\eps}{\eps}
								\psi\left(x,\frac{y-x}{\eps}\right) \de y \de x \\
				& - \frac{1}{4} \int_\R\fint_x^{x+\eps}
											\frac{\psi\left(x,\tfrac{y-x}{\eps}\right)^2}{\lambda_\eps(x,y)} \de y \de x,
			\end{split}
		\end{equation}
	where, in agreement with \eqref{eq:lambda},
		\begin{equation*}
			\lambda_\eps(x,y) \coloneqq
				\int_0^1 (1-\theta)f''\left((1-\theta)\fint_{x}^{x+\eps}u_\eps(z)\de z+\theta u_\eps(y)\right)d\theta
			\geq \frac{\alpha}{2}.
	\end{equation*}
	
	Let us focus on the right-hand side of \eqref{eq:psi-estimate}.
	We have
		\[\begin{split}
			\frac{1}{\eps}\int_\R\fint_x^{x+\eps} &
				\left(\fint_{x}^{x+\eps} u_\eps(z) \de x \right)\psi\left(x,\frac{y-x}{\eps}\right) \de y \de x \\
	= & \frac{1}{\eps^3}\int_\R\int_x^{x+\eps}\int_x^{x+\eps} u_\eps(z) \psi\left(x,\frac{y-x}{\eps}\right) \de y \de z \de x \\
	= 
	& \frac{1}{\eps^3}\int_\R\int_{z-\eps}^{z} \int_x^{x+\eps}u_\eps(z) \psi\left(x,\frac{y-x}{\eps}\right) \de y \de x \de z,
	\end{split}\]
	and, similarly,
	\begin{equation}\label{eq:psi}
	\begin{split}
	\int_\R\fint_x^{x+\eps} & \frac{u_\eps(y) - \fint_{x}^{x+\eps} u_\eps}{\eps} \psi\left(x,\frac{y-x}{\eps}\right) \de y \de x \\
	= 
	& \frac{1}{\eps}\int_\R\fint_{y-\eps}^{y} \fint_x^{x+\eps}
	u_\eps(y)\left[\psi\left(x,\frac{y-x}{\eps}\right) - \psi\left(x,\frac{z-x}{\epsilon}\right)\right] 
	\de z \de x \de y.
	\end{split}
	\end{equation}
	By changing variables, we get
	\[
	\begin{aligned}
	\fint_{y-\eps}^{y} \fint_x^{x+\eps} \psi\left(x,\frac{y-x}{\eps}\right) \de z \de x
	& = \int_0^1\int_0^1\psi(y - \eps r,r) \de q \de r, \\
	\fint_{y-\eps}^{y} \fint_x^{x+\eps} \psi\left(x,\frac{z-x}{\eps}\right) \de z \de x
	& = \fint_{y-\eps}^{y} \int_0^1 \psi(x,r) \de r \de x \\
	& = \int_0^1 \int_0^1 \psi(y-\eps q,r) \de q \de r,
	\end{aligned}
	\]
	hence
	\[\begin{split}
	\frac{1}{\eps} \fint_{y-\eps}^{y} \fint_x^{x+\eps} 
	& \left[\psi\left(x,\frac{y-x}{\eps}\right) - \psi\left(x,\frac{z-x}{\eps}\right)\right] \de z \de x \\
	= & \int_0^1 \int_0^1 \frac{ \psi(y - \eps r, r) -  \psi(y - \eps q, r) }{h} \de q \de r \\
	= & -\int_0^1 \int_0^1 \int_q^r \partial_1 \psi(y - \eps s, r) \de s \de q \de r \\
	= & -\int_0^1 \int_0^1 (r-q) \fint_q^r \partial_1 \psi(y-\eps s,r) \de s \de q \de r.
	\end{split}\]
	Since $\psi$ is smooth, we have that
	$\partial_1\psi(y - \eps s,r) = \partial_1\psi(y,r)+O(\eps)$ as $\eps \to 0^+$, uniformly for $s\in [0,1]$.
	Consequently,
	\begin{multline*}
	\frac{1}{\eps} \fint_{y-\eps}^{y} \fint_x^{x+\eps} 
	\left[\psi\left(x,\frac{y-x}{\eps}\right) - \psi\left(x,\frac{z-x}{\eps}\right)\right] \de z \de x \\
	= -  \int_0^1 \left(r-\frac{1}{2}\right) \partial_1\psi(y,r) \de r + O(\eps).
	\end{multline*}
	Plugging this equality in \eqref{eq:psi} yields
	\begin{multline*}
	\int_\R\fint_x^{x+\eps} \frac{u_\eps(y) - \fint_{x}^{x+\eps} u_\eps}{h} \psi\left(x,\frac{y-x}{\eps}\right) \de y \de x
	\\ =  - \int_\R u_\eps(y)
	\int_0^1 \left(r-\frac{1}{2}\right) \partial_1\psi(y,r) \de r + O(\eps).
	\end{multline*}
	Since $u_\eps \to u$ in $L^2(\R)$,
	it is possible to take the limit $\eps \to 0^+$ in the previous formula,
	getting
	\begin{multline}\label{eq:1st}
	\lim_{\eps \to 0^+} 
	\int_\R\fint_x^{x+\eps} \frac{u_\eps(y) - \fint_{x}^{x+\eps} u_\eps}{h} \psi\left(x,\frac{y-x}{\eps}\right) \de y \de x 
	 \\ =- \int_\R \int_0^1 u(y) (r-\tfrac{1}{2}) \partial_1\psi(y,r) \de r \de y.
	\end{multline}
	
	Next, we focus on the second addendum
	in the right-hand side of \eqref{eq:psi-estimate}.
	By Fubini's Theorem and a change of variables, we have
	\begin{equation*}
	\begin{split}
	\int_\R\fint_x^{x+\eps} \frac{\psi(x,\frac{y-x}{\eps})^2}{\lambda_\eps(x,y)} \de y \de x 
	=
	\int_\R \int_{0}^{1} \frac{\psi\left(y - \eps r, r\right)^2}{\lambda_\eps(y - \eps r, y)} \de r \de y.
	\end{split}
	\end{equation*}
	The function $\psi$ has compact support and $\lambda_\eps\geq \alpha /2$ for all $\eps>0$,
	therefore we can apply Lebesgue's Convergence Theorem
	to let $\eps$ vanish in the equality above.
	We find
	\begin{multline*}
	\lim_{\eps \to 0^+} \int_\R \int_{0}^{1}
	\frac{\psi\left(y - \eps r, r\right)^2}{\int_0^1 (1-\theta) f''\left((1-\theta)\fint_{y - \eps r}^{y +(1-r)\eps}u_\eps(z)\de z+\theta u_\eps(y)\right)d\theta} \de r \de y \\				 
	= \int_\R \int_0^1 \frac{\psi(y,r)^2}{\int_0^1 (1-\theta) f''(u(y))d\theta} \de r \de y,
	\end{multline*}
	thus
	\begin{equation}\label{eq:2nd}
	\lim_{\eps \to 0^+} \int_\R\fint_x^{x+\eps} \frac{\psi(x,\frac{y-x}{\eps})^2}{\lambda_\eps(x,y)} \de y \de x = 
	2\int_\R \int_0^1 \frac{\psi(y,r)^2}{ f''(u(y))} \de r \de y.
	\end{equation}		
	Summing up, by \eqref{eq:1st} and \eqref{eq:2nd},
	we deduce
	\begin{equation*}
	\liminf_{\eps \to 0^+} E_\eps(u_\eps) \geq 
	- \int_\R \int_0^1 \left[
	u(y) \left(r-\frac{1}{2}\right) \partial_1\psi(y,r) \de r \de y
	+ \frac{1}{2} \int_\R \int_0^1 \frac{\psi(y,r)^2}{ f''(u(y))}
	\right] \de r \de y,
	\end{equation*}
	for all $\psi \in \Cc^\infty (\R^2)$.
	
	Let $\eta\in \Cc^\infty(\R)$. By a standard approximation argument,
	one can prove that
		\[
			\psi(x,y) \coloneqq \eta(x) \left(y- \frac{1}{2}\right)
		\]
	is an admissible test function in the previous estimate.
	This choice of $\psi$ gets
		\[
			\begin{split}
				\liminf_{\eps \to 0^+} E_\eps(u_\eps) & \geq 
					- \int_0^1 \left(r-\frac{1}{2}\right)^2\de r
						\int_\R  u(y) \eta'(y) \de y \\
				&\quad 
					- \frac{1}{2} \int_0^1 \left(r- \frac{1}{2}\right)^2 \de r
								\int_\R \frac{\eta(y)^2}{ f''(u(y))} \de y \\
				& = \frac{1}{12} \left[ \int_\R u'(y) \eta(y) \de y
						- \frac{1}{2} \int_\R \frac{\eta(y)^2}{ f''(u(y))} \de y \right],
			\end{split}
		\]
	where we used the identity $\int_0^1 (r -1/2)^2 \de r = 1/12$
	and $u'\in L^2(\R)$ is the distributional derivative of $u$
	(recall that $u\in H^1(\R)$).		
	Since the test function $\eta$ is arbitrary,
	we recover \eqref{eq:1D-Gliminf} by taking the supremum
	w.r.t. $\eta \in \Cc^\infty(\R)$.
\end{proof}
	\section{$\Gamma$-convergence of the rate functionals}\label{sec:Gconv-BBM}
		We devote this section to the proof of Theorem~\ref{stm:Gconv-BBM},
which relies on the $\Gamma$-convergence result of the previous section.

We firstly remark that, when the dimension is $1$,
the $\Gamma$-convergence of the functionals $\E_\eps$ follows easily from Theorem \ref{stm:1D-Gconv}.

\begin{cor}\label{stm:1D-cor}
	Let $\Omega \subset \R$ be bounded and open,
	and let $X \coloneqq \Set{ u\in H^1_\loc(\R) : u' = 0 \text{ in } \co{\Omega} }$.
	Suppose that
	$G\colon \R \to [0,+\infty)$ and $f\colon [0,+\infty) \to [0,+\infty)$
	satisfy {\bf A1}, {\bf A2}, and {\bf A3}.
	If for $\eps>0$ and $u\in H^1_\loc(\R)$ we define
		\[
			\E_\eps(u) \coloneqq \frac{1}{\eps^2}\int_{\R} \int_{\R}
								G_\eps(z) \left[
								f\big( \lvert u'(x) \rvert \big) -  f \left(\left|\frac{u(x+z)-u(x)}{z}\right|\right)
								\right] \de z\de x,
	\]			
	then the restrictions of the functionals $\E_\eps$ to $X$ $\Gamma$-converge
	w.r.t. the $H^1_\loc(\R)$-convergence to
		\[
			\E_0(u) \coloneqq 
				\begin{cases}
				\begin{displaystyle}
					\frac{1}{24} \left(\int_{\R} G(z)z^2 \de z\right) \int_\R f''(u'(x)) \left|u''(x)\right|^2 \de x
				\end{displaystyle} 		&\text{if } u\in X \cap  H^2_\loc(\R),  \\
				+\infty 						&\text{otherwise}.
				\end{cases} 
		\]			
	\end{cor}

\begin{proof}
	A change of variables gets
		\[
			\E_\eps(u) = \int_{\R} G(z) z^2
				\left[	\frac{1}{(\eps z)^2} \int_{\R}
					f\big( \lvert u'(x) \rvert \big) -  f \left(\left|\frac{u(x+\eps z)-u(x)}{\eps z}\right|\right) \de x
				\right] \de z.
		\]
	By \eqref{eq:Eh},
	the quantity between square brackets is $E_{\eps z}(u')$,
	therefore the conclusion follows by adapting the proof of Theorem \ref{stm:1D-Gconv}
	(see also Proposition \ref{stm:intermediate} below).
\end{proof}

We may henceforth assume that $d\geq 2$.
In this case, the proof of Theorem \ref{stm:Gconv-BBM} is more elaborated,
because, to be able to apply the results of Section \ref{sec:1D},
we need the functions $w_{\hat z,\xi}$ in \eqref{eq:slicing}
to admit a second order weak derivative for a.e. $z$ and $\xi$.
This is no real issue as long as the upper limit inequality is concerned,
since we may reason on regular functions;
to cope with the lower limit one, instead, 
we shall make use of the compactness criterion provided by Lemma \ref{stm:cpt-BBM}.

For the moment being, we can establish a provisional result:

\begin{prop}\label{stm:intermediate}
	Let $u\in X$. Then:
	\begin{enumerate}
		\item If $u\in X\cap H^2_\loc(\Rd)$,
		for any family $\Set{u_\eps}\subset X$
		that converges to $u$ in $H^1_\loc(\Rd)$, there holds
		\begin{equation*}
		\E_0(u)\le \liminf_{\eps \to 0^+} \E_\eps(u_\eps).
		\end{equation*}
		\item If $u \in X\cap C^3(\Rd)$, then
		\begin{equation}\label{eq:pointlim-BBM}
		\E_0(u) = \lim_{\eps \to 0^+} \E_\eps(u).
		\end{equation}
	\end{enumerate}		
\end{prop}

\begin{proof}
	We prove both the assertions by using the slicing formula \eqref{eq:slicing}.
	\begin{enumerate}
		\item For all $\eps>0$, $z\in\Rdmz$, and $\xi \in \hat{z}^\perp$,
		we let $w_{\eps; \hat z,\xi}\colon \R \to \R$ be defined
		as $w_{\eps; \hat z,\xi}(t)\coloneqq u_\eps( \xi + t\hat{z} )$.				
		Then,
		\[
		\E_\eps(u_\eps) = \int_{\Rd}\int_{z^\perp} G(z) \left| z \right|^2  E_{ \eps \left| z \right| } (w'_{\eps ; \hat z,\xi}) \de\Hdmu(\xi)\de z,
		\]
		and, by Fatou's Lemma,
		\begin{equation}\label{eq:liminf-sl}
		\liminf_{\eps \to 0^+} \E_\eps(u_\eps) 
		\geq \int_{\Rd}\int_{z^\perp} G(z) \left| z \right|^2 
		\big[ \liminf_{\eps \to 0^+} E_{ \eps \left| z \right| } (w'_{\eps ; \hat z,\xi}) \big] \de\Hdmu(\xi) \de z.
		\end{equation}
		
		Let $w_{\hat z,\xi}$ be as in Lemma \ref{stm:slicing}.				
		We remark that,
		for any kernel $\rho\colon \Rd \to [0,+\infty)$ such that
		$\lVert \rho \rVert_{L^1(\Rd)} = 1$,
		we have
			\begin{align*}
				& \int_{\Rd} \left| \nabla u_\eps - \nabla u  \right|^2 \\
					 & \qquad \geq \int_{\Rd} \rho( z ) \int_{\hat{z}^\perp} \int_{\R} 
									\left| \big(\nabla u_\eps(\xi + t\hat{z}) - \nabla u(\xi + t\hat{z}) \big) \cdot \hat{z}\right|^2
								\de t \de \Hdmu (\xi) \de z \\
					&  \qquad = \int_{\Rd} \rho( z ) \int_{\hat{z}^\perp} \int_{\R}
									\left| w'_{\eps; \hat z,\xi}(t) - w'_{\hat z,\xi}(t) \right|^2
								\de t \de \Hdmu(\xi) \de z.
			\end{align*}
		The left-hand side vanishes as $\eps \to 0^+$,
		and, therefore, there exists a subsequence of $\Set{w'_{\eps ; \hat z,\xi}}$,
		which we do not relabel,
		that converges to $w'_{\hat z,\xi}$ in $L^2(\R)$
		for a.e. $z\in\Rd$ and $\Hdmu$-a.e. $\xi\in\hat{z}^\perp$.
		
		We remind that, by assumption, $w'_{\hat z,\xi}\in H^1(\R)$ for a.e. $(z,\xi)$
		and it equals $0$ on the complement of some open interval $I_{\hat z,\xi}$.		
		It follows that we can appeal to Proposition \ref{stm:1D-Gliminf},
		obtaining
			\[
				\liminf_{\eps \to 0^+} \E_\eps(u_\eps) 
					\geq \int_{\Rd}\int_{z^\perp} G(z) \left| z \right|^2 E_0 (w'_{\hat z,\xi}) \de\Hdmu(\xi) \de z = \E_0(u).
			\]
		\item As above, for any fixed $z\in\Rd\setminus \Set{0}$ and $\xi\in\hat{z}^\perp$,
		we define the function $w_{\hat z,\xi}\in C^3(\R)$ setting
		$w(t)\coloneqq u(\xi + t\hat{z})$.
		Since $\Omega$ is bounded, there exists $r>0$ such that,
		for any choice of $z$, $w'_{\hat z,\xi}(t) = \nabla u(\xi + t \hat{z}) \cdot \hat{z} = 0$
		whenever $\xi\in z^\perp$ satisfies $\left| \xi \right| \geq r$,
		while $w'_{\hat z,\xi}$ is supported
		in an open interval $I_{\hat z,\xi}$ if $\left| \xi \right| < r$.
		
		We apply the slicing formula \eqref{eq:slicing}, which yields
		\[
		\left| \E_\eps(u) - \E_0(u) \right|
		\leq \int_{\Rd}\int_{z^\perp} G(z) \left| z \right|^2  
		\left| E_{ \eps \left| z \right| } (w'_{\hat z,\xi}) - E_0(w'_{\hat z,\xi})\right|
		\de \Hdmu(\xi) \de z
		\]
		According to Proposition \ref{stm:1D-pointlim}.
		there are a constant $c>0$ and a continuous, bounded, and increasing function
		$m\colon [0,+\infty) \to [0,+\infty)$ such that $m(0)=0$ and that
		\[
		\left| \E_\eps(u) - \E_0(u) \right|
		\leq  c \int_{\Rd}\int_{z^\perp}
		G(z) \left| z \right|^2 m(\eps \left| z \right|) \de \Hdmu(\xi) \de z.
		\]
		We highlight that here $m$ can be chosen 
		as to depend on $\nabla u$ only, and not on $\hat z$ and $\xi$.		
		Thus, by assumption {\bf A1}, we can apply Lebesgue's Convergence Theorem,
		and \eqref{eq:pointlim-BBM} is proved.			\qedhere
	\end{enumerate}
\end{proof}
	
	\begin{rmk}\label{remscala2} 
	Proposition \ref{stm:intermediate} entails that 
	the $\Gamma$-limit of the family $\Set{\eps \E_\eps}$ is $0$;
	the same limit being also found
	w.r.t. the $L^2_\loc(\Rd)$-topology on $X$.
	This is the analogue in arbitrary dimension
	of what we observed in Remark \ref{remscala1}.
	\end{rmk}	

We now deal with the proof of the compactness result,
that is, \ref{stm:cptgen} in Theorem \ref{stm:Gconv-BBM}.
As in Section \ref{sec:1D},
we start with a lower bound on the energy functionals.
Lemma \ref{stm:lbound} below shows that,
when $f$ fulfils {\bf A3},
$\E_\eps(u)$ is greater than a double integral
which takes into account, for each $z\in \Rdmz$,
the squared projection of the difference quotients of $\nabla u$ in the direction of $z$.
Thanks to the slicing formula, the inequality follows with no effort
by applying Lemma \ref{stm:1D-enbound} on each line of $\Rd$.

Our approach results in the appearance of an effective kernel $\tilde G$
in front of the difference quotients.
This function stands as a multidimensional counterpart
of the kernel $H$ in Lemma \ref{stm:1D-enbound}, and,
actually, $\tilde G$ depends both on $G$ and on $H$ 
(see \eqref{eq:tildeK} for the precise definition).
Some properties of the effective kernel are collected in Lemma \ref{stm:tildeK}.

\begin{lemma}[Lower bound on the energy]\label{stm:lbound}
	Let $\Omega$, $G$, and $f$ be as above, and set
		\begin{equation}\label{eq:tildeK}
			\tilde G(z) \coloneqq \int_{-1}^{1} H(r) G_{\left| r \right|}(z) \de r
				\qquad\text{for a.e. } z\in\Rd,				
		\end{equation}
	with $H$ as in Lemma \ref{stm:1D-enbound}.
	Then, it holds
		\begin{equation}\label{eq:lbound}
			\E_\eps(u)\geq 
				\frac{\alpha}{4} \int_{\Rd}\int_{\Rd}
				\tilde G(z) \left[ \frac{\big( \nabla u(x+ \eps z) - \nabla u(x)\big)\cdot\hat{z}}{\eps} \right]^2
				\de x \de z.
		\end{equation}			
	\end{lemma}
	\begin{proof}
	In the following lines, we use $\de'$ as a shorthand for $\de\Hdmu$.
	We reduce to the $1$-dimensional case by slicing, 
	and then we take advantage of the lower bound provided by Lemma \ref{stm:1D-enbound}.
	Keeping the notation of Lemma \ref{stm:slicing}, we find
		\[\begin{split}
			&\E_\eps(u) \\
			&\; \geq \frac{\alpha}{4} \int_{\Rd}\int_{z^\perp}
										\int_{\R}\int_{-\eps\left| z \right|}^{\eps\left| z \right|} H_{\eps \left| z \right|}(r) G(z) \left| z \right|^2 
										\left(\frac{w'_{\hat z,\xi}(t+r) - w'_{\hat z,\xi}(t)}{\eps\left| z \right|}\right)^2
							\de r \de t \de'\xi \de z \\
			&\; = \frac{\alpha}{4} \int_{\Rd}\int_{z^\perp}
									\int_{\R}\int_{-\eps\left| z \right|}^{\eps\left| z \right|} H_{\eps \left| z \right|}(r) G(z) 
									\left(\frac{w'_{\hat z,\xi}(t+r) - w'_{\hat z,\xi}(t)}{\eps}\right)^2
						\de r\de t\de'\xi \de z.
		\end{split}\]
	To retrieve \eqref{eq:lbound} from this bound,
	it suffices to change variables and use Fubini's Theorem:
		\[\begin{split}
			I &\coloneqq \int_{\Rd}\int_{z^\perp} \int_{\R}
					\int_{-\eps\left| z \right|}^{\eps\left| z \right|} H_{\eps \left| z \right|}(r) G(z) 
					\left(\frac{w'_{\hat z,\xi}(t+r) - w'_{\hat z,\xi}(t)}{\eps}\right)^2
					\de r \de t \de'\xi \de z \\
				& =  \int_{\Rd}\int_{z^\perp} \int_{\R} \int_{-1}^{1} H(r) G(z) 
					\left(\frac{w'_{\hat z,\xi}(t + \eps\left|z\right| r) - w'_{\hat z,\xi}(t)}{\eps}\right)^2
					\de r \de t \de'\xi \de z \\
				& =  \int_{-1}^{0} \int_{\Rd}\int_{z^\perp} \int_{\R}
						H(r) G_{-r}( z ) 
						\left(\frac{w'_{\hat z,\xi}(t - \eps\left|z\right|) - w'_{\hat z,\xi}(t)}{\eps}\right)^2
						\de t \de'\xi \de z \de r \\
				&\quad + \int_{0}^{1} \int_{\Rd}\int_{z^\perp} \int_{\R}
							H(r) G_r(z) 
							\left(\frac{w'_{\hat z,\xi}(t + \eps\left|z\right|) - w'_{\hat z,\xi}(t)}{\eps}\right)^2
							\de t \de'\xi \de z \de r
				\end{split}\]			
		Since $w'_{-\hat z,\xi}(-s) = - w'_{\hat z,\xi}(s)$ for all $s\in\R$, we have
		\[\begin{split}
			& \int_{-1}^{0} \int_{\Rd}\int_{z^\perp} \int_{\R}
					H(r) G_{-r}( z )
					\left(\frac{w'_{\hat z,\xi}\big( t + \eps\left|z\right| \big) - w'_{\hat z,\xi}( t )}{\eps}\right)^2
				\de t \de'\xi \de z \de r \\
			&\; = \int_{-1}^{0} \int_{\Rd}\int_{z^\perp} \int_{\R}
					H(r) G_{-r}( z ) 
					\left(\frac{w'_{-\hat z,\xi}\big( -(t + \eps\left|z\right|) \big) - w'_{-\hat z,\xi}( -t )}{\eps}\right)^2
				\de t \de'\xi \de z \de r \\
			&\; = \int_{-1}^{0} \int_{\Rd}\int_{z^\perp} \int_{\R}
					H(r) G_{-r}( z )
					\left(\frac{w'_{\hat z,\xi}\big( t + \eps\left|z\right| \big) - w'_{\hat z,\xi}( t )}{\eps}\right)^2
				\de t \de'\xi \de z \de r.
			\end{split}\]
	Thus, we conclude that			
		\[\begin{split}
			I & = 
				\int_{\Rd} \int_{z^\perp} \int_{\R}\left(\int_{-1}^{1} H(r) G_{\left| r \right|}(z) \de r\right)					
				\left(\frac{w'_{\hat z,\xi}\big(t + \eps\left|z\right|\big) - w'_{\hat z,\xi}(t)}{\eps}\right)^2
				\de t \de'\xi \de z \\
				& = \int_{\Rd}\int_{\Rd}
				\tilde G(z) \left[ \frac{\big( \nabla u(x+ \eps z) - \nabla u(x)\big)\cdot\hat{z}}{\eps}\right]^2
				\de x \de z.
		\end{split}\]
\end{proof}

By assumption {\bf A2},
the kernel $G$ is bounded away from $0$ in a suitable annulus.
The next lemma shows that
the effective kernel $\tilde G$ in \eqref{eq:lbound} inherits a similar property.

\begin{lemma}\label{stm:tildeK}
	Let $\tilde G\colon \Rd \to [0,+\infty)$ be as in \eqref{eq:tildeK}.
	Then,
	\begin{equation}\label{eq:summ-tildeK}
	\int_{\Rd} \tilde G(z)\left( 1 + \left| z \right|^2 \right) \de z < +\infty,
	\end{equation}
	and, if $\beta_d$ and $r_1$ are the constants in \eqref{eq:sigmad} and \eqref{eq:Kpos}, 
	then,
	\begin{equation}\label{eq:tildeKpos}
	\mathrm{ess} \inf \left\{ \tilde G(z) : z \in B(0,\beta_d r_1) \right\} > 0.
	\end{equation}
	\end{lemma}
	\begin{proof}
	The convergence of the integral in \eqref{eq:summ-tildeK} follows easily from \eqref{eq:intG}.
	Indeed, by the definition of $\tilde G$, there holds
		\[
			\int_{\Rd} \tilde G(z) \de z
				= \int_{-1}^{1} \int_{\Rd} H(r) G_{\left| r \right|} (z) \de z \de r
				= \int_{\Rd} G(z) \de z.
		\]
	Analogously,
		\[
		\int_{\Rd} \tilde G(z) \left| z \right|^2 \de z = c \int_{\Rd} G(z) \left| z \right|^2 \de z,
		\]
	for some $c>0$.
	
	For what concerns \eqref{eq:tildeKpos},
	let us set $\gamma\coloneqq \mathrm{ess} \inf \Set{ G(z) : z \in B(0,r_1)\setminus B(0,r_0) }$.
	In view of \eqref{eq:Kpos}, $\gamma>0$.
	
	We distinguish between the case $ z\in B(0,r_0) $ and the case $ z\in B(0,r_1)\setminus B(0,r_0) $.
	In the first situation, for a.e. $z\in\Rd$,
		\[\begin{split}
			\tilde G(z) & \geq 2 \int_{\frac{ \left| z \right| }{r_1}}^{\frac{ \left| z \right| }{r_0}} H(r) G_r(z) \de r
				\geq  2\gamma \int_{\frac{ \left| z \right| }{r_1}}^{\frac{ \left| z \right| }{r_0}} \frac{1}{r^d}H(r) \de r \\
			& = \frac{2\gamma}{ \left| z \right|^{d-1} }
				\int_{r_0 }^{r_1} s^{d-2} \left( 1 - \frac{\left| z \right|}{s}  \right) \de s.
			\end{split}\]
	When $ z\in B(0,r_1)\setminus B(0,r_0) $, instead, similar computations get
			\[
			\tilde G(z)	\geq 2 \int_{\frac{ \left| z \right| }{r_1}}^{1} H(r) G_r(z) \de r
				= \frac{2\gamma}{ \left| z \right|^{d-1} }
					\int_{\left| z \right| }^{r_1} s^{d-2} \left( 1 - \frac{\left| z \right|}{s}  \right) \de s
			\qquad \text{for a.e. } z\in\Rd,
		\]
	so that we obtain
	\begin{equation}\label{eq:est-tildeK}
	\tilde G(z) 
	\geq  \frac{2\gamma}{\left| z \right|^{d-1} }
	\int_{\mathrm{max}(r_0,\left| z \right|) }^{r_1} s^{d-2} \left( 1 - \frac{\left| z \right|}{s}  \right) \de s
	\qquad \text{for a.e. } z\in\Rd.
	\end{equation}
	
	When $d=2$, the previous estimate becomes
	\[
	\tilde G(z) \geq 2\gamma
	\left[ \frac{r_1 - \mathrm{max}(r_0,\left| z \right|) }{\left| z \right| }
	-\log \left( \frac{r_1}{\mathrm{max}(r_0,\left| z \right|)}\right)
	\right]
	\qquad \text{for a.e. } z\in\Rd.
	\]
	By the concavity of the logarithm, we see that
	the lower bound is strictly positive
	if $ \left| z \right| < r_1 = \beta_2 r_1$.
	
	In the case $d\geq 3$, instead,
	the right-hand side in \eqref{eq:est-tildeK} equals
	\[
	\frac{2\gamma}{(d-1)(d-2) \left| z \right|^{d-1}}
	\left[ 
	(d-2) \left( r_1^{d-1} - M^{d-1} \right)
	- 
	(d-1)\left| z \right|\left(r_1^{d-2} - M^{d-2}\right)
	\right],
	\]
	where we put $M\coloneqq \mathrm{max}(r_0,\left| z \right|)$ for shortness.
	Hence,
	\begin{equation*} \label{eq:tildeK>0}\begin{split}
	\tilde G(z) \geq & \frac{2\gamma M^{d-2}}{(d-1)(d-2)\left| z \right|^{d-1}} \\
	& \cdot \left\{
	\left( \frac{r_1}{M}\right)^{d-2}
	\left[ (d-2)r_1 - (d - 1)\left| z \right| \right]
	- 								 
	\left[ (d-2)M - (d - 1)\left| z \right| \right]
	\right\} 
	\end{split}\end{equation*}
	for a.e. $z\in\Rd$. When $\left| z \right| < (d-2)r_1/(d-1) = \beta_d r_1$,
	the quantity between braces is strictly positive if
	\[
	\frac{(M - \left| z \right| ) d - (2M - \left| z \right| )}{(r_1 - \left| z \right| ) d - (2r_1 - \left| z \right| )}
	< \left( \frac{r_1}{M} \right)^{d-2}.
	\]
	Observe that both sides are strictly increasing in $d$, and
	that the left-hand one is bounded above by $(M - \left| z \right| ) / (r_1 - \left| z \right| )$.
	In conclusion, the last inequality holds if
	\[
	\frac{M - \left| z \right|}{r_1 - \left| z \right|} < \frac{r_1}{M},
	\]
	which, in turn, is true for all $z\in B(0,r_1)$.	
	\end{proof}

\begin{rmk}
	For our analysis, it is not relevant to identify the largest ball
	in which $\mathrm{ess}\inf \tilde G$ is strictly positive.
	For this reason, we are not interested
	in the optimal $\beta_d$ in \eqref{eq:tildeKpos} for $d \geq 3$.
\end{rmk}

We can finally prove compactness for families with equibouded energies.

\begin{lemma}[Compactness] \label{stm:cpt-BBM}
	Assume that $\Omega$, $G$, and $f$ are as above.
	Then, if $\Set{u_\eps}\subset X$ satisfies $\E_\eps(u_\eps)\leq M $ for some $M\geq 0$,
	there exist a subsequence $\Set{\eps_\ell}$ and
	a function $u\in H^2(\Rd)$ vanishing outside $\Omega$
	such that $\nabla u_{\eps_\ell}\to \nabla u$ in $L^2(\Rd)$.
\end{lemma}	
\begin{proof}
	Let $\tilde \gamma \coloneqq \mathrm{ess} \inf \{\, \tilde G(z) : z \in B(0,\beta_d r_1) \,\}$.
	Thanks to Lemma \ref{stm:tildeK}, we know that $\tilde \gamma >0$,
	and therefore there exists a function $\rho \in \Cc^\infty([0,+\infty))$ such that
		\[
			\rho(r) = 0 \qquad \text{if } r \in \left[ \frac{\beta_d r_1}{\sqrt{2}},+\infty \right)
		\]
	and that
		\[
			0\leq \rho(r) \leq \tilde \gamma
				\quad\text{and}\quad
			\left|\rho'(r)\right| \leq \tilde \gamma.
		\]	
	For $\eps>0$ and $y\in \Rd$, we set 
	\[
	\rho_\eps(y)\coloneqq \frac{1}{c\eps^d}\rho\left(\frac{\left| y \right|}{\eps}\right),
	\qquad \text{with } c\coloneqq \int_{\Rd} \rho(\left| y \right|)\de y,
	\]
	and we introduce the functions
		\[
			v_\eps(x) \coloneqq \big(\rho_{\eps}\ast (u_\eps - \sum_{i\in I} a_{\eps,i}\chi_{C_i})\big)(x),
		\]
	where $\Set{C_i}_{i\in I}$ is the family of connected components of $\co{\Omega}$
	and each $a_{\eps,i} \in\R$ satisfies $u_\eps(x) = a_{\eps,i}$ in $C_i$
	for all $i\in I$ and $\epsilon>0$.
	
	Each function $v_\eps$ is smooth and, for all $\delta \in(0,1)$,
	its support is contained in
	$\Omega_{\delta} \coloneqq \Set{x : \dist(x,\Omega)\leq 2^{-1/2}\delta \beta_d r_1 }$
	if $\eps\in(0,\delta)$.
	
	Let us fix $\delta$ small enough, so that
	$\partial \Omega_{\delta}$ is Lipschitz.
	For such $\delta$, we now prove that
	the family $\Set{ v_\eps }_{\eps\in(0,\delta)}$ is relatively compact in $H^1_0(\Omega_{\delta})$.
	To this aim, we first remark that
		\begin{equation}\label{eq:hess-lap}
			\int_{\Omega_{\delta}} \left| \nabla^2 v_\eps \right|^2
			= \int_{\Omega_{\delta}} \left| \Delta v_\eps \right|^2,
		\end{equation}
	and next we show that the right-hand side is uniformly bounded. 
	
	We observe that $\int_{\Rd} \nabla \rho_{\eps}(y)\de y = 0$ for all $\eps>0$,
	because $\rho$ is compactly supported. Hence,
	\[
	\begin{split}
	\lVert \Delta v_\eps \rVert^2_{ L^2( \Omega_{\delta} ) }
	= & \int_{\Rd} \left| \Delta v_\eps \right|^2 \\
	= &	\int_{\Rd} \left| \int_{\Rd}
	\nabla \rho_{\eps}(y) \cdot \big( \nabla u_\eps(x+y) -  \nabla u_\eps(x) \big)
	\de y \right|^2 \de x \\
	\leq & \int_{\Rd} \left[\frac{1}{c \eps^{d+1}}\int_{\Rd}
	\left| \rho' \left( \frac{\left| y \right|}{\eps}\right) \right|
	\left|\big( \nabla u_\eps(x+y) -  \nabla u_\eps(x) \big) \cdot \hat{y} \right| 
	\de y \right] ^2 \de x.
	\end{split}\]
	By our choice of $\rho$ and \eqref{eq:tildeKpos}, we find
	\[\begin{split}
	\lVert \Delta v_\eps \rVert^2_{ L^2( \Omega_{\delta} ) }
	\leq & \int_{\Rd} \left[ \frac{1}{ c\eps } \int_{\Rd}
	\tilde G_\eps ( y )
	\left| \big( \nabla u_\eps( x+y) -  \nabla u_\eps(x) \big) \cdot \hat{y} \right|
	\de y \right] ^2 \de x \\
	\leq & \int_{\Rd} \left[ \frac{1}{ c\eps } \int_{\Rd}
	\tilde G( z )
	\left| \big( \nabla u_\eps( x + \eps z ) -  \nabla u_\eps(x) \big) \cdot \hat{z} \right|
	\de z \right] ^2 \de x		
	\end{split}\]
	Further, since $\tilde G\in L^1(\Rd)$, Jensen's Inequality and Fubini's Theorem yield
	\[
	\lVert \Delta v_\eps \rVert^2_{ L^2( \Omega_{\delta} ) }
	\leq \frac{\lVert \tilde G \rVert_{L^1(\Rd)}}{c^2}
	\int_{\Rd} \int_{\Rd} \tilde G(z)
	\left[ \frac{\big( \nabla u_\eps( x+ \eps z) -  \nabla u_\eps(x) \big) \cdot \hat{z}}{\eps}\right]^2
	\de x \de z.
	\]				
	The lower bound \eqref{eq:lbound} entails
	\[
	\lVert \Delta v_\eps \rVert^2_{ L^2( \Omega_{\delta} ) }
	\leq \frac{4}{c^2 \alpha} \lVert \tilde G \rVert_{L^1(\Rd)} \E_\eps(u_\eps),
	\]
	so that, in view of the assumption $\E_\eps(u_\eps)\leq M$ and of \eqref{eq:hess-lap}, we get
	\begin{equation}\label{eq:nabla2}
	\lVert \nabla^2 v_\eps \rVert^2_{ L^2( \Omega_{\delta} ) }
	\leq \frac{4 M}{c^2 \alpha} \lVert \tilde G \rVert_{L^1(\Rd)} .
	\end{equation}
	
	We argue as in the proof of Lemma \ref{stm:1D-cpt}.
	We recall that,  for $\eps\in(0,\delta)$,
	each $v_\eps$ vanishes on the complement of $\Omega_{\delta}$,
	and thus, by Poincar\'e's Inequality,
	\eqref{eq:nabla2} implies a uniform bound on the norms
	$\lVert v_\eps \rVert_{H^2_0(\Omega_{\delta})}$.
	As a consequence, by Rellich-Kondrachov's Theorem,
	the family $\Set{\tilde v_\eps}_{\eps\in(0,\delta)}$
	of the restrictions of the functions $v_\eps$ to $\Omega_{\delta}$
	admits a subsequence $\Set{\tilde v_{\eps_\ell}}$ that converges in $H^1_0(\Omega_{\delta})$
	to a function $\tilde u \in H^2_0(\Omega_{\delta})$.
	Actually, the support of $\tilde u$ is contained in $\bar\Omega$, and,
	if we put,	
	\[
	u(x)\coloneqq\begin{cases}
	\tilde u(x) 	&\text{if } x\in \bar\Omega, \\
	0					& \text{otherwise},
	\end{cases}
	\]
	we infer that $\Set{v_{\eps_\ell}}$ converges in $H^1(\Rd)$
	to $u \in \Set{v\in H^2(\Rd) : v = 0 \text{ a.e. in } \co{\Omega}} \subset X$.
	
	To accomplish the proof, we need to show that
	the $L^2(\Rd)$-distance between $\nabla u_\eps$ and $\nabla v_\eps$ vanishes when $\eps \to 0^+$.
	Since $\rho_\eps$ has unit $L^1(\Rd)$-norm and is radial, we have
		\[\begin{split}
			\int_{\Rd} \lvert \nabla v_\eps(x) & - \nabla u_\eps(x)\rvert^2 \de x 
				= \int_{\Rd} \left| \int_{\Rd} 
					\rho_{\eps}(y) \big( \nabla u_\eps(x+y) -  \nabla u_\eps(x) \big)\de y
					\right|^2 \de x \\
			& \leq \frac{1}{4}\int_{\Rd} \left| \int_{\Rd}
					\rho_{\eps}(y)  \big( \nabla u_\eps(x+y) + \nabla u_\eps(x-y) -  2\nabla u_\eps(x) \big)\de y
					\right|^2 \de x \\
			& \leq \frac{1}{4}\int_{\Rd} \int_{\Rd}
					\rho_{\eps}(y) \left| \nabla u_\eps(x+y) + \nabla u_\eps(x-y) -  2\nabla u_\eps(x) \right|^2 \de y  \de x.
		\end{split}\]
	Let $\id$ be the identity matrix.
	For any fixed $y\in\Rdmz$ and for all $p\in\Rd$, we can rewrite the equality
	$ \left| p \right|^2 = \left| p\cdot y \right|^2 + \left| (\id - y\otimes y) p \right|^2 $
	as
	\begin{equation}\label{eq:proj}\begin{split}
	\left| p \right|^2 	= & \left| p\cdot y \right|^2
	+ \int_{\hat{y}^\perp}
	\pi( \left| \eta \right| )\left| p\cdot \eta\right|^2 \de'\eta \\
	= &  \left| p\cdot y \right|^2
	+ \frac{1}{\eps^2} \int_{\hat{y}^\perp}
	\pi_\eps(\eta)\left| p\cdot \eta\right|^2 	\de'\eta,
	\end{split}\end{equation}
	where $\pi\colon [0,+\infty) \to [0,+\infty)$ is a continuous function such that
	\[
	\int_{e_d^\perp} \pi( \left| \eta \right| ) \left| \eta \right|^2 \de'\eta = 1,
	\]
	and $\pi_\eps(\eta)\coloneqq \eps^{-d+1} \pi( \left| \eta \right| / \eps)$.
	We further prescribe that
	\begin{equation*}
	\pi(r) = 0 \qquad \text{if } r\in \left[ \frac{\beta_d r_1}{\sqrt{2}}, +\infty \right)
	\end{equation*}
	and that the limit $\lim_{r \to 0^+} \pi(r) / r$ is finite.
	
	We apply formula \eqref{eq:proj}
	to $p_\eps(x,y) \coloneqq \nabla u_\eps(x+y) + \nabla u_\eps(x-y) -  2\nabla u_\eps(x)$
	and we find that
	\begin{equation}\label{eq:nablas}
	\int_{\Rd} \left| \nabla v_\eps(x) - \nabla u_\eps(x)\right|^2 \de x  \leq \frac{1}{4}\left( I_1 + I_2 \right),
	\end{equation}
	where
	\[
	\begin{gathered}
	I_1 \coloneqq \int_{\Rd} \int_{\Rd} 
	\rho_{\eps}(y) \left| y \right|^2 
	\left| p_\eps(x,y) \cdot \hat{y} \right|^2 \de y  \de x, \\
	I_2 \coloneqq \frac{1}{\eps^2}\int_{\Rd} \int_{\Rd} \int_{\hat{y}^\perp} \rho_{\eps}(y) \pi_\eps(\eta)
	\left| p_\eps(x,y)\cdot \eta \right|^2 \de'\eta \de y  \de x .
	\end{gathered}
	\]
	We first consider $I_1$. 
	Keeping in mind that $\rho$ is compactly supported and
	$\rho(\left| y \right|) \leq \tilde \gamma \leq \tilde G(y)$ for a.e. $y\in B(0, 2^{-1/2}\beta_d r_1)$,
	we get
	\[
	\begin{split}
	I_1 \leq  \frac{(\beta_d r_1)^2}{c} 
	& \left[
	\int_{\Rd} \int_{\Rd}
	\tilde G_{\eps}(y) \left| \big( \nabla u_\eps(x+y) - \nabla u_\eps(x) \big) \cdot \hat{y} \right|^2 \de y \de x
	\right. \\
	& \left. + \int_{\Rd} \int_{\Rd}
	\tilde G_{\eps}(y) \left| \big( \nabla u_\eps(x-y)  - \nabla u_\eps(x) \big)\cdot \hat{y} \right|^2 \de y  \de x
	\right],
	\end{split}
	\]
	and, by \eqref{eq:lbound},
	\begin{equation}\label{eq:I1}
	I_1 \leq \frac{ 8 (\beta_d r_1)^2 M}{c \alpha} \eps^2.
	\end{equation}
	
	As for $I_2$, we claim that there exist a constant $L>0$,
	depending on $d$, $\beta_d$, $r_1$, $\tilde \gamma$, and $c$, such that
	\begin{equation}\label{eq:I2}
	I_2 \leq \frac{ L M}{\alpha} \eps^2.
	\end{equation}
	To prove this, we write the integrand appearing in $I_2$ as follows:
	\[
	\begin{split}
	p_\eps(x,y)\cdot \eta = & \big( \nabla u_\eps(x+y) + \nabla u_\eps(x-y) -  2\nabla u_\eps(x) \big)\cdot \eta \\
	= & \big( \nabla u_\eps(x+y) + \nabla u_\eps(x-y) - 2 \nabla u_\eps(x-\eta) \big)\cdot \eta \\
	& + 2 \big( \nabla u_\eps(x-\eta) - \nabla u_\eps(x) \big)\cdot \eta \\
	= & \big( \nabla u_\eps(x+y) - \nabla u_\eps(x-\eta) \big)\cdot ( \eta + y ) \\
	& + \big( \nabla u_\eps(x-y) - \nabla u_\eps(x-\eta) \big)\cdot ( \eta - y ) \\
	& - \big( \nabla u_\eps(x+y) - \nabla u_\eps(x-y) \big)\cdot y
	+ 2 \big( \nabla u_\eps(x-\eta) - \nabla u_\eps(x) \big)\cdot \eta .
	\end{split}
	\]
	We plug this expression in the definition of $I_2$.
	Using again the symbol $\de'$ as a short for $\de\Hdmu$,
	we find
	\[
	\begin{split}
	I_2 & \leq \frac{4}{c} \int_{\Rd} \int_{\Rd} \int_{\hat{y}^\perp} \rho(\left| y \right|) \pi( \left| \eta \right| )
	\left| 
	\big(\nabla u_\eps(x+\eps y) - \nabla u_\eps(x-\eps\eta) \big)\cdot ( \eta + y )
	\right|^2 \de'\eta \de y  \de x  \\
	&\quad + \frac{4}{c}  \int_{\Rd} \int_{\Rd} \int_{\hat{y}^\perp} \rho(\left| y \right|) \pi(\left| \eta \right|)
	\left|
	\big(\nabla u_\eps(x-\eps y) - \nabla u_\eps(x-\eps\eta) \big)\cdot ( \eta - y )
	\right|^2 \de'\eta \de y  \de x \\
	& \quad + \frac{8}{c}  \lVert \pi \rVert_{L^1(e_d^\perp)}
	\int_{\Rd} \int_{\Rd} \rho(\left| y \right|) \left| y \right|^2
	\left|
	\big(\nabla u_\eps(x+\eps y) - \nabla u_\eps(x) \big)\cdot \hat y
	\right|^2 \de y  \de x \\
	&\quad + \frac{16}{c}  \int_{\Rd} \int_{\Rd} \int_{\hat{y}^\perp}
	\rho(\left| y \right|) \pi( \left| \eta \right| )
	\left|
	\big(\nabla u_\eps(x+\eps\eta) - \nabla u_\eps(x) \big)\cdot \eta
	\right|^2 \de'\eta \de y \de x.
	\end{split}
	\]
	We estimate separately each of the contributions on the right-hand side.
	
	Let us set
	$\Sdmu_+ \coloneqq \Set{ e \in \Sdmu : e \cdot e_d > 0}$ and
	$\Sdmu_- \coloneqq \Set{ e \in \Sdmu : e \cdot e_d < 0}$.
	Hereafter, we denote by $L$ any strictly positive constant
	depending only on $d$, $\beta_d$, $r_1$, and on the norms of $\rho$ and $\pi$,
	possibly changing from line to line.
	
	By Coarea Formula, we can rewrite the first addendum as follows:
		\[\begin{split}
			& \int_{\Rd} \int_{\Rd} \int_{\hat{y}^\perp} \rho(\left| y \right|) \pi( \left| \eta \right|) \left| 
					\big(\nabla u_\eps(x+\eps y) - \nabla u_\eps(x-\eps\eta) \big)\cdot ( \eta + y )
					\right|^2 \de'\eta \de y  \de x \\
			&\quad = \int_{\Rd}\int_{\Rd} \int_{\hat{y}^\perp} \rho( \left| y \right|) \pi( \left| \eta \right|) \left|
					\big( \nabla u_\eps(x+\eps(\eta+y)) - \nabla u_\eps(x) \big)\cdot ( \eta + y )
					\right|^2 \de'\eta \de x  \de y \\
			&\quad= \int_{\Sdmu_+} \int_{\Rd} \int_{\R} \int_{e^\perp}
					r^{d-1}\rho( r ) \pi( \left| \eta \right| ) \\
					&\qquad\qquad\qquad\qquad\qquad \cdot \left|
						\big( \nabla u_\eps(x+\eps(\eta + r e)) - \nabla u_\eps(x) \big)\cdot ( \eta + r e )
					\right|^2 \de'\eta \de r \de x  \de'e \\
			&\quad= \int_{\Sdmu_+} \int_{\Rd} \int_{\Rd}
					\left| y \right|^2 \left| y\cdot e \right|^{d-1}\rho( \left| y\cdot e \right| )
					\pi\big( \left| (\id - e\otimes e) y \right| \big) \\
					&\qquad\qquad\qquad\qquad\qquad
						\cdot \left| \big( \nabla u_\eps(x+\eps y) - \nabla u_\eps(x) \big)\cdot \hat y \right|^2
					\de y \de x  \de'e.
		\end{split}\]
	Similarly, we have
	\[\begin{split}
		& \int_{\Rd} \int_{\Rd} \int_{\hat{y}^\perp} \rho(\left| y \right|) \pi( \left| \eta \right|)
					\left|
						\big(\nabla u_\eps(x-\eps y) - \nabla u_\eps(x-\eps\eta) \big)\cdot ( \eta - y )
					\right|^2 \de'\eta \de y  \de x
		\\ &\quad = \int_{\Sdmu_-} \int_{\Rd} \int_{\Rd}
					\left| y \right|^2 \left| y\cdot e \right|^{d-1}\rho( \left| y\cdot e \right| )
					\pi\big( \left|(\id - e\otimes e) y \right| \big) \\		
		& \qquad\qquad\qquad\qquad\qquad
				\cdot \left| \big( \nabla u_\eps(x+\eps y) - \nabla u_\eps(x) \big)\cdot \hat y \right|^2
				\de y \de x  \de'e,
		\end{split}\]
	and thus
	\[\begin{split}
	& \int_{\Rd}\int_{\Rd} \int_{\hat{y}^\perp} \rho(\left| y \right|) \pi( \left| \eta \right|)
	\left| 
	\big(\nabla u_\eps(x+\eps y) - \nabla u_\eps(x-\eps\eta) \big)\cdot ( \eta + y )
	\right|^2 \de'\eta \de y  \de x \\
	&\quad + \int_{\Rd} \int_{\Rd} \int_{ \hat{y}^\perp } \rho(\left| y \right|) \pi( \left| \eta \right|)
	\left|
	\big(\nabla u_\eps(x-\eps y) - \nabla u_\eps(x-\eps\eta) \big)\cdot ( \eta - y )
	\right|^2 \de'\eta \de y  \de x \\
	&\; =  \int_{ \Sdmu} \int_{\Rd} \int_{\Rd}
	\left| y\cdot e \right|^{d-1} \left| y \right|^2 \rho( \left| y\cdot e \right| )
	\pi\big( \left|(\id - e\otimes e) y \right|\big) \\
	& \qquad\qquad\qquad\qquad\qquad \cdot \left| \big( \nabla u_\eps(x+\eps y) - \nabla u_\eps(x) \big)\cdot \hat y \right|^2
	\de y \de x \de'e.
	\end{split}\]
	Let us recall that $\rho(r) = \pi(r) = 0$ if $r \notin [0,2^{-1/2}\beta_d r_1)$,
	whence, for any $e\in \Sdmu$,
	the product $\rho( \left| y\cdot e \right| )\pi\big( \left|(\id - e\otimes e) y \right|\big)$
	vanishes outside the cylinder
	\[
	C_e \coloneqq \Set{ y\in\Rd : 
		\left| y\cdot e \right|, \left|(\id - e\otimes e) y \right| 
		\in [0,2^{-1/2}\beta_d r_1)
	} \subset B(0,\beta_d r_1).
	\]
	Consequently, the last multiple integral equals
		\[\begin{split}
		& \int_{ \Sdmu} \int_{\Rd} \int_{C_e}
		\left| y\cdot e \right|^{d-1} \left| y \right|^2 \rho( \left| y\cdot e \right| )
		\pi\big( \left|(\id - e\otimes e) y \right|\big) \\
		&\qquad\qquad\qquad\qquad\qquad \cdot \left| \big( \nabla u_\eps(x+\eps y) - \nabla u_\eps(x) \big)\cdot \hat y \right|^2
		\de y \de x \de'e,
		\end{split}\]
	which, in turn, is bounded above by
		\[
			L \int_{ \Sdmu} \int_{\Rd} \int_{C_e}
		\tilde G(y) \left| \big( \nabla u_\eps(x+\eps y) - \nabla u_\eps(x) \big)\cdot \hat y \right|^2
		\de y \de x \de'e
		\leq \frac{L M}{\alpha} \eps^2 .
		\]
	We then obtain
	\begin{multline}\label{eq:1}
	\frac{4}{c} \int_{\Rd} \int_{\Rd} \int_{\hat{y}^\perp} \rho(\left| y \right|) \pi( \left| \eta \right| )
	\left| 
	\big(\nabla u_\eps(x+\eps y) - \nabla u_\eps(x-\eps\eta) \big)\cdot ( \eta + y )
	\right|^2 \de'\eta \de y  \de x  \\
	+ \frac{4}{c}  \int_{\Rd} \int_{\Rd} \int_{\hat{y}^\perp} \rho(\left| y \right|) \pi(\left| \eta \right|)
	\left|
	\big(\nabla u_\eps(x-\eps y) - \nabla u_\eps(x-\eps\eta) \big)\cdot ( \eta - y )
	\right|^2 \de'\eta \de y  \de x \\
	\leq  \frac{ LM }{ \alpha} \eps^2.
	\end{multline}
	Next, we have
	\begin{equation}\label{eq:2}
	\frac{8}{c}  \lVert \pi \rVert_{L^1(e_d^\perp)}
	\int_{\Rd} \int_{\Rd} \rho(\left| y \right|) \left| y \right|^2
	\left|	\big(\nabla u_\eps(x+\eps y) - \nabla u_\eps(x) \big)\cdot \hat y \right|^2
	\de y  \de x 
	\leq \frac{L M}{\alpha} \eps^2,
	\end{equation}
	\begin{equation}\label{eq:3}
	\frac{16}{c}  \int_{\Rd} \int_{\Rd} \int_{\hat{y}^\perp}
	\rho(\left| y \right|) \pi( \left| \eta \right| )
	\left| \big(\nabla u_\eps(x+\eps\eta) - \nabla u_\eps(x) \big)\cdot \eta \right|^2
	\de'\eta \de y \de x
	\leq \frac{L M}{\alpha} \eps^2. 			
	\end{equation}
	The bound in \eqref{eq:2} may be deduced as the one in \eqref{eq:I1},
	so, to establish \eqref{eq:I2}, we are only left to prove \eqref{eq:3}.
	To this purpose, let $\psi\in \Cc^\infty(\Rd\times\Rd)$ be a test function.
	By a standard argument and Fubini's Theorem we have that
	\[\begin{split}
	\int_{\Rd} \int_{\hat{y}^\perp}
	& \rho(\left| y \right|) \pi( \left| \eta \right| ) \psi(y,\eta)  \de'\eta \de y
	\\ = & \lim_{\delta \to 0^+} 
	\int_{\Rd}\int_{\Rd} \frac{\left| y \right|}{2\delta} \chi_{\Set{ t < \delta}}(\left| \eta \cdot y\right|)
	\rho(\left| y \right|) \pi( \left| \eta \right| ) \psi(y,\eta)  d\eta \de y
	\\ = & \lim_{\delta \to 0^+}
	\int_{\Rd} \frac{\pi( \left| \eta \right| )}{\left| \eta \right| }
	\left(\int_{\Rd} \frac{\left| \eta \right| }{2\delta} \chi_{\Set{ t < \delta}}(\left| \eta \cdot y\right|)
	\rho(\left| y \right|) \left| y \right|  \psi(y,\eta)  \de y
	\right) d\eta
	\\ = &  \int_{\Rd} \int_{\hat{\eta}^\perp}
	\frac{\pi( \left| \eta \right| )}{\left| \eta \right| }
	\rho(\left| y \right|) \left| y \right| \psi(y,\eta) 
	\de'y d\eta.
	\end{split}\]
	It follows that
	\[\begin{split}
	\int_{\Rd} \int_{\Rd} & \int_{\hat{y}^\perp}
	\rho(\left| y \right|) \pi( \left| \eta \right| )
	\left| \big(\nabla u_\eps(x+\eps\eta) - \nabla u_\eps(x) \big)\cdot \eta \right|^2
	\de'\eta \de y \de x
	\\ = & \int_{\Rd} \int_{\Rd} \int_{\hat{\eta}^\perp}
	\frac{\pi( \left| \eta \right| )}{\left| \eta \right| }
	\rho(\left| y \right|) \left| y \right| 
	\left| \big(\nabla u_\eps(x+\eps\eta) - \nabla u_\eps(x) \big)\cdot \eta \right|^2
	\de'y d\eta \de x
	\\ \leq & L \int_{\Rd} \int_{\Rd}
	\tilde G(\eta)\left| \big(\nabla u_\eps(x+\eps\eta) - \nabla u_\eps(x) \big)\cdot \eta \right|^2
	d\eta \de x
	\end{split}\] 
	(recall that we assume $\lim_{r\to 0^+} \pi(r) / r$ to be finite).
	
	In conclusion, combining \eqref{eq:nablas}, \eqref{eq:I1}, and \eqref{eq:I2},
	we infer
		\[
			\int_{\Rd} \left| \nabla v_\eps(x) - \nabla u_\eps(x)\right|^2 \de x  \leq \frac{L M}{\alpha} \eps^2.
		\]
	This concludes the proof.
\end{proof}

Eventually, we observe that
Theorem \ref{stm:Gconv-BBM} follows
from the results of the current section.

	\begin{proof}[Proof of Theorem \ref{stm:Gconv-BBM}]
	Lemma \ref{stm:cpt-BBM} shows that statement \ref{stm:cptgen} holds.
	
	As for the lower limit inequality, 
	for any $u\in X$ and for any $\Set{u_\eps} \subset X$
	that converges to $u$ in $H^1_\loc(\Rd)$,
	we can assume that $\E_\eps(u_\eps) \leq M$ for all $\eps>0$ and some $M\geq 0$;
	otherwise, the lower limit inequality holds trivially.
	Then, by Lemma \ref{stm:cpt-BBM},
	$u$ belongs to $H^2_\loc(\Rd)$ and we can invoke Proposition \ref{stm:intermediate},
	which yields statement \ref{stm:Gliminf}.	
		
	We prove the upper limit inequality
	in the same fashion as the $1$-dimensional case
	(see the proof of Proposition \ref{stm:1D-pointlim}).
	Let us suppose that $u\in X \cap H^2_\loc(\Rd)$.
	By mollification
	we can obtain a sequence $\Set{u_\ell} \subset X$ of smooth functions
	that converges to $u$ in $H^2_\loc(\Rd)$. More precisely,
	$\Set{\nabla u_\ell}$ and $\Set{\nabla^2 u_\ell}$ converge
	in $L^2(\Rd)$  respectively to $\nabla u$ and $\nabla^2 u$.
	If the gradient of $u$ is bounded in $\Omega$ or $f''$ is bounded,
	$f''(\left|\nabla u_\ell(x) \cdot \hat{z}\right|)\leq M$
	for a.e. $x$, all $z$ and some $M>0$.
	It follows that we can exploit the Dominated Convergence Theorem
	to get $\lim_{\ell \to +\infty} \E_0(u_\ell) = \E_0(u)$.
	So, statements \ref{stm:Glimsup-a} and \ref{stm:Glimsup-b} are achieved
	through the approximation by smooth functions,
	Proposition \ref{stm:intermediate}, and Lemma \ref{stm:prop-Gconv}.
	\end{proof}

	\begin{rmk}[A regularity criterion] \label{stm:carattH2}
		When $\Omega$, $K$, and $f$ are as in Theorem~\ref{stm:Gconv-BBM}
		and $f''$ is bounded,
		as a consequence of our results, we see that
		a function $u\in X$ belongs to $H^2_\loc(\Rd)$
		if and only if $\E_\eps(u)\leq M$ for some $M>0$ and for all sufficiently small $\eps$.
		Indeed, on one hand, if $\E_\eps(u)\leq M$
		we can invoke Lemma \ref{stm:cpt-BBM} with the choice $u_\eps=u$.
		On the other, by slicing and Remark \ref{stm:f''bounded},
		when $f''\leq c$ we get
		\[
		\E_\eps(u) \leq
		\frac{c}{2} \left(\int_{\Rd} K(z) \left| z \right|^2 dz\right)
		\int_{\Rd} \left| \nabla^2 u(x)\right|^2 \de x.
		\]
	\end{rmk}
		
\chapter{Nonlocal perimeters and~nonlocal~curvatures}\label{ch:nlpnlc}
		In the first chapter, we revised the key points of the well-established theory
of finite perimeter sets in $\Rd$.
Now, we go beyond it and we focus on its nonlocal extension.
The main sources of our exposition are the papers \cites{BP,Pa}.

Given a reference set $\Omega\subset \Rd$ and
a measurable function $K\colon \Rd \to [0,+\infty)$,
for us the \emph{nonlocal perimeter} associated with $K$
of a set $E\subset \Rd$ in $\Omega$ is
	\begin{equation*}\label{eq:PerK}
		\begin{split}
		\PerK(E,\Omega) \coloneqq& \int_{E\cap\Omega}\int_{E^c\cap\Omega}K(y-x)\de y\de x \\
		& + \int_{E\cap\Omega}\int_{E^c \cap \Omega^c}K(y-x)\de y\de x +\int_{E\cap\Omega^c}\int_{E^c\cap\Omega}K(y-x)\de y\de x
		\end{split}
	\end{equation*}
(recall that, when $F$ is a set, $\co{F}\coloneqq \Rd \setminus F$).
Functionals of this form were introduced in \cite{CRS}
by L. Caffarelli, J. Roquejoffre, and O. Savin,
who focused on the fractional case, that is, $K(x) = \va{x}^{-d-s}$ with $s\in(0,1)$.
Their analysis was motivated by phase field models where long range interactions occur.
We refer to the Introduction for a brief account of the reasons to study this sort of functionals, 
which, by now, have been intensively investigated.
We shall suggest the works that are more closely related to ours
among the ones in the vast available literature in the body of the chapter.

In analogy with the classical case,
it is natural to consider the nonlocal counterparts of two other functionals,
namely the total variation and the mean curvature.
We define them in Section \ref{sec:ca-pl} and Section \ref{sec:nlc}, respectively.
Ahead of this, to get acquainted with nonlocal perimeters,
Section \ref{sec:elem-nlp} provides an overview of their basic properties:
semicontinuity, submodularity, relation with the De Giorgi's perimeter.
In Section \ref{sec:ca-pl} we also deal with existence of solutions to Plateau's problem.
In this, an important role is played by a generalised Coarea Formula,
whose simple proof we recall.
Then, in Section \ref{sec:calib}, we propose a notion of calibration
that is tailored for the nonlocal Plateau's problem,
and we prove that calibrated functions are optimal w.r.t. their own boundary condition
(see the recent paper \cite{C} by X. Cabr\'e for similar results, too).
A case in which we are able to produce an explicit calibration
is the one of halfspaces, which turn out to be also the unique minimisers,
see Theorem \ref{stm:piani}.
Finally, moving from sufficient to necessary conditions for optimality,
in Section \ref{sec:nlc} we establish an upper bound
on the nonlocal curvature of sets with $C^{1,1}$ boundary.
	
	\section{Elementary properties of~nonlocal~perimeters}\label{sec:elem-nlp}
		In a na\"\i ve manner, we might say that
a functional named perimeter should measure the extension
of the locus that separates a set from its complement.
In the case of nonlocal perimeters,
this is achieved by means of a measurable function $K\colon \Rd \to [0,+\infty)$,
which we imagine to express some interaction between the points in $\Rd$.
Then, we define the \emph{nonlocal perimeter}\index{perimeter!nonlocal --!in $\Rd$}
associated with $K$ of a set $E\subset \Rd$ as
	\begin{equation}\label{eq:PerK-Rd}
		\PerK(E) \coloneqq \int_{E}\int_{\co{E}} K(y-x) \de y \de x.
	\end{equation}
The intuitive idea behind such a position is that
the interaction between the points $x\in E$ and $y\in\co{E}$ must ``cross'' the boundary of $E$,
so the size of the latter can by quantified by the iterated integral at stake.

We would also like to give a notion
of nonlocal perimeter restricted to a reference set $\Omega$,
in analogy to \eqref{eq:dfnPer}.
To this aim, we firstly introduce for any measurable $E,F$ the coupling
	\[
		L_K(E;F)\coloneqq\int_E\int_F K(y-x)\de y\de x,
	\]
and we discuss some of its properties.

By Tonelli's Theorem,
	\[
		L_K(E;F) = L_K(F;E) = \int_{E\times F} K(y-x)\de y \de x,
	\]
because $K$ is positive.
It follows that 
we may assume without loss of generality that $K$ is even, i.e.
	\begin{equation}\label{eq:Keven}
		K(x)=K(-x) \quad\text{for a.e. } x\in\Rd.
	\end{equation}
We also observe that when $E_1,E_2,F$ are measurable 
	\begin{gather*}
		L_K(E_1;F) = L_K(E_2;F) \quad\mbox{if } \Ld(E_1\symdif E_2)=0, \\
		L_K(E_1\cup E_2;F) = L_K(E_1;F)+L_K(E_2;F) \quad\mbox{if } \Ld(E_1\cap E_2)=0,
	\end{gather*}
and
	\[
		\PerK(E) = L_K(E;\co{E}).
	\]	

Whether the interaction $L_K(E;F)$ is finite,
it depends on the ``regularity'' of $K$, $E$, and $F$.
We consider various possibilities:
	\begin{enumerate}
		\item If $K$ is in $L^1(\Rd)$ and its support is contained in $B(0,r)$ for some $r>0$,
			then the interaction is determined by the points
			that are at a distance smaller than $r$ from $E$ and $F$:
				\begin{equation*}
				L_K(E;F) = \int_{\set{x\in E : \dist(x,F)<r}} \int_{\set{y\in F : \dist(y,E)<r}} K(y-x)\de y \de x.
				\end{equation*}
			In \cite{MRT}, J. Maz\'on, J. Rossi, and J. Toledo studied
			nonlocal perimeters and curvatures defined by kernels in this class.
		\item When $K\in L^1(\Rd)$,
			$L_K(E;F)$ is finite as soon as one between $E$ and $F$ has finite measure.
			Indeed,
				\[
					L_K(E;F) = \int_{\Rd}K(z)\va{E \cap (F-z)} \de z 
									\leq \norm{K}_{L^1(\Rd)}\big( \Ld(E) \wedge \Ld(F)\big).
				\]
		\item As a third case, we allow $K$ to be singular in the origin.
			Precisely, let us assume that
				\begin{equation} \label{eq:summK}
				\int_{\Rd} K(x)(1 \wedge \va{x}) \de x<+\infty.
				\end{equation}
			Then, we can bound the coupling $L_K$ from above,
			provided some information about the mutual position of the sets is available.
			Indeed, on one hand, because of the singularity of $K$,
			$L_K$ might blow up if the sets overlap on a region of full measure.
			On the other hand, if $\Ld( E \cap F) = 0$,
			then $F \subset \co{E}$ up to negligible sets, so that
				\[
					L_K(E;F) \leq L_K(E;\co{E}) = \PerK(E).
				\]
			We shall prove in Lemma \ref{stm:CMP} that
			$\PerK(E)$ is bounded above by the $\BV$-norm of $\chi_E$,
			up to a constant that depends on $K$.
			
			Observe that, by \eqref{eq:summK}, $K \in L^1(\co{B(0,r)})$
			for all balls $B(0,r)$ with centre in the origin and radius $r>0$.
	\end{enumerate}

We gather here some examples of kernels such that \eqref{eq:summK} holds.

\begin{exa}
	Functions in $L^1$ satisfy \eqref{eq:summK} trivially.
	Instead, a relevant example of kernels
	that exhibit a singularity in $0$ is given by the fractional ones
	(\cites{CRS,L}).
	We say that $K$ is of \emph{fractional}\index{fractional kernel} type if
		\[
			K(x)=\frac{a(x)}{\va{x}^{d+s}},
		\]
	with $s\in(0,1)$ and $a\colon \Rd \to \R$ a measurable even function
	such that $0<m\leq a(x)\leq M$ for any $x\in\Rd$ for some positive $m$ and $M$.
	
	A third class is formed by the kernels $K$ such that $\int_{\Rd} K(x)\va{x} \de x <+\infty$.
	This requirement allows for a fractional-type behaviour near the origin,
	but it implies faster-than-$L^1$ decay at infinity at the same time. 
	We shall deal with functions of this type in the next chapter, see Theorem \ref{stm:JK-Gconv}.
	\end{exa}

Now, we use the coupling $L_K$
to define the nonlocal perimeter restricted to a reference measurable set $\Omega$,
which we always assume to have strictly positive measure.

Again, we consider the interaction between the given set $E$ and its complement.
An initial attempt might be setting
$\PerK(E;\Omega)$ equal to $L_K( E \cap \Omega; \co{E} \cap \Omega )$,
but in this way we would be neglecting contributions
arising from the portions of boundary that $E$ and $\Omega$ share.
For this reason, we take into consideration also terms where $\co{\Omega}$ appears,
and we define the \emph{nonlocal perimeter}\index{perimeter!nonlocal --!in a reference set}
of a measurable set $E$ in $\Omega$ as
	\begin{equation}\label{eq:PerK-Om}
		\begin{split}
			\PerK(E;\Omega) & \coloneqq L_K(E\cap\Omega;\co{E}\cap\Omega) \\
										& \quad 	+ L_K(E\cap\Omega;\co{E} \cap \co{\Omega}) + L_K(E\cap\co{\Omega};\co{E}\cap\Omega).
		\end{split}
	\end{equation}
An illustration of what the different terms encode is provided by Figure \ref{fig:PerK}.

A simple but useful fact is that $\PerK(E;\Omega)$ may be rewritten
in terms of the characteristic function of $E$: 
\begin{equation}\label{eq:semicont}
\begin{split}
\PerK(E;\Omega) & = \frac{1}{2}\int_\Omega\int_\Omega K(y-x)\va{\chi_E(y)-\chi_E(x)}\de y\de x \\
&\quad + \int_{\Omega}\int_{\co{\Omega}}K(y-x)\va{\chi_E(y)-\chi_E(x)}\de y \de x.
\end{split}
\end{equation}
We also note that
the nonlocal perimeter in \eqref{eq:PerK-Om} coincides with $\PerK(\,\cdot\,;\Rd)$
and that $\PerK(E;\Omega)=\PerK(E)$ when $\Ld(E\cap\co{\Omega})=0$.
		
		\begin{figure}
		\caption{The three contributions that compound $\PerK$.
						The thick lines represent the portions of boundary crossed by the interaction.}
		\label{fig:PerK}		
		\centering
		\subfloat[][$L_K(E\cap\Omega;\co{E}\cap\Omega)$]
		{\begin{tikzpicture}[scale=0.80,rotate=25]
			\fill[lightgray] (-1.5,1) .. controls (-1,1.5) and  (0,2) .. (0.5,2)
			.. controls (1,2) and (3,3) .. (3.5,2.5)
			.. controls (4,2) and (3,2) .. (4,1.5)
			.. controls (4.5,1) and (3,-1) .. (2.5,0)
			.. controls (2,1) and (2.5,0.5)  .. (3,1)
			.. controls (3.2,1.2) and (3,1.3) .. (2,1.3)
			.. controls (1.8,1.3) and (1.5,-1) .. (1,-1)
			.. controls (0.7,-1) and (-0.2,-0.7) .. (0,-0.5)
			.. controls (0.5,0) and (-2,0.5) .. (-1.5,1);						
			\draw[thick] (3,1) .. controls (3.2,1.2) and (3,1.3) .. (2,1.3)
			.. controls (1.8,1.3) and (1.5,-1) .. (1,-1)
			.. controls (0.7,-1) and (-0.2,-0.7) .. (0,-0.5);			
			\path (1,1) node{$E$};			
			\path (2.3,-1) node{$\Omega$};
			\draw[dashed] (-1.5,1)  .. controls (-1,1.5) and (-1,2) .. (0,2.5);
			\draw[dashed] (0,2.5) .. controls (1,3) and (1,3) .. (1.5,3);
			\draw[dashed] (1.5,3) .. controls (2.25,3) and (3.5,1.5) .. (3,1);
			\draw[dashed] (3,1) .. controls (2.5,0.5) and (2,1) .. (2.5,0);
			\draw[dashed] (2.5,0) .. controls (3,-1) and (6,-0.5) .. (4,-1.5);
			\draw[dashed] (4,-1.5).. controls (2,-2.5) and (-0.5,-1) .. (0,-0.5);
			\draw[dashed] (0,-0.5) .. controls (0.5,0) and (-2,0.5) .. (-1.5,1);						
			\end{tikzpicture}
		}
		
		\subfloat[][$L_K(E\cap\Omega;\co{E}\cap\co{\Omega})$]
		{\begin{tikzpicture}[scale=0.80,rotate=25]
			\fill[lightgray] (-1.5,1) .. controls (-1,1.5) and  (0,2) .. (0.5,2)
			.. controls (1,2) and (3,3) .. (3.5,2.5)
			.. controls (4,2) and (3,2) .. (4,1.5)
			.. controls (4.5,1) and (3,-1) .. (2.5,0)
			.. controls (2,1) and (2.5,0.5)  .. (3,1)
			.. controls (3.2,1.2) and (3,1.3) .. (2,1.3)
			.. controls (1.8,1.3) and (1.5,-1) .. (1,-1)
			.. controls (0.7,-1) and (-0.2,-0.7) .. (0,-0.5)
			.. controls (0.5,0) and (-2,0.5) .. (-1.5,1);
			\path (1,1) node{$E$};
			\path (2.3,-1) node{$\Omega$};
			\draw[dashed] (-1.5,1)  .. controls (-1,1.5) and (-1,2) .. (0,2.5);
			\draw[dashed] (0,2.5) .. controls (1,3) and (1,3) .. (1.5,3);
			\draw[dashed] (1.5,3) .. controls (2.25,3) and (3.5,1.5) .. (3,1);
			\draw[dashed] (3,1) .. controls (2.5,0.5) and (2,1) .. (2.5,0);
			\draw[dashed] (2.5,0) .. controls (3,-1) and (6,-0.5) .. (4,-1.5);
			\draw[dashed] (4,-1.5).. controls (2,-2.5) and (-0.5,-1) .. (0,-0.5);
			\draw[thick] (0,-0.5) .. controls (0.5,0) and (-2,0.5) .. (-1.5,1);		
			\end{tikzpicture}
		}	
		\subfloat[][$L_K(E\cap\co{\Omega};\co{E}\cap\Omega)$]
		{\begin{tikzpicture}[scale=0.85,rotate=25]
				\fill[lightgray] (-1.5,1) .. controls (-1,1.5) and  (0,2) .. (0.5,2)
				.. controls (1,2) and (3,3) .. (3.5,2.5)
				.. controls (4,2) and (3,2) .. (4,1.5)
				.. controls (4.5,1) and (3,-1) .. (2.5,0)
				.. controls (2,1) and (2.5,0.5)  .. (3,1)
				.. controls (3.2,1.2) and (3,1.3) .. (2,1.3)
				.. controls (1.8,1.3) and (1.5,-1) .. (1,-1)
				.. controls (0.7,-1) and (-0.2,-0.7) .. (0,-0.5)
				.. controls (0.5,0) and (-2,0.5) .. (-1.5,1);
				\path (1,1) node{$E$};
				\path (2.3,-1) node{$\Omega$};
				\draw[dashed] (-1.5,1)  .. controls (-1,1.5) and (-1,2) .. (0,2.5);
				\draw[dashed] (0,2.5) .. controls (1,3) and (1,3) .. (1.5,3);
				\draw[dashed] (1.5,3) .. controls (2.25,3) and (3.5,1.5) .. (3,1);
				\draw[thick] (3,1) .. controls (2.5,0.5) and (2,1) .. (2.5,0);
				\draw[dashed] (2.5,0) .. controls (3,-1) and (6,-0.5) .. (4,-1.5);
				\draw[dashed] (4,-1.5).. controls (2,-2.5) and (-0.5,-1) .. (0,-0.5);
				\draw[dashed] (0,-0.5) .. controls (0.5,0) and (-2,0.5) .. (-1.5,1);		
				\end{tikzpicture}
			}
	\end{figure}
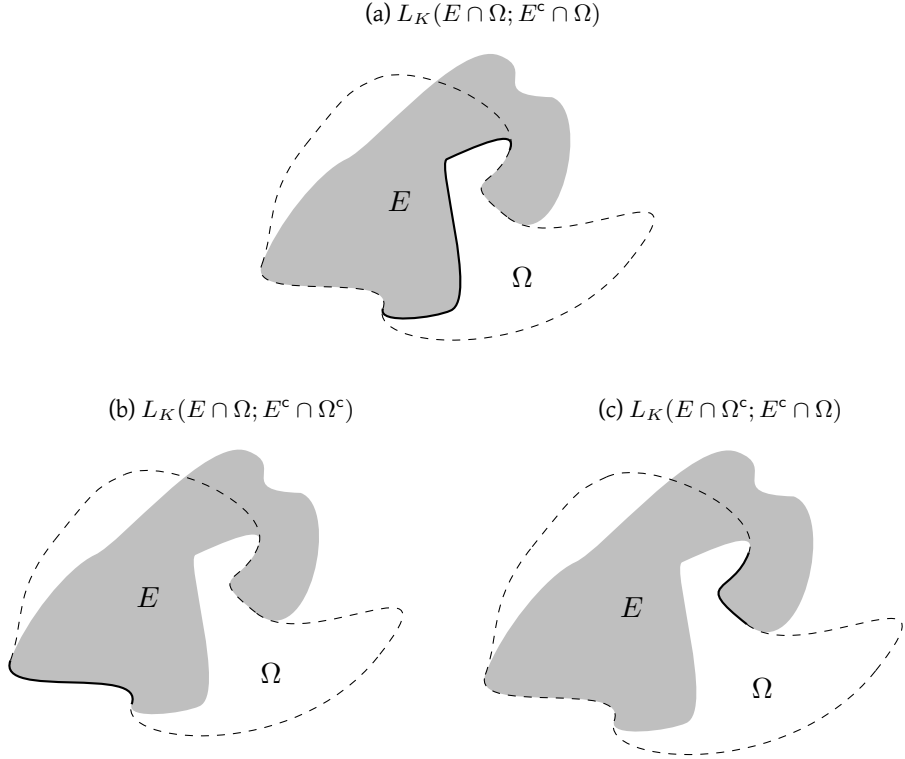

We collect some basic properties of the nonlocal perimeter in the following lemma:
	\begin{lemma}\label{stm:CMP}
		Let $\Omega\subset \Rd$ be an open set.
		If $K\colon\Rd \to [0,+\infty)$ fulfils \eqref{eq:Keven} and \eqref{eq:summK},
		then the functional $\PerK(\,\cdot\,;\Omega)$ defined by \eqref{eq:PerK-Om} satisfies the following:
		\begin{enumerate}
			\item\label{stm:per1} $\PerK(\emptyset;\Omega) = 0$,
					$\PerK(E;\Omega) = \PerK(F;\Omega)$ if $\Ld(E \symdif F) = 0$,
					and $\PerK(E+z;\Omega+z) = \PerK(E;\Omega)$ for all $z\in\Rd$.
			\item\label{stm:per2} Let $E$ be a set such that $\Ld(E)<+\infty$.
					If $\Omega\neq \Rd$ and there exists $r>0$ such that
					$E$ is a finite perimeter set in $\Omega_r\coloneqq \set{x \in\Rd: \dist(x,\Omega)<r}$,
					then 
							\begin{equation}\label{eq:PerK-leq-Per}
								\begin{split}
								\PerK(E;\Omega)
									& \leq \left( \Ld(E) \vee \frac{\Per(E;\Omega_r)}{2} \right) \int_{\Rd} K(x)(1\wedge \va{x}) \de x \\
									& \quad	+ \frac{\Ld(E)}{2 r} \int_{B(0,1)} K(x)\va{x}\de x;
								\end{split}
							\end{equation}
					in particular, $E$ has finite nonlocal $K$-perimeter in $\Omega$.					
					Also, if $E$ is a finite perimeter set in $\Rd$,
					then 
							\begin{equation*}
							\PerK(E) \leq \left( \Ld(E) \vee \frac{\Per(E)}{2} \right) \int_{\Rd} K(x)(1\wedge \va{x}) \de x.
							\end{equation*}
			\item\label{stm:per3} $\PerK(\,\cdot\,;\Omega)$ is lower semicontinuous w.r.t. the $L^1_\loc(\Rd)$-convergence.
			\item\label{stm:per4} $\PerK(\,\cdot\,;\Omega)$ is submodular, i.e. 
						for all $E,F\in\M$ it holds
						\[
							\PerK(E\cap F;\Omega) + \PerK(E\cup F;\Omega) \leq \PerK(E;\Omega) + \PerK(F;\Omega).
						\]
		\end{enumerate}
	\end{lemma}
	\begin{proof}
		Statement \ref{stm:per1} follows easily from the definition of perimeter
		thanks to the properties of Lebesgue's integral.
		
		To prove \ref{stm:per2}, we take advantage of formula \eqref{eq:semicont}:
			\[\begin{split}
				L_K( E\cap\Omega;\co{E}\cap\Omega)
					& = \frac{1}{2}\int_\Omega\int_\Omega K(y-x)\va{\chi_E(y)-\chi_E(x)}\de y\de x \\
					& = \frac{1}{2}\int_\Omega\int_{\Rd} K(z)\va{\chi_E(x+z)-\chi_E(x)} \chi_\Omega(x+z)\de z \de x \\
					& = \frac{1}{2}\int_{\Rd} \int_{\Omega \cap (\Omega - z)}
													K(z) \va{\chi_E(x+z)-\chi_E(x)}\de x \de z \\
					& = \frac{1}{2}\int_{B(0,1)} \int_{\Omega \cap (\Omega - z)}
												K(z)\va{\chi_E(x+z)-\chi_E(x)}\de x \de z \\
					&	\quad +\frac{1}{2}\int_{\co{B(0,1)}} \int_{\Omega \cap (\Omega - z)}
												K(z)\va{\chi_E(x+z)-\chi_E(x)}\de x \de z
			\end{split}			
			\]
		
		The last double integral may be easily bounded by the triangular inequality:
			\[
				\frac{1}{2}\int_{\co{B(0,1)}} \int_{\Omega \cap (\Omega - z)} K(z)\va{\chi_E(x+z)-\chi_E(x)}\de x \de z
					\leq \Ld(E\cap\Omega) \int_{\co{B(0,1)}} K(z)\de z.
			\]
		Moreover, for any $z\in B(0,1)$ and $x\in\Omega\cap (\Omega - z)$,
		we may appeal to Proposition \ref{stm:char-BV},
		which gets
			\[
				\frac{1}{2}\int_{B(0,1)} \int_{\Omega \cap (\Omega - z)}
						K(z)\va{\chi_E(x+z)-\chi_E(x)}\de x \de z
					\leq \frac{\Per(E;\Omega)}{2} \int_{B(0,1)} K(z)\va{z}\de z.
			\]
		On the whole,
			\begin{equation}\label{eq:OmOm0}
				L_K( E\cap\Omega;\co{E}\cap\Omega) \leq
					\left( \Ld(E\cap\Omega) \vee \frac{\Per(E;\Omega)}{2} \right) \int_{\Rd} K(z)(1\wedge \va{z}) \de z;
			\end{equation}
		this concludes the proof when $\Omega=\Rd$.
		
		We focus now on the case when $\Omega$ is a proper subset of $\Rd$.
		Without loss of generality, we may assume that
		$E$ is a finite perimeter set in $\Omega_r$ for $r\in(0,1]$.
		By reasoning as above, we obtain
			\[\begin{split}
				L_K( E\cap\Omega;\co{E}\cap\co{\Omega}) & + L_K( E\cap\co{\Omega};\co{E}\cap \Omega) \\
					& = \int_\Omega\int_{\co{\Omega}} K(y-x)\va{\chi_E(y)-\chi_E(x)}\de y\de x \\
					& = \frac{1}{2}\int_{B(0,r)} \int_{\Omega \cap (\co{\Omega} - z)}
								K(z)\va{\chi_E(x+z)-\chi_E(x)}\de x \de z \\
					&	\quad +\frac{1}{2}\int_{\co{B(0,r)}} \int_{\Omega \cap (\co{\Omega} - z)}
								K(z)\va{\chi_E(x+z)-\chi_E(x)}\de x \de z.
			\end{split}\]
		Again, we are in position to apply Proposition \ref{stm:char-BV} and the triangular inequality.
		We get
			\[\begin{split}
				L_K( E\cap\Omega;\co{E}\cap\co{\Omega}) & + L_K( E\cap\co{\Omega};\co{E}\cap \Omega) \\
					& \leq \frac{\Per(E;\Omega_r)}{2} \int_{B(0,r)} K(z)\va{z}\de z
						+\frac{\Ld(E)}{2}\int_{\co{B(0,r)}} K(z)\de z \\
					& \leq \frac{\Per(E;\Omega_r)}{2} \int_{B(0,1)} K(z)\va{z}\de z \\
					& \quad +\frac{\Ld(E)}{2}
									\left( \frac{1}{r}\int_{B(0,1)} K(z)\va{z}\de z
											+ \int_{\co{B(0,1)}} K(z)\de z
									\right).
			\end{split}\]
		We now recover \eqref{eq:PerK-leq-Per}
		by combing \eqref{eq:OmOm0} and the last estimate.
		
		For what concerns assertion \ref{stm:per3}, the lower semicontinuity is a consequence
		of formula \eqref{eq:semicont} and of Fatou's Lemma.
		
		Eventually, we prove statement \ref{stm:per4}.
		It suffices to rewrite suitably each contribution:
		for instance, we have
			\[\begin{split}
				L_K( (E\cup F)\cap \Omega & ; \co{E} \cap \co{F}\cap\Omega) \\
					& = L_K(E\cap\Omega ; \co{E} \cap\Omega)
							+ L_K(F\cap\Omega ; \co{F} \cap\Omega) \\
					& \quad - L_K(E\cap\Omega ; \co{E} \cap F\cap\Omega)
								- L_K(F\cap\Omega ; E\cap \co{F} \cap\Omega) \\
					& \quad -L_K(E\cap F\cap \Omega; \co{E}\cap \co{F}\cap \Omega)
		\end{split}\]
		and
			\[\begin{split}
				L_K(E\cap F\cap \Omega & ;(\co{E} \cup \co{F})\cap\Omega) \\
					& = L_K(E\cap F\cap \Omega;\co{E}\cap \co{F}\cap \Omega) \\
					& \quad +L_K(E\cap F\cap \Omega;\co{E} \cap F\cap\Omega)
								+L_K(E\cap F\cap \Omega;E\cap \co{F} \cap\Omega).
		\end{split}\] 
		In the end, one finds
			\[\begin{split}
				\PerK(E;\Omega)+\PerK(F;\Omega)
					& = \PerK(E\cap F ; \Omega)+\PerK(E\cup F ; \Omega)   \\
					& \quad +2 L_K(E\cap \co{F}\cap \Omega ; \co{E}\cap F\cap \Omega) \\
					& \quad	+2L_K(E\cap \co{F}\cap \Omega ; \co{E}\cap F\cap \co{\Omega}) \\
					& \quad	+2L_K(E\cap \co{F}\cap \co{\Omega} ; \co{E}\cap F\cap \Omega).
		\end{split}\]
		
	\end{proof}

	\begin{rmk}[Generalised perimeters\index{perimeter!generalised --}]
		Definition \eqref{eq:PerK-Om} might seem very different from the one of De Giorgi's perimeter.
		Nevertheless, Lemma \ref{stm:CMP} shows
		that the classical and the nonlocal perimeter have some useful properties in common.
		Actually, in \cite{CMP}
		A. Chambolle, M. Morini, and M. Ponsiglione proposed that
		a set functional $p$ is a \emph{generalised perimeter} if
			\begin{enumerate}
				\item $p(\emptyset)=0$, $p(E)=p(F)$ whenever $\Ld(E\symdif F)=0$,
					and $p$ is invariant under translations;
				\item it is finite on the closures of open sets with compact $C^2$ boundary;
				\item it is $L^1_\mathrm{loc}(\Rd)$-lower semicontinuous;
				\item it is submodular.
			\end{enumerate}
		
		In \cite{CMP},
		the functional in \eqref{eq:PerK-Rd} appears as an instance of generalised perimeter.
		In that work, the authors are concerned with geometric evolutions driven by generalised curvatures,
		and $\PerK$ is listed as an example of functional
		whose first variation is a curvature in that sense.
		We shall come back to this point later on,
		see Section \ref{sec:nlc} and Chapter \ref{ch:nlc}.
	\end{rmk}

	\section{Coarea Formula and Plateau's problem}\label{sec:ca-pl}
		We recalled in Section \ref{sec:aniso}
that the theory of De Giorgi's perimeter may be formulated
in terms of functions of bounded variation.
In the same spirit, we introduce a nonlocal functional
that might be regarded as a \emph{nonlocal total variation}\index{total variation!nonlocal --}.
Standing the previous assumptions on $\Omega$ and $K$,
for any measurable $u\colon \Rd \to \R$ we set
	\begin{gather}
		J_K^1(u;\Omega) \coloneqq \frac{1}{2}\int_\Omega\int_\Omega K(y-x)\va{u(y)-u(x)}\de x\de y, \nonumber\\
		J_K^2(u;\Omega) \coloneqq \int_{\Omega}\int_{\Omega^c}K(y-x)\va{u(y)-u(x)}\de x \de y, \\ \label{eq:JK}
		J_K(u;\Omega) \coloneqq J_K^1(u;\Omega) + J_K^2(u;\Omega) \nonumber.
	\end{gather}
By a small abuse of notation,
we shall write $J^i_K(E;\Omega) \coloneqq J^i_K(\chi_E;\Omega)$ for $i=1,2$
and $J_K(E;\Omega) \coloneqq J_K(\chi_E;\Omega)$,
so that \eqref{eq:semicont} reads
	\[
		\PerK(E;\Omega) = J_K^1(E;\Omega) + J_K^2(E;\Omega)
									= J_K(E;\Omega) .
	\]

By imitating the proof of statement \ref{stm:per2} in Lemma \ref{stm:CMP},
the reader may verify that
the $\BV$-norm bounds the nonlocal total variation:
	\begin{prop}
		Let $\Omega$ be an open set
		and let $K\colon \Rd \to [0,+\infty)$ satisfy \eqref{eq:Keven} and \eqref{eq:summK}.
		Suppose also that $u\in L^1(\Rd)$.
		Then, if $\Omega\neq \Rd$ and there exists $r>0$ such that
		$u\in\BV(\Omega_r)$, with $\Omega_r\coloneqq \Set{x \in\Rd: \dist(x,\Omega)<r}$,
			\begin{equation*}\label{eq:JK-leq-BV}
			\begin{split}
			J_K(u;\Omega)
					& \leq \left( \norm{u}_{L_1(\Rd)} \vee \frac{\va{\D u}(\Omega_r)}{2} \right)
							\int_{\Rd} K(x)(1\wedge \va{x}) \de x \\
					& \quad	+ \frac{\norm{u}_{L_1(\Rd)}}{2 r} \int_{B(0,1)} K(x)\va{x}\de x.
			\end{split}
			\end{equation*}
		Moreover, if $u\in\BV(\Rd)$, then 
			\begin{equation*}
				J_K(u) \leq \left( \norm{u}_{L_1(\Rd)} \vee \frac{\norm{\D u}}{2} \right) \int_{\Rd} K(x)(1\wedge \va{x}) \de x.
			\end{equation*}
	\end{prop}

In the remainder of this Section,
we investigate some of the main features of the nonlocal total variation.
We firstly state two fundamental properties of $J_K(\,\cdot\,;\Omega)$
whose validation is straightforward.
	
	\begin{prop}\label{stm:JK}
		The nonlocal total variation $J_K(\,\cdot\,;\Omega)$ is convex
		and lower semicontinuous w.r.t. the $L^1_\loc(\Rd)$-convergence.
	\end{prop}

Another tool that will be crucial for the forthcoming analysis is a generalised version of the Coarea Formula.
Following A. Visintin \cites{Vi,Vi2}, we say that
a functional $J$ defined on measurable scalar functions fulfils the
\emph{generalised Coarea Formula}\index{Coarea Formula! generalised --} if
	\begin{equation}\label{eq:gencoarea}
	J(u)=\int_{-\infty}^{+\infty}J(\chi_{\Set{u>t}})\de t.
	\end{equation}
Hereafter, if $u\colon \Rd \to \R$ is a function and $t\in\R$, we set
	\[
		\Set{u>t} \coloneqq \Set{ x\in\Omega : u(x)>t}.
	\]
	
It is not difficult to show that
the nonlocal total variation $J_K$ and the $K$-perimeter satisfy a coarea-type equality.
	\begin{prop}[Nonlocal Coarea Formula]
	\label{stm:JK-coarea}\index{Coarea Formula!nonlocal --}
	Let $K\colon \Rd \to [0,+\infty)$ and $u\colon \Rd \to \R$ be measurable.
	It holds 
		\[
			J_K^i(u;\Omega) = \int_{-\infty}^{+\infty} J_K^i(\Set{u>t};\Omega)\de t \quad\text{for } i=1,2,
		\]
	thus
		\[
			J_K(u;\Omega) = \int_{-\infty}^{+\infty}\PerK(\Set{u>t};\Omega)\de t.
		\]
	\end{prop}
The proof has already appeared several times in the literature
(see \cite{CSV} or
\cites{Vi,CN,MRT,ADM}).		
	\begin{proof}
	Given $x,y\in\Omega$, we may suppose that $u(x)\leq u(y)$.
	We observe that the function $t\mapsto\chi_{\Set{u>t}}(x)-\chi_{\Set{u>t}}(y)$ vanishes
	for $t\in (-\infty,u(x)) \cup (u(y),+\infty)$,
	therefore
		\[\begin{split}
			\int_{\R} \va{\chi_{\Set{u>t}}(y)-\chi_{\Set{u>t}}(x)}\de t
			& = \int_{u(x)}^{u(y)} \va{\chi_{\Set{u>t}}(y)-\chi_{\Set{u>t}}(x)}\de t \\
			& = \va{u(y)-u(x)}.
		\end{split}\]
	By Tonelli's Theorem,
		\[\begin{split}
			J^1_K(u;\Omega) & \coloneqq \frac{1}{2} 	\int_{\Omega}\int_{\Omega}K(x-y)\va{u(x)-u(y)}\de x\de y \\
				&= \frac{1}{2}  \int_{\R}\int_{\Omega}\int_{\Omega} K(x-y)\va{\chi_{\Set{u>t}}(x)-\chi_{\Set{u>t}}(y)}\de x\de y\de t \\
				& = \int_{\R} J_K^1(\Set{u>t};\Omega)\de t
		\end{split}\]
	The equality involving $J_K^2(u;\Omega)$ may be proved in the same manner.
	\end{proof}

The validity of the generalised Coarea Formula has interesting consequences from a variational perspective.
For instance, in \cite{CGL}, A. Chambolle, A. Giacomini, and L. Lussardi proved that
a proper $L^1$-lower semicontinuous functional
that satisfies the generalised Coarea Formula
and whose restriction to characteristic functions is submodular
is necessarily convex.
Thus, we might see the convexity of $J_K(\,\cdot\,;\Omega)$ as an example of this general principle.
Another useful contribution in \cite{CGL}
that involves the Coarea Formula will be invoked in Section \ref{sec:strategy},
when we deal with the asymptotics of rescaled energies.

Proposition \ref{stm:JK-coarea} comes in handy also when
addressing a quite natural variational problem concerning $J_K$.
Its classical analogue is the well-known Plateau's problem,
which amounts to minimise the perimeter under prescribed boundary conditions.
If $\Omega$ is a given references set and
$E_0$ is a measurable set playing the role of the boundary datum,
one usually considers the class of competitors
$\F\coloneqq \Set{E :  E \cap \co{\Omega} = E_0 \cap \co{\Omega}}$.
Remember that, in our notation, equality holds up to negligible sets.
Existence of minimisers may be easily achieved
thanks to the compactness criterion Theorem \ref{stm:BVcpt},
and the $L^1_\loc$-semicontinuity of De Giorgi's perimeter.

The approach does not require modifications
if the classical perimeter is replaced by an anisotropic surface energy
such as $\Per_\sigma$ in \eqref{eq:aPer}, provided the anisotropy $\sigma$ is a norm.
Conversely, the very same strategy is not effective in the nonlocal framework,
because sequences
that have uniformly bounded nonlocal perimeter
are not necessarily precompact in $L^1$.
As an example of this phenomenon,
we may consider a sequence $\Set{E_\ell}$
whose elements are subsets of some $\Omega$ with finite measure,
and we observe that, for any $K \in L^1(\Omega)$, it holds
$\PerK(E_\ell;\Omega)\leq3\norm{K}_{L^1(\Rd)}\Ld(\Omega)$.

\begin{rmk}\label{stm:coercivity}
	It would be natural
	to investigate under which general conditions on $K$
	coercivity of $\PerK(\,\cdot\,,\Omega)$ is ensured.
	In a recent paper \cite{CS}, J. Chaker and L. Silvestre provided an answer
	for a functionals of the form
		\[
		\int_{\Rd}\int_{\Rd}K(x,y)\va{u(y) - u(x)}^2 \de x \de y.
		\]
\end{rmk}

Even though, in general, a uniform bound on the nonlocal perimeter is not informative,
the direct method of calculus of variations is still a viable option.

\begin{thm}[Existence of solutions to Plateau's problem]\label{stm:plateau}
	Let $K\colon \Rd \to [0,+\infty)$ be measurable
	and let $\Omega\subset \Rd$ be an open set with finite measure.
	We suppose that $E_0$ is a measurable set such that $\PerK(E_0; \Omega)<+\infty$,
	and we define the family
	\begin{equation}\label{eq:F}
	\call{F}\coloneqq\Set{v\colon \Rd \to [0,1] :
		v \text{ is measurable and }
		v = \chi_{E_0} \text{ in } \co{\Omega} }.
	\end{equation}
	Then, there exists $E\subset \Rd$ such that $\chi_E\in\call{F}$ and
	\[
	\PerK(E;\Omega) \leq J_K(v;\Omega) \quad\text{for any } v\in\call{F}.
	\]
\end{thm}
We premise some comments about the statement
before dealing with its proof.

\begin{rmk}[Truncation]
	For $s\in\R$, let us set $T(s) \coloneqq 0 \vee  (s \wedge 1)$.
	Observe that $T \circ \chi_{E_0} = \chi_{E_0}$ and $J_K( T \circ u; \Omega ) \leq J_K( u; \Omega )$.
	Therefore,
	\[
	\inf\Set{ J_K(v;\Omega) : v\in\call{F} }
	=
	\inf\Set{J_K(v;\Omega) :
		v\colon \Rd \to \R \text{ is measurable and }
		v = \chi_{E_0} \text{ in } \co{\Omega},
	}.
	\]
	and we see that choice of $\call{F}$ as the class of competitors is not restrictive.
\end{rmk}

\begin{rmk}[The class of competitors is nonempty]
	Let us suppose that $E_0$ has finite perimeter in some open set $\Omega_0$
	that strictly  contains the closure of $\Omega$ and that $K$ fulfils \eqref{eq:summK}.
	Then, by Lemma \ref{stm:CMP}--\ref{stm:per2},
	$\PerK(E_0; \Omega ) <+\infty$.
\end{rmk}

\begin{proof}[Proof of Theorem \ref{stm:plateau}]
	Let $\Set{u_\ell}\subset \call{F}$ be a minimising sequence.
	Since $\Omega$ has finite measure,
	for any choice of $p\in(1,+\infty)$,
	$\Set{u_\ell}$ is bounded in $L^p(\Omega;[0,1])$.
	It  follows that there exist a subsequence $\Set{u_{\ell_m}}$
	and some $u\in L^p(\Omega;[0,1])$
	such that $\Set{u_{\ell_m}\rvert_\Omega}$ weakly converges in $L^p$ to $u$.
	Let us extend $u$ to the whole $\Rd$
	setting $u\rvert_{\co{\Omega}} = \chi_{E_0}$.
	By Proposition \ref{stm:JK}, we know that
	$J_K$ is convex and lower semicontinuous w.r.t. strong convergence in $L^p(\Rd;[0,1])$
	for any $p\in[1,+\infty)$,
	therefore it is lower semicontinuous w.r.t. the weak $L^p(\Omega;[0,1])$-convergence.
	We deduce
	\[
	\limseq{k}J_K(u_{\ell_m};\Omega)\geq J_K(u;\Omega)\]
	(recall that $u_{\ell_m} = u = \chi_{E_0}$ in $\co{\Omega}$),
	that is, $u$ is a minimiser for $J_K(\,\cdot\,,\Omega)$.
	
	To conclude the proof, it suffices to show that
	we can recover a set $E$ such that $\chi_E\in\call{F}$
	and $\PerK(E;\Omega)\leq J_K(u;\Omega)$.
	We do this by the Coarea Formula: since
	\[
	J_K(u;\Omega) = \int_0^1\PerK(\Set{u>t};\Omega)\de t,
	\]
	for some $t^*\in(0,1)$ it must hold
	$J_K(u;\Omega) \geq \PerK(\Set{u>t^*};\Omega)$.
	Then, the choice $E\coloneqq \chi_{\Set{u>t^*}}$ gets the conclusion.	
\end{proof}
	\section{A nonlocal notion of calibration}\label{sec:calib} 
		In the previous Section
we proved existence of solutions to the nonlocal Plateau's problem.
Now, relying on the paper \cite{Pa},
we want to validate the nonlocal version of the classical sufficiency principle for minimality
based on the concept of calibration.

The notion of calibration may be expressed in very general terms
(see \cites{Mo1,HL} and references therein).
As far as least area surfaces are concerned,
we say that a (classical) \emph{calibration} for the finite perimeter set $E$
is a divergence-free vector field  $\zeta\colon \Rd \to \Rd$
such that $\left| \zeta(x) \right| \leq 1$ a.e. and
$\zeta(x) = \hat{n}(x)$ for $\call{H}^{d-1}$-a.e. $x\in\partial^\ast E$,
$\hat{n}$ being the measure-theoretic inner normal to $E$ defined in \eqref{eq:in-norm}.

It can be shown that if the set $E$ admits a calibration,
then its perimeter equals the infimum in the classical, isotropic Plateau's problem.
To transpose this minimality criterion in the nonlocal setting,
a suitable notion of calibration has to be introduced.

We remind that we assume that $\Rd\times\Rd$ is equipped with the product measure $\Ld \otimes \Ld$.

\begin{dfn}\label{stm:calib}
	Let $u\colon \Rd \to [0,1]$ and	$\zeta\colon \Rd\times\Rd \to \R$ be measurable functions.
	We say that $\zeta$ is a \emph{nonlocal calibration} for $u$ if the following hold:
	\begin{enumerate}
		\item $\va{\zeta(x,y)}\leq 1$ for a.e. $(x,y)\in\Rd\times\Rd$;
		\item for a.e. $x\in\Rd$,
		\begin{equation}\label{eq:div=0}
		\lim_{r\to 0^+}\int_{\co{B(x,r)}} K(y-x)\left(\zeta(y,x)-\zeta(x,y)\right)\de y = 0;
		\end{equation}
		\item for a.e. $(x,y)\in\Rd\times\Rd$ such that $u(x)\neq u(y)$,
		\begin{equation}\label{eq:nlnormal}
		\zeta(x,y)(u(y)-u(x))=\va{u(y)-u(x)}.
		\end{equation}
	\end{enumerate}
\end{dfn}

Some comments about the definition are in order.
Suppose that $\zeta\colon \Rd\times\Rd \to \R$ is a calibration for $u\colon \Rd \to [0,1]$.
	\begin{enumerate}
		\item Up to replacing $\zeta$ with $\tilde\zeta(x,y)\coloneqq ( \zeta(x,y)-\zeta(y,x) )/2$,
			we may always assume that $\zeta(x,y)=-\zeta(y,x)$.
		\item In view of \eqref{eq:summK}, the integral in \eqref{eq:div=0} is convergent for each $r>0$.
		We may interpret \eqref{eq:div=0} as a nonlocal reformulation
		of the classical vanishing divergence condition.
		Nonlocal notions of gradient and divergence operators
		where introduced by G. Gilboa and S. Osher in \cite{GO},
		and they have already been exploited
		to study nonlocal perimeters by Maz\'on, Rossi, and Toledo in \cite{MRT},
		where the authors propose a notion of $K$-calibrable set
		in relation to a nonlocal Cheeger energy.
		\item Suppose that $u=\chi_E$ for some measurable $E\subset \Rd$.
		By \eqref{eq:nlnormal}, $\zeta$ must satisfy 
		\[
		\zeta(x,y) = \begin{cases}
		-1	& \text{if } x\in E, y\in \co{E} \\
		1	& \text{if } x\in \co{E}, y\in E.
		\end{cases}
		\]
		Heuristically, this means that 
		the calibration gives the sign of the inner product
		between the vector $y-x$ and the inner normal to $E$ at the ``crossing point'',
		provided the boundary of $E$ is sufficiently regular (see Figure \ref{fig:zeta}).
		Indeed, if we imagine to displace a particle from to $x$ and $y$,
		$\zeta$ equals $-1$ when the particle exits $E$,
		and it equals $1$ if the particles enters $E$.
	\end{enumerate}

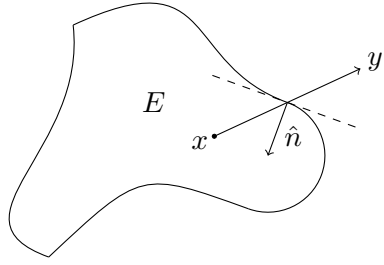
\begin{figure}
	\caption{If $\zeta$ is a calibration for the set $E$
		(i.e. for $\chi_E$) and $x,y$ are as in the picture,
		then $\zeta(x,y)=-1$.}
	\label{fig:zeta}
	\centering
	\begin{tikzpicture}[scale=1.5,rotate=-20]
	\path (-0.1,-0.4) node{$E$};
	
	\draw (-1,0) .. controls (0,1) and (0,0) .. (1,0);
	\draw (1,-1) arc (90:270:-0.5);
	\draw (-0.5,-2) .. controls (0,-1) and (0,-1) .. (1,-1);
	\draw (-0.5,-2) .. controls (-1.5,-2) and (-0.5,-1) .. (-1,0);
	
	\path (0.4,-0.6) node{$x$};
	\filldraw [black] (0.5,-0.5) circle (0.5pt);
	\path (1.6,0.6) node{$y$};
	\path (1.15,-0.26) node{$\hat n$};
	\draw[->] (0.5,-0.5) -- (1.5,0.5);
	\draw[->] (1,0) -- (1,-0.5);
	\draw[dashed] (0.3,0) -- (1.7,0);
	\end{tikzpicture}
\end{figure}

\begin{rmk}
In a very recent, independent work \cite{C},
Cabr\'e studied essentially the same concept of nonlocal calibration:
given an open bounded set $\Omega\subset \Rd$
and a measurable $E\subset\Rd$ such that
$E = \Set{f_E >0}$ for some measurable $f_E\colon \Rd \to \R$,
he introduced the set functional
	\[
		\call{C}_\Omega(F) \coloneqq
			\int\int_{\co{(\co{\Omega} \times \co{\Omega})}} K(y-x)(\chi_F(y)-\chi_F(x)) \sign(f_E(y)-f_E(x)) \de y \de x,
	\]
where $F\subset \Rd$ satisfies $F\cap \co{\Omega}=E\cap \co{\Omega}$.
In \cite{C}*{Theorem 2.4} the author gives sufficient conditions
for the set $E$ to be a minimiser for Plateau's problem 
as well as conditions to grant uniqueness.
As applications, he establishes the local minimality of graphs with $0$ nonlocal curvature
and, very interestingly, re-proves a result in \cite{CRS}
stating that minimisers have null nonlocal curvature in a viscosity sense. 
\end{rmk}
	
The existence of a calibration
is a sufficient condition for a function $u$
to minimise the energy $J_K$ w.r.t compact perturbations,
as the following statement shows:

	\begin{thm}\label{stm:calibcrit}
		Let $E_0\subset \Rd$ be a measurable set such that $J_K(\chi_{E_0};\Omega) <+\infty$,
		and let $\F$ be the family in \eqref{eq:F}.
		If for some $u\in\F$ there exists a calibration $\zeta$, then
		\[
			J_K(u;\Omega)\leq J_K(v;\Omega)\quad \text{for all } v\in\call{F}.
		\]
		Also, when $K>0$, if $\tilde u\in\call{F}$ is another minimiser,
		then $\zeta$ is a calibration for $\tilde u$ as well.
	\end{thm}
	
	\begin{proof}
		For any $v\in\call{F}$, let
			\begin{gather*}
				a(v)\coloneqq \frac{1}{2}\int_{\Omega}\int_{\Omega}K(y-x)\zeta(x,y)(v(y)-v(x))\de y\de x,\\
				b_1(v)\coloneqq - \int_{\Omega}\int_{\co{\Omega}} K(y-x) \zeta(x,y) v(x)\de y \de x, \\
				b_0\coloneqq \int_{\Omega}\int_{\co{\Omega}} K(y-x) \zeta(x,y) \chi_{E_0}(y)\de y \de x.		
			\end{gather*}
		Then, by the definitions of $J_K(\,\cdot\,;\Omega)$, $\zeta$, and $\call{F}$,
			\begin{equation}\label{eq:lowbound0}
				J_K(v;\Omega)\geq a(v) + b_1(v) + b_0.
			\end{equation}
		
		Since it is not restrictive to assume that $J_K(v;\Omega)$ is finite,
		we can suppose that  $a(v)$, $b_1(v)$, and $b_0$  are finite as well.
		
		First of all, we show that $a(v)=-b_1(v)$ for all $v\in\call{F}$.
		Since we may always suppose that $\zeta$ is antisymmetric,
		we have
		\begin{equation*}
		a(v)=-\int_{\Omega}\int_{\Omega}K(y-x)\zeta(x,y)v(x)\de y\de x.
		\end{equation*}
		Besides, \eqref{eq:div=0} gets
		\[\begin{split}
		0 & = -2\lim_{r\to 0^+}\int_{\co{B(x,r)}} K(y-x) \zeta(x,y)\de y \\
		& = -2\lim_{r\to 0^+}\int_{\co{B(x,r)} \cap \Omega} K(y-x) \zeta(x,y)\de y
		-2\int_{\co{\Omega}} K(y-x) \zeta(x,y)\de y,
		\end{split}\]
		whence
		\[
		a(v) = -\lim_{r\to 0^+} \int_\Omega\int_{\co{B(x,r)} \cap \Omega} K(y-x) \zeta(x,y) v(x)\de y \de x
		= -b_1(v).
		\]		
		
		We plug the previous equality in \eqref{eq:lowbound0}
		and obtain
			\begin{equation}\label{eq:lowbound}
			J_K(v;\Omega)\geq b_0 \quad\text{for all } v\in\call{F}.
			\end{equation}
		To infer the optimality of $u$, observe that
		$u$ attains the lower bounds,
		because, being $\zeta$ a calibration for $u$, equality holds in \eqref{eq:lowbound0} when $v=u$.
		
		Now, let $K$ be strictly positive and
		assume that $\tilde u\in\call{F}$ is another minimiser of $J_K(\,\cdot\,;\Omega)$,
		or, in other words, assume that $J_K(\tilde u;\Omega)=b_0$.
		We claim that for a.e. $(x,y)\in\Rd\times\Rd$ such that $ \tilde u(x)\neq\tilde u(y) $ it holds
			\begin{equation}\label{eq:optcond}
				\zeta(x,y)\left(\tilde u(y)-\tilde u(x)\right) = \va{\tilde u(y)-\tilde u(x)}.
			\end{equation}		
		
		The equality holds trivially
		for a.e. $(x,y)\in\co{\Omega} \times \co{\Omega}$, because $u=\tilde u$ in $\co{\Omega}$. 
		Furthermore, from \eqref{eq:lowbound0} we have
		\[b_0=J_K(\tilde u;\Omega)\geq a(\tilde u)+b_1(\tilde u)+b_0 =b_0,\]
		thus
		\[\begin{split}
		\frac{1}{2}\int_{\Omega}\int_{\Omega}K(y-x)
		\left[\va{\tilde u(y)-\tilde u(x)}-\zeta(x,y)(\tilde u(y)-\tilde u(x))\right]\de y \de x \\
		+\int_{\Omega}\int_{\co{\Omega}}K(y-x)
		\left[\va{\tilde u(y)-\tilde u(x)}-\zeta(x,y)(\tilde u(y)-\tilde u(x))\right]\de y \de x = 0.
		\end{split}\]
		Since the integrand is positive and $K>0$,
		\eqref{eq:optcond} is satisfied for a.e. $(x,y)\in\Omega\times\Rd$.
		Eventually, in the case $x\in\co{\Omega}$ and $y\in\Omega$,
		we can derive the conclusion thanks to the antisymmetry of $\zeta$.
	\end{proof}

We put the previous theorem in action
to prove that halfspaces are the unique local minimisers of $J_K(\,\cdot\,;B(0,1))$
w.r.t. their own boundary condition.
This is in line with the case of local perimeters like $\Per_\sigma$ in \eqref{eq:aPer}
associated with strictly convex anisotropies $\sigma$ (see the monograph by F. Maggi \cite{Ma}).
An analogue in the nonlocal setting is already available for fractional perimeters \cite{CRS,ADM}
and, in wider generality, when the kernel is radial and strictly decreasing \cite{BP}.
The use of calibrations enables us to extend the result to all positive, even kernels.

\begin{thm}\label{stm:piani}
	For all $\hat{n}\in\mathbb{S}^{d-1}$,
	set $H_{\hat{n}}\coloneqq \Set{x\in\Rd : x\cdot \hat{n} > 0}$
	and
		\begin{equation}\label{eq:zeta}
		\zeta_{\hat{n}}(x,y)\coloneqq \sign((y-x)\cdot \hat{n}).
		\end{equation}
	Then, $\zeta_{\hat{n}}$ is a calibration for $\chi_{H_{\hat{n}}}$
	in the sense of Definition \ref{stm:calib}, and
		\[	J_K( \chi_{H_{\hat{n}}}; B(0,1) ) \leq J_K( v; B(0,1) )	\]
	for all measurable $v\colon \Rd \to [0,1]$
	such that $v(x)=\chi_{H_{\hat{n}}}(x)$ in $\co{B(0,1)}$.
	
	Moreover, if $K>0$, for any other minimiser $u$ satisfying the same constraint,
	it holds $u(x)=\chi_{H_{\hat{n}}}(x)$.
\end{thm}
\begin{proof}
	It is readily shown that $\zeta$ is a calibration for $\chi_{H_{\hat{n}}}$,
	so that, in the light of Theorem \ref{stm:calibcrit},
	$\chi_{H_{\hat{n}}}$ is a minimiser of the nonlocal perimeter
	subject to its own boundary conditions.
	
	To prove uniqueness, we suppose that $u\colon \Rd \to [0,1]$ is another minimiser. 
	Again by Theorem \ref{stm:calibcrit}, we infer that
	$\zeta_{\hat{n}}$ is a calibration for $u$ as well.
	Explicitly,
	\begin{equation*}\label{eq:optcond-H}
	\sign((y-x)\cdot \hat{n})(u(y)-u(x))=\va{u(y)-u(x)}\qquad \text{a.e. $(x,y)\in \Rd\times\Rd$,}
	\end{equation*}
	whence
	\begin{equation}\label{eq:mon}
	u(x) \leq u(y)	\qquad\text{for a.e. } x,y \in \Rd \text{ such that } x\cdot\hat{n}<y\cdot\hat{n}.
	\end{equation}
	
	We now focus on the superlevel sets of $u$: for $t\in(0,1)$, we define
	\[
	E_t \coloneqq \Set{ u > t},
	\]
	and we observe that if $(x,y)\in E_t  \times \co{E_t}$, it must be $x\cdot \hat n \geq y \cdot \hat n$,
	otherwise, by \eqref{eq:mon} we would have $u(x) \leq u(y)$.
	Therefore, there exists $\lambda_t \in \R$ such that
	$E_t \subset \Set{x : x\cdot \hat{n} \geq \lambda_t}$ and
	$\co{E_t} \subset \Set{y : y\cdot \hat{n} \leq \lambda_t}$,
	whence $\Ld( E_t\, \symdif\,  \Set{x : x\cdot \hat{n} \geq \lambda_t} ) = 0$ for all $t\in(0,1)$.
	Recalling that it holds $u = \chi_H$ in $\co{B(0,1)}$,
	we infer that $\lambda_t = 0$ and this gets
	\[
	\Ld( E_t \symdif H_{\hat{n}} ) = 0 \qquad \text{for all $t\in(0,1)$}.
	\]
	
	Summarising, we found that $u\colon \Rd \to [0,1]$ is a function
	such that, for all $t\in(0,1)$,
	the superlevel set $E_t$ coincides with the halfspace $H_{\hat{n}}$,
	up to a negligible set.
	To conclude,
	we let $\Set{t_\ell}_{\ell\in \N}\subset (0,1)$ be a sequence that converges to $0$ when $\ell\to +\infty$.
	Because it holds
	\[
	\Set{ x : u(x) = 0} = \bigcap_{\ell\in\N} \co{E_{t_\ell}} 
	\quad\text{and}\quad
	\Set{ x : u(x) = 1} = \bigcap_{\ell\in\N} E_{1 - t_\ell},
	\]
	we see that
	$\Ld( \Set{ x : u(x) = 0} \symdif  \co{H} _{\hat{n}}) = 0$ and $\Ld( \Set{ x : u(x) = 1} \symdif  H_{\hat{n}}) = 0$.
	Thus, $u = \chi_{H_{\hat{n}}}$ in $\Rd$
\end{proof}

\begin{rmk}
	Observe that $\zeta_{\hat{n}}$ in \eqref{eq:zeta} is a calibration
	for all $u\colon \Rd \to \R$ such that
	$u(x)=f(x\cdot \hat{n})$ for some increasing $f\colon \R \to \R$.
	Hence, functions of this form are local minimisers.
\end{rmk}

We shall make use of Theorem \ref{stm:piani} in the next chapter,
see Section \ref{sec:sigmaK}.
	\section{Nonlocal curvatures}\label{sec:nlc}
		Let $E\subset \Rd$ be a set with smooth boundary
and let $E_t \coloneqq \Phi(t,E)$,
where $\Phi\colon \R \times \Rd \to \Rd$ is a smooth function
such that $\Phi(0,x) = x$ and $\Phi(t,\,\cdot\,)$ is a diffeomorphism for all $t\in \R$.
Then, it can be proved \cite{CMP} that
	\[
		\der{\PerK(E_t)}{t}\rvert_{t=0}
			= \int_{\partial E} H_K(E,x) \zeta(x)\cdot \hat{n}(x) \de\Hdmu(x),
	\]
where $\zeta(x) \coloneqq \partial_t \Phi(0,x)$,
$\hat{n}(x)$ is the outer unit normal to $\partial E$ at $x$, and
	\begin{equation}\label{eq:dfncurv}\index{curvature!nonlocal --}
	H_K(E,x) \coloneqq -\lim_{r\to 0^+}\int_{\co{B(x,r)}}K(y-x)\left(\chi_E(y)-\chi_{\co{E}}(y)\right)\de y.
	\end{equation}
Thus, we see that $H_K$ is the geometric first variation of the nonlocal perimeter,
and, consistently with the classical case, we dub it \emph{nonlocal curvature}.

Once again, in the fractional case
this kind of nonlocal functional was introduced in the seminal work \cite{CRS}.
For a discussion of the properties of fractional curvatures,
we refer to the paper \cite{AV} by N. Abatangelo and E. Valdinoci.
A wider study of general, translation-invariant curvature functionals was carried out in \cite{CMP}
(see Remark \ref{stm:gencurv} and Chapter \ref{ch:nlc} for more details).

When $K\in L^1(\Rd)$, the nonlocal curvature can be written as a convolution: 
	\[
		H_K(E,x) = -(K\ast \tilde\chi_E)(x) \coloneqq -\int_{\Rd}K(y-x)\tilde\chi_E(y)\de y.
	\]
Hereafter, if $E$ is a set,
	\[
		\tilde\chi_E(x) \coloneqq
			\begin{cases}
				1 & \text{if } x\in E,\\
				-1 & \text{if } x\in \co{E}.
			\end{cases}
	\]
In general, $H_K$ has to be defined by means of the principal value
because we include singular kernels in our analysis.
Precisely, we make the following assumptions:
$K\colon \Rd \to [0,+\infty)$ is a measurable, even function such that
	\begin{equation}\label{eq:complB}
		K\in L^1(\co{B(0,r)}) \qquad\text{for all } r>0
	\end{equation}
and
	\begin{equation}\label{eq:quadriche}
	K \in L^1(Q_\lambda(e)) \qquad \text{for all } e \in \Sdmu \text{ and } \lambda>0,
	\end{equation}
where, for any $e \in \Sdmu$ and $\lambda>0$,
$e^\perp \coloneqq \Set{y\in \Rd : y \cdot e = 0}$,
$\pi_{e^\perp} \coloneqq \id - e\otimes e$ is the orthogonal projection on $e^\perp$, and
	\[
		Q_\lambda(e)\coloneqq
		\left\{y\in \Rd : \va{y\cdot e}\leq \frac{\lambda}{2} \va{\pi_{e^\perp}(y)}^2\right\}
	\]
(see Figure \ref{fig:quadriche}).
The two requirements grant that sets with $C^{1,1}$ boundary have finite nonlocal curvature,
see Proposition \ref{stm:imbert} below.

	\begin{figure}
	\caption{$Q_\lambda(e)$ is the region comprised between an upward and a downward quadric.}
	\label{fig:quadriche}
	\centering
	\begin{tikzpicture}[scale=2,rotate=15]				
	\draw  (0,0) parabola (-1,1);
	\draw  (0,0) parabola (-1,-1);
	\draw  (0,0) parabola (1,1);
	\draw  (0,0) parabola (1,-1);
	
	\shade[right color=lightgray, left color=white] (0,0) parabola (-1,1) -- (-1,0) -- cycle;
	\shade[right color=white, left color=lightgray] (0,0) parabola (1,1) -- (1,0) -- cycle;
	\shade[right color=lightgray, left color=white] (0,0) parabola (-1,-1) -- (-1,0) -- cycle;
	\shade[right color=white, left color=lightgray] (0,0) parabola (1,-1) -- (1,0) -- cycle;
	
	\path (1.3,0.5) node{$Q_\lambda(e)$};
	
	\path (-0.8,0.2) node{$\lambda$};
	\draw[<->]  (-0.7,-0.49) -- (-0.7,0.49);
	\path (-0.1,0.5) node{$e$};
	\draw[->] (0,0) -- (0,0.5);
	
	\path (-1.5,0) node{$e^\perp$};
	\draw (-1.3,0) -- (1.3,0);
	\path (0.05,-0.15) node{$0$};						
	\end{tikzpicture}
\end{figure}
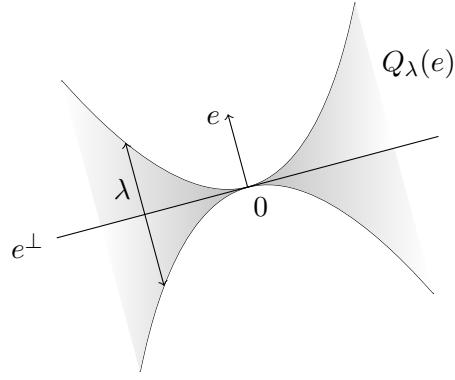

We collect some basic properties of the nonlocal curvature.

\begin{lemma}\label{stm:basicsHK}
	Let $E\subset\Rd$ be an open set
	such that $H_K(E,x)$ is finite for some $x\in\partial E$.
	\begin{enumerate} 
		\item For any $z\in \Rd$ and any orthogonal matrix $R$,
				if $T(y)\coloneqq Ry + z$, then
					\begin{equation}\label{eq:nl-cng-var}
						H_K(E,x)=H_{\tilde K}(T(E),T(x)),
					\end{equation}
				where $\tilde K \coloneqq K\circ R^\trasp$,
				$R^\trasp$ being the transpose of $R$.
				In particular, $H_K$ is invariant under translation.
		\item If $F\subset E$ and $x\in\partial E \cap\partial F$,
				then $H_K(E,x)\leq H_K(F,x)$.
		\item If $\Set{E_\ell}$ is a sequence of sets
				that converges to $E$ in $C^2$
				and the points $x_\ell\in\partial E_\ell$ satisfy $x_\ell\to x$, 
				then $H_K(E_\ell,x_\ell)$ converges to $H_K(E,x)$.
		\item\label{stm:convsets} If $E$ is convex, then $H_K(E,x) \geq 0$.
	\end{enumerate}
	\end{lemma}
	\begin{proof}
		The first two statements follow readily from the definition of $H_K$,
		while the third one can be established similarly to Proposition \ref{stm:imbert} below.		
		As for \ref{stm:convsets}, let $T(y) \coloneqq 2x - y$. 
		Notice that, for any $r>0$,
			\[
				\int_{\co{B(0,r)}} K(y-x) \chi_{E}(y)\de y = \int_{\co{B(0,r)}} K(y-x) \chi_{T(E)}(y)\de y
			\]
		and that, by convexity, $T(E) \subset \co{E}$.
		Then, the right-hand side of \ref{eq:dfncurv} is positive.
	\end{proof}

Finally, we show that the nonlocal curvature is finite
on sets with $C^{1,1}$ boundaries \cites{I,CMP}.
We include the proof for the sake of completeness,
and also to recover \eqref{eq:est-reg-hyp}, which we shall utilise in Chapter \ref{ch:nlc}.
We use the following notation: 
for $e\in\Sdmu$, $x\in\Rd$ and $\delta>0$,
we denote the cylinder of centre $x$ and axis $e$ as 
	\begin{equation*}
	C_e(x,\delta)\coloneqq
		\Set{y\in\Rd : y = x+z+te,\text{ with } z\in e^\perp\cap B(0,\delta),t\in(-\delta,\delta)}.
	\end{equation*}

\begin{prop}\label{stm:imbert}
	Let $E\subset \Rd$ be an open set
	such that $\partial E$ is a $C^{1,1}$-hypersurface.
	Then, for all $x\in\partial E$,
	there exist $\bar{\delta},\lambda>0$ such that 
	\begin{equation}\label{eq:est-reg-hyp}
	\va{H_K(E,x)}\leq \int_{Q_{\lambda,\bar{\delta}}(\hat n)} K(y)\de y
	+ \int_{\co{B(0,\bar\delta)}}K(y)\de y,	
	\end{equation}
	where $\hat{n}$ is the outer unit normal in $x$
	and $Q_{\lambda,\bar\delta}(\hat n)\coloneqq
	\Set{y\in Q_\lambda(\hat n) : \va{\pi_{\hat{n}^\perp}(y)}<\bar{\delta}}$.
	In particu-lar, $H_K(E,x)$ is finite.
\end{prop}
\begin{proof}	
	Since $\Sigma\coloneqq \partial E$ is regular, 
	there exist $\bar{\delta}\coloneqq \bar{\delta}(x)$
	and a function $f\colon \hat{n}^\perp\cap B(0,\bar{\delta})\to (-\bar{\delta},\bar{\delta})$ of class $C^{1,1}$ 
	such that  
		\begin{gather}
			\Sigma\cap C_{\hat n}(x,\bar\delta)
				= \Set{y=x+z-f(z)\hat n: z\in \hat n^\perp \cap B(0,\bar\delta)},
				\nonumber	 \\ 
			E\cap C_{\hat n}(x,\bar\delta) 
				= \Set{ y=x+z-t\hat n: z\in \hat n^\perp \cap B(0,\bar\delta), t\in(f(z),\bar{\delta})}, 
				\nonumber	 \\ 
			 \va{f(z)}\leq \dfrac{\lambda}{2} \va{z}^2\quad\text{for some } \lambda>0. 	\label{ftre}
	\end{gather}
	We may always suppose that $r<\bar\delta$
	and split the integral in \eqref{eq:dfncurv} into the sum
	\[\int_C K(y-x)\tilde\chi_E(y)\chi_{\co{B(x,r)}}(y) \de y
	+ \int_{\co{C}}K(y-x)\tilde\chi_E(y)\de y,\]
	where $C\coloneqq C_{\hat n}(x,\bar\delta)$ for short. 
	Thanks to \eqref{eq:complB}, the second term is finite:
	indeed, since $B(x,\bar\delta)\subset C$, we have that
	\begin{equation}\label{eq:stima-lontano}
	\va{\int_{\co{C}}K(y-x)\tilde\chi_E(y)\de y} \leq \int_{\co{B(0,\bar\delta)}}K(y)\de y.
	\end{equation}
	
	So, it remains to show that the integral
	\[I_r \coloneqq \int_C K(y-x)\tilde\chi_E(y)\chi_{\co{B(x,r)}}(y) \de y\]
	is bounded by a constant that does not depend on $r$.	
	In view of \eqref{eq:epigraph},  and since $K$ belongs to $L^1(\co{B(0,r)})$ for any $r>0$,
	we can write
	\[	I_r	 = \int_{\hat n^\perp\cap B(0,\bar\delta)}\left[
	\int_{f(z)}^{\bar\delta} K(z-t\hat n) b_r(z,t) \de t
	-\int_{-\bar\delta}^{f(z)} K(z-t\hat n) b_r(z,t) \de t
	\right]\de\Hdmu(z),
	\]
	where, for $(z,t)\in [\hat n^\perp\cap B(0,\bar\delta)]\times (-\bar{\delta},\bar{\delta})$,
	\begin{equation}\label{eq:br}
	b_r(z,t)\coloneqq \begin{cases}
	0 & \text{if } \va{z}< r\text{ and } \va{t} < \sqrt{r^2-\va{z}^2}\\
	1 & \text{otherwise}.
	\end{cases}
	\end{equation}
	
	Recalling that $K$ is even,
	we get 
	\[\begin{split}
	I_r = & \int_{\hat n^\perp\cap B(0,\bar\delta)}\left[
	\int_{f(z)}^{\bar\delta} K(z-t\hat n) b_r(z,t) \de t
	- \int_{-f(-z)}^{\bar\delta} K(z-t\hat n) b_r(z,t) \de t
	\right]\de\Hdmu(z) \\
	= & -\int_{\hat n^\perp\cap B(0,\bar\delta)}
	\int_{-f(-z)}^{f(z)} K(z-t\hat n) b_r(z,t) \de t \de\Hdmu(z)
	\end{split}\]
	By \eqref{ftre}, we infer
	\[\begin{split}
	\va{I_r}	\leq &
	\int_{\hat n^\perp\cap B(0,\bar\delta)}
	\int_{-\frac{\lambda}{2}\va{z}^2}^{\frac{\lambda}{2}\va{z}^2} K(z-t\hat n) b_r(z,t) \de t \de\Hdmu(z) \\
	= &
	\int_{Q_{\lambda,\bar{\delta}}(\hat n)} K(y)\chi_{\co{B(0,r)}}(y)\de y.
	\end{split}\]
	Assumption \eqref{eq:quadriche} allows to take the limit in the last inequality,
	and we obtain \eqref{eq:est-reg-hyp}.
\end{proof}

\begin{rmk}
	The second integral in \eqref{eq:est-reg-hyp} takes into account the ``tails'' of $K$,
	while the first one is related to the mean curvature of $\Sigma$.
	Indeed, $\lambda$ in \eqref{ftre} measures how much
	the hypersurface bends in a neighbourhood of $x$.
\end{rmk}

\begin{rmk}	[Generalised curvatures]\label{stm:gencurv}\index{curvature!generalised --}
	According to \cite{CMP}, a function $H$ defined on the couples $(E,x)$,
	where $E\subset \Rd$ is an open set with compact $C^2$ boundary and $x\in \partial E$,
	is a \emph{generalised curvature} if
		\begin{enumerate}
			\item it is invariant under translation;
			\item it is monotone, that is, if $E\subset F$ and $x\in\partial E \cap \partial F$,
				then $H(E,x) \geq H(F,x)$;
			\item it is continuous, that is,
				if $E_\ell\to E$ in $C^2$ and $x_\ell\to x$,
				with $x_\ell\in\partial E_\ell$ and $x\in\partial E$, 
				then $H(E_\ell,x_\ell)$ converges to $H(E,x)$.
		\end{enumerate}
	Well-posedness for the level set flows
	associated with motions by generalised curvature (be it local or nonlocal)
	was established in the same paper.
	Lemma \ref{stm:basicsHK} and Proposition \ref{stm:imbert}
	ensure that $H_K$ lies in framework of \cite{CMP}.
	We shall gain profit of this in Chapter \ref{ch:nlc}.
	\end{rmk}
	
\chapter{$\Gamma$-convergence of rescaled~nonlocal~total~variations}\label{ch:judith}
		In Chapter \ref{ch:nlpnlc} we discussed some general properties
of the nonlocal functionals $\PerK$ and $J_K$,
and we drew some comparisons
between them and, respectively, De Giorgi's perimeter and the total variation.
A further link, of asymptotic nature,
is the subject of the present chapter,
whose content firstly appeared in \cite{BP}.

Let us describe the problem we deal with.
When $u\colon \Rd \to  [0,1]$ is a measurable function,
recalling \eqref{eq:JK}, we define
	\begin{gather*}
		J_\eps^1(u; \Omega) \coloneqq\frac{1}{2}\int_\Omega\int_\Omega K_\eps(y-x)\va{u(y)-u(x)}\de y\de x, \\
		J_\eps^2(u; \Omega) \coloneqq\int_{\Omega}\int_{\co{\Omega}}K_\eps(y-x)\va{u(y)-u(x)}\de y \de x,\\
		J_\eps(u; \Omega) \coloneqq J_{K_\eps}(u; \Omega) = J_\eps^1(u; \Omega) + J_\eps^2(u; \Omega),
	\end{gather*}
where, for $\epsilon>0$, we let
	\begin{equation*}\label{eq:resc-K}
		K_\epsilon(x) \coloneqq \frac{1}{\epsilon^d} K\left( \frac{x}{\eps} \right).
	\end{equation*}
We are interested in the limit
of the family of functionals $\Set{ J_\epsilon(\,\cdot\,;\Omega) }$ 
with respect to the $L^1_{\loc}(\Rd)$-convergence as $\eps \to 0^+$.
This is somewhat reminiscent of the Bourgain-Brezis-Mironescu formula
that we mentioned in Chapter \ref{ch:nle}.
We shall be more precise on this connection in a while,
when we comment the statement of the main result of the chapter,
Theorem \ref{stm:JK-Gconv} below.

Note that the rescaled kernel $K_\epsilon$ has the same mass as $K$,
i.e. formally $\norm{K_\eps}_{L^1(\Rd)} = \norm{K}_{L^1(\Rd)}$,
but, as $\epsilon$ approaches $0$, it gets more and more concentrated around the origin.
So, qualitatively, we guess that
the limiting procedure has a sort of localisation effect, and,
in particular, we expect to recover a total-variation-type functional.
We shall see in Section \ref{sec:statement}
that $\Gamma$-convergence is the tool to describe the phenomenon rigorously.
For the definition of this kind of convergence and its fundamental properties,
we refer to Section \ref{sec:def-Gconv}.

Besides $\Gamma$-convergence,
we also present a compactness property of families
that have equibounded energies.
We establish it in Section \ref{sec:statement}, while
in the following one we outline our approach
to the $\Gamma$-convergence result. 
Detailed proofs are expounded in the remaining sections.

\section{Main result and proof of compactness} \label{sec:statement}
The current chapter is devoted to the proof of the following:

	\begin{thm}[Compactness and $\Gamma$-convergence of the rescaled energy]\label{stm:JK-Gconv}
		Let $\Omega \subset \Rd$ be an open bounded set
		with Lipschitz boundary.
		Let also $K\colon \Rd \to (0,+\infty)$ be an even function such that
		\begin{equation}\label{eq:fastK}
		c_K \coloneqq \frac{1}{2} \int_{\Rd} K(x)\va{x} dx < +\infty,
		\end{equation}
		and define for any measurable $u\colon \Rd \to [0,1]$
			\begin{equation*}
			J_0(u;\Omega) \coloneqq 
			\begin{cases}
			\displaystyle{
				\frac{1}{2}\int_{\Rd} K(z) \left( \int_\Omega \va{ z \cdot \D u} \right) \de z
			}
			& \text{if } u\in\BV(\Omega), \\
			+\infty & \text{otherwise}.
			\end{cases}
			\end{equation*}
		Then, the following hold:
		\begin{enumerate}
			\item\label{stm:JK1-cpt}\index{compactness!for the rescaled nonlocal total variations}
			For all $\ell\in \N$, let $\epsilon_\ell >0$ and
			let $u_\ell \colon \Omega \to [0,1]$ be a measurable function.
			If $\epsilon_\ell \to 0^+$ and
			\begin{equation*}
			\frac{1}{\eps_\ell}J_{\epsilon_\ell}^1(u_\ell ; \Omega) \quad\mbox{is uniformly bounded,}
			\end{equation*}
			then  $\Set{u_\ell}$ is precompact in $L^1(\Omega)$ and
			its cluster points belong to $\BV(\Omega;[0,1])$.  
			\item\index{$\Gamma$-convergence!of the rescaled nonlocal total variations}
			As $\epsilon\to 0^+$, the family $\Set{\epsilon^{-1} J_\epsilon(\,\cdot\,;\Omega)}$ $\Gamma$-converges
			in $L^1_\loc(\Rd)$ to $J_0(\,\cdot\,;\Omega)$.
		\end{enumerate}
	\end{thm}
Assertion \ref{stm:JK1-cpt} provides a compactness criterion for families
that are equibounded in energy and
it grants some extra regularity for the limits point of such families.
The second aspect will be useful for the proof of the lower limit inequality.
The criterion has been already established in \cite{AB2}
by G. Alberti and G. Bellettini in a slightly different framework.
For the reader's convenience, in this Section we shall include its proof. 

Our theorem is very close to a result by Ponce:
	
	\begin{thm}[Corollary 2 and Theorem 8 in \cite{P}]\label{stm:ponce}
		Let $\Omega$ an opens set with compact, Lipschitz boundary
		and let $K$ and $K_\epsilon$ be as above.
		If $u\in \BV(\Omega)$, then
			\begin{equation}\label{eq:pointlim}
				\lim_{\epsilon\to 0^+} \frac{1}{\eps}
					\int_{\Omega}\int_{\Omega}
						K_\eps(y-x)\va{u(y)-u(x)} \de y \de x = J_0(u;\Omega)
			\end{equation}
		In addition, if $\Omega$ is also bounded,
		the right-hand side is the $\Gamma$-limit as $\epsilon\to 0^+$
		of the left-hand one w.r.t. the $L^1(\Omega)$-distance.
	\end{thm}

We remark that, in \cite{P},
contributions to the energy functional of the form of $J_\epsilon^2$ are not considered,
so, to prove our result,
some explanation concerning the $\Gamma$-upper limit would be still required,
even if one took the results by Ponce for granted.
Besides, there, the approach to the $\Gamma$-lower limit inequality relies
on representation formulas for the relaxations of a certain class of integral functionals.
In fact, following \cite{BP}, we propose an independent argument,
which combines the pointwise limit \eqref{eq:pointlim}
and Theorem \ref{stm:piani} about the minimality of halfspaces.

We notice that
in \eqref{eq:fastK} we require a condition that is stronger than \eqref{eq:summK}.
Heuristically, the faster-than-$L^1$ decay at infinity entails that
the ``local'' contribution to the energy, i.e. $J^1_\epsilon$, prevails in the large scale limit;
for this reason, the same limit of Theorem \ref{stm:ponce} is recovered.
We shall motivate the choice of the scaling factor below in this Section,
see Remark \ref{stm:scfactor}.

We conclude this section by proving the first statement of Theorem \ref{stm:JK-Gconv},
that is, the compactness criterion. 
We recall that an analogous result appeared in \cite{AB2}.
We premise a couple of lemmas.

\begin{lemma}\label{stm:G-ineq}
	Let $G \in L^1(\Rd)$ be a positive function.
	Then, for any $u\in L^\infty(\Rd)$ it holds
	\begin{equation*}
	\int_{\Rd\times\Rd} (G\ast G)(z) \va{u(x+z)-u(x)}\de z\de x \leq 4 \norm{G}_{L^1(\Rd)} J_G(u;\Rd).
	\end{equation*}
\end{lemma}
\begin{proof}
	The proof is elementary:
		\[
		\begin{split}
			& \int_{\Rd\times\Rd} (G\ast G)(z) \va{u(x+z)-u(x)}\de z\de x \\
			& \quad = \int_{\Rd} \int_{\Rd} \int_{\Rd} G(y)G(z-y) \va{u(x+z)-u(x)} \de y \de z\de x \\
			& \quad \leq \int_{\Rd} \int_{\Rd} \int_{\Rd} G(y)G(z-y) \va{u(x+z)-u(x+y)} \de y \de z\de x \\
			& \qquad 			+ \int_{\Rd} \int_{\Rd} \int_{\Rd} G(y)G(z-y) \va{u(x+y)-u(x)} \de y \de z\de x \\
			& \quad = 2 \norm{G}_{L^1(\Rd)} \int_{\Rd\times\Rd} G(y) \va{u(x+y)-u(x)}\de y\de x.
		\end{split}
		\]		
\end{proof}

The second lemma is more closely related to the nonlocal nature of the problem we treat.
It shows that 
the rescaled nonlocal energy is bounded, or even vanishing for $\epsilon\to 0^+$,
provided that the sets appearing in the functional are separated by a regular ``frame'',
or that the function under consideration is sufficiently regular.

\begin{lemma}\label{stm:frame}
	Let $E_1$ and $E_2$ be measurable sets in $\Rd$
	and let $u\colon \Rd \to [0,1]$ be measurable.
	\begin{enumerate}
		\item\label{stm:frame1}
		If either there exists a set $F$ with finite perimeter in $\Rd$   
		such that, up to negligible sets, $E_1\subset F$ and $E_2\subset \co{F}$,
		or if $u\in\BV(\Rd)$, then
			\[
				\frac{1}{ \eps}\int_{E_1}\int_{E_2} K_\epsilon(y-x)\va{u(y)-u(x)}\de y \de x \leq c,
			\]
		where $c$ is a constant such that
		either $c = c(K,\Per(F))$ or $c = c(K,\norm{\D u})$.
		\item\label{stm:frame2}
		If $\dist(E_1,E_2)>0$, then
			\[
			\lim_{\epsilon\to0^+} \frac{1}{ \eps}\int_{E_1}\int_{E_2} K_\epsilon(y-x)\va{u(y)-u(x)}\de y \de x = 0
			\]
	\end{enumerate}
\end{lemma}
\begin{proof}
	The argument is the same as Lemma \ref{stm:CMP}-\ref{stm:per2}.
	
	For what concerns \ref{stm:frame1},
	suppose firstly that $E_1\subset F$ and $E_2\subset \co{F}$.
	We bound the double integral over $E_1$ and $E_2$
	by the nonlocal $K_\eps$-perimeter of $F$ in $\Rd$,
	which in turn is bounded by the classical perimeter of $F$:
		\[
		\begin{split}
			& \int_{E_1}\int_{E_2} K_\epsilon(y-x)\va{u(y)-u(x)}\de y \de x \\
					& \quad \leq 2 \int_{F}\int_{\co{F}}K_\epsilon(y-x)\de y \de x \\
					& \quad = \int_{\Rd}\int_{\Rd}
							K(z) \va{\chi_{F}(x+\eps z)-\chi_{F}(x)}\de z \de x \\
					& \quad  \leq 2 \eps c_K \Per(F).
		\end{split}
		\]
	On the other hand, if $u\in \BV(\Rd)$, a similar reasoning shows
		\[
		\begin{split}
				\int_{E_1}\int_{E_2} K_\epsilon(y-x)\va{u(y)-u(x)}\de y \de x
					& \leq \int_{\Rd}\int_{\Rd} K_\epsilon(y-x) \va{u(y)-u(x)}\de y \de x \\
					& \leq 2 \eps c_K \norm{\D u}.
		\end{split}
		\]
		
	Turning to \ref{stm:frame2}, we have
		\[
		\begin{split}
			\frac{1}{ \eps}\int_{E_1}\int_{E_2} K_\epsilon(y-x)\va{u(y)-u(x)}\de y \de x
				& \leq \frac{2}{\eps r}\int_{E_1}\int_{E_2} K_\epsilon(y-x)\va{y-x}\de y \de x \\
				& = \frac{2}{r}\int_{E_1}\int_{\Rd}K(z)\va{z}\chi_{E_2}(x+\eps z)\de z \de x,
		\end{split}
		\]
	where $r\coloneqq \dist(E_1,E_2)>0$.
	Statement \ref{stm:frame2} follows by Lebesgue's Theorem. 
\end{proof}

\begin{rmk}[Choice of the scaling factor]\label{stm:scfactor}
	The previous Lemma shows that
	$J_\epsilon(u;\Omega) = O(\eps)$
	when $u$ is a function of bounded variation in $\Rd$.
	Accordingly, we multiply $J_\epsilon$ by $\epsilon^{-1}$ to recover a nontrivial limit
	in our $\Gamma$-convergence theorem.
\end{rmk}

Now we are ready to prove the compactness criterion.

\begin{proof}[Proof of statement \ref{stm:JK1-cpt} in Theorem \ref{stm:JK-Gconv}]
	For notational convenience,
	we write respectively $\epsilon$ and $u_\epsilon$
	in place of $\epsilon_\ell$ and of $u_{\ell}$.
	
	We are going to exhibit a sequence $\Set{v_\eps}$ 
	that is asymptotically equivalent in $L^1(\Rd)$ to $\Set{u_\epsilon}$,
	and that is precompact as well.
	To this aim, we extend each $u_\epsilon$ outside $\Omega$
	setting $u_\eps = 0$
	(in wider generality, any constant extension would fit).
	Let $\rho\in C^\infty_c(\Rd)$ be positive
	and let
		\[
			v_\eps (x) \coloneqq (\rho_\eps \ast u_\eps)(x),
		\]
	where
		\[
			\rho_\epsilon(x) \coloneqq \frac{1}{c \epsilon^d} \rho\left(\frac{x}{\eps}\right),
			\quad\text{with } c \coloneqq \int_{\Rd}\rho(x)\de x.
		\]
	Notice that any $v_\epsilon$ is supported in some ball $B$ containing $\Omega$.
	By properties of convolutions,
		\begin{equation}\label{eq:asympteq}
			\int_{\Rd} \va{v_\eps(x)-u_\eps}(x)\de x
							\leq \int_{\Rd}\int_{\Rd} \va{\rho_\eps(z)}\va{u_\eps(x+z) - u_\eps(x)}\de z\de x
		\end{equation}
	and
		\begin{equation}\label{eq:BVbound}
			\begin{split}
				\int_B \va{\nabla v_\eps(x)}\de x
							& =\int_{\Rd} \va{\nabla v_\eps(x)}\de x \\
							& \leq \int_{\Rd}\int_{\Rd}\va{\nabla\rho_\eps(z)}\,
										\va{u_\eps(x+z) - u_\epsilon(x)}\de z\de x.
		\end{split}
		\end{equation}
	We claim that, if we choose the mollifier $\rho$ suitably, then
	the asymptotic equivalence between $\Set{u_\eps}$ and $\Set{v_\eps}$
	and a uniform bound on the $\BV$-norm of $\Set{v_\eps}$
	follow respectively from \eqref{eq:asympteq} and \eqref{eq:BVbound}.
	In turn, the claim entails the conclusion:
	by virtue of Theorem \ref{stm:BVcpt}, up to extraction of subsequences,
	$\Set{v_\epsilon}$ converges to some $u\in\BV(B)$ in $L^1(B)$,
	and hence $\Set{u_\epsilon}$ tends in $L^1(B)$ to the same function.
	
	Let us now validate the claim.
	We define the function
	$T(s)\coloneqq 0 \vee (s \wedge 1)$
	and the truncated kernel $G \coloneqq T \circ K\in L^1(\Rd)\cap L^\infty(\Rd)$.
	We observe that the convolution $G*G$ is positive and continuous,
	therefore we can pick $\rho\in C_c^{\infty}(\Rd)$ such that
		\begin{equation}
			0 \leq \rho  \leq G\ast G \quad \text{and} \quad \va{\nabla \rho} \leq G \ast G. 
		\end{equation}
	Setting
		\[
			G_\epsilon(x) \coloneqq \frac{1}{\epsilon^d}G\left(\frac{x}{\eps}\right),
		\]
	we infer from \eqref{eq:asympteq} and \eqref{eq:BVbound} 
		\begin{equation*}\label{eq:asympteq2}
			\int_{\Rd} \va{v_\eps(x)-u_\eps(x)}\de x
				\leq \int_{\Rd}\int_{\Rd} \va{G_\eps\ast G_\eps(z)}\va{u_\eps(x+z) - u_\eps(x)}\de z\de x
		\end{equation*}
	and
		\begin{equation*}
			\int_{\Rd} \va{\nabla v_\eps(x)}\de x
				\leq \frac{1}{\eps}\int_{\Rd}\int_{\Rd}
					\va{G_\eps\ast G_\eps(z)}\va{u_\eps(x+z)-u_\eps(x)}\de z\de x.
	\end{equation*}
	Both the right-hand sides of the inequalities can be bounded above
	by Lemma \ref{stm:G-ineq}. Indeed, we have 
		\[
			\begin{split}
			& \int_{\Rd}\int_{\Rd} \va{G_\eps\ast G_\eps(z)} \va{u_\eps(x+z)-u_\eps(x)}\de z\de x \\
			& \quad \leq 4 \norm{G}_{L^1(\Rd)} J_{G_\eps}(u_\eps;\Rd) \\
			& \quad \leq 4 \norm{G}_{L^1(\Rd)} J_\eps(u_\eps;\Rd) \\
			& \quad = 4 \norm{G}_{L^1(\Rd)}
						\left( J^1_\eps(u_\eps;\Omega) + J^2_\eps(u_\eps;\Omega) + J^1_\eps(u_\eps;\co{\Omega}) \right)
		\end{split}
		\]
	Note that $J^1_\eps(u_\eps;\co{\Omega})=0$,
	because we set $u_\eps = 0$ outside of $\Omega$.
	Moreover, by assumption, there exists $M\geq 0$ such that
		\[
			J^1_\eps(u_\eps;\Omega) \leq \eps M.
		\]
	In third place, we can invoke Lemma \ref{stm:frame} on $J^2_\epsilon(u_\epsilon;\Omega)$,
	because the boundary of $\Omega$ is Lipschitz.
	On the whole, we deduce
		\[
			\int_{\Rd}\int_{\Rd} \va{G_\eps\ast G_\eps(z)}\va{\chi_{E_\eps}(x+z)-\chi_{E_\eps}(x)}\de z\de x = O(\eps),
		\]
	as desired.		
\end{proof}

\begin{rmk}[Locality defect]
	In the previous proof,
	we utilised the identity
		\[
			J_\eps(u;\Rd) = J^1_\eps(u;\Omega) + J^2_\eps(u;\Omega) + J^1_\eps(u;\co{\Omega}).
		\]
	More broadly, let $F$ be an arbitrary measurable set.
	If we split a given reference domain  $\Omega$
	in the disjoint regions $\Omega\cap F$ and $\Omega\cap \co{F}$,
	we have
		\[
			\begin{split}
			J_K^1(u;\Omega) & = J_K^1(u;\Omega\cap F) + J_K^1(u;\Omega\cap \co{F}) \\
										& \quad + \int_{\Omega\cap F}\int_{\Omega\cap \co{F}} K(y-x)\va{u(y)-u(x)}\de x\de y.
			\end{split}
		\]
	We remark that the sum of the energies stored in each of the sets of the partition
	is smaller than the energy of $\Omega$.
	This can be seen as a feature of nonlocality.
	The difference is given by the mutual interaction
	between $\Omega\cap F$ and $\Omega\cap \co{F}$,
	which, borrowing the terminology suggested in \cite{AB2},
	we may dub \emph{locality defect}\index{locality defect}.
\end{rmk}
	
	\section{Overview of the proving strategy}\label{sec:strategy}
		Before discussing in depth the proof
of the $\Gamma$-convergence result in Theorem \ref{stm:JK-Gconv},
we summarise in this Section the strategy that we adopt.

The key point of our approach is the possibility  
of reasoning in terms of sets rather than functions.
Indeed, thanks to an abstract result in \cite{CGL}
by A. Chambolle, A. Giacomini, and L. Lussardi,
the $\Gamma$-convergence of $\Set{\epsilon^{-1}J_\epsilon}$
regarded as a family of functionals on measurable sets
to the restriction of $J_0$
is sufficient to yield the more general result of Theorem \ref{stm:JK-Gconv}.
We include the precise statement in the next lines;
we recall that we say that
the collection of measurable sets $\Set{E_\eps}$ $L^1_\loc(\Rd)$-converges to $E$
if $\Set{\chi_{E_\eps}}$ converges to $\chi_E$ w.r.t. the $L^1_\loc(\Rd)$-convergence.
	
	\begin{prop}[Proposition 3.5 in \cite{CGL}]\label{stm:CGL}
		Let $U\subset \Rd$ be open and bounded.
		Suppose that $\Set{J_\ell}$ is a sequence of convex functionals on $L^1(U)$
		such that the generalised Coarea Formula \eqref{eq:gencoarea} holds for all $\ell\in\N$
		and let us set for all measurable $E\subset U$
			\[
				\tilde J_\ell(E) \coloneqq J_\ell(\chi_E)
			\]
		Suppose that there exists a functional
		$\tilde J$ defined on measurable sets of $U$
		such that the sequence $\{\ \tilde J_\ell\ \}$ 
		$\Gamma$-converges to $\tilde J$ w.r.t. the $L^1$-convergence,
		and put
			\[
				J(u)\coloneqq\int_{-\infty}^{+\infty}\tilde J(\chi_{\Set{u>t}})\de t.
			\]
		Then, the sequence $\Set{J_\ell}$ $\Gamma$-converges to $J$ w.r.t. $L^1$-convergence.				
	\end{prop}

Recalling Propositions \ref{stm:a-coarea}, \ref{stm:JK}, and \ref{stm:JK-coarea},
we see that the $\Gamma$-convergence result in Theorem \ref{stm:JK-Gconv} holds as soon as
we prove the following:
	\begin{thm}[$\Gamma$-convergence of rescaled nonlocal perimeters]
		\label{stm:PerK-Gconv}\index{$\Gamma$-convergence!of the rescaled nonlocal perimeters}
		Let $\Omega$ and $K$ be as in Theorem \ref{stm:JK-Gconv}.
		For $E\subset \Rd$ measurable, consider
			\begin{gather*}
				J^i_\epsilon(E;\Omega) \coloneqq J^i_\epsilon(\chi_E;\Omega)
					\quad\text{for } i=1,2,\quad
				J_\eps(E;\Omega) \coloneqq J_\epsilon(\chi_E;\Omega),
			\end{gather*}
		and
			\[
				J_0(E;\Omega) \coloneqq J_0(\chi_E;\Omega).
			\]
		Then, for any given measurable $E$, we have that
			\begin{enumerate}
				\item for any family $\Set{E_\eps}$
					that converges to $E$ in $L^1_\loc(\Rd)$ as $\epsilon\to0^+$, it holds
					\begin{equation}\label{eq:PerK-Gliminf}
						J_0(E;\Omega) \leq \liminf_{\epsilon\to 0^+} \frac{1}{\eps} J_\epsilon^1(E_\epsilon;\Omega).
					\end{equation}
				\item there exists a family $\Set{E_\eps}$
					that converges to $E$ in $L^1_\loc(\Rd)$ as $\epsilon\to0^+$
					and that satisfies					
					\begin{equation*}
						\limsup_{\epsilon\to 0^+} \frac{1}{\eps} J_\epsilon (E_\epsilon;\Omega) \leq J_0(E;\Omega).
					\end{equation*}
			\end{enumerate}
	\end{thm}

The two assertions in Theorem \ref{stm:PerK-Gconv} clearly entail
the $\Gamma$-convergence of $\Set{\epsilon^{-1}J_\epsilon(\,\cdot\,;\Omega)}$ to $J_0(\,\cdot\,;\Omega)$
w.r.t. the $L^1_\loc(\Rd)$-convergence,
because $J_\eps^2(\,\cdot\,;\Omega)$ is positive.

Our result is akin to several others already appeared in the literature.
In addition to \cites{BBM,Da,P},
we refer to the aforementioned \cite{AB2},
where the asymptotics of a nonlocal model for phase transitions is studied.
The authors consider the rescaled energy
	\[
		\F_\eps(u;\Omega) \coloneqq \frac{1}{4}\int_{\Omega}\int_{\Omega} K_\eps(y-x)\big( u(y) - u(x) \big)^2 \de y \de x
										+ \int_{\Omega} W(u(x)) \de x,
	\]
where $W\colon \R \to [0,+\infty)$ is a double-well potential.
Under the same summability assumptions on $K$ as ours,
they establish the $\Gamma$-convergence of the family $\Set{\epsilon^{-1}\F_\eps(\,\cdot\,;\Omega)}$
to a limit anisotropic surface energy.
Interactions between $\Omega$ and its complement are not relevant in this framework.
The reader interested in this sort of energies may consult also \cite{BBP}.

Another result that is linked to Theorem \ref{stm:PerK-Gconv} was proved
by L. Ambrosio, G. De Philippis, and L. Martinazzi in \cite{ADM}.
There, they showed that fractional perimeters $\Gamma$-converge,
as the fractional parameter approaches $1$, to De Giorgi's perimeter.
The scaling is substantially different from ours,
but some of the techniques exploited in that work still fit in our setting.
For the asymptotics of fractional energies, we refer to the survey \cite{V}. 

Let us now outline how we deal with the proof of Theorem \ref{stm:PerK-Gconv}.
We start with a qualitative picture.

The contributions $J^1_\epsilon$ and $J^2_\epsilon$
that compound the rescaled nonlocal perimeter functional have different asymptotic behaviours:
when $\epsilon$ is small,
the former is concentrated near the portions of the boundary of $E$ inside $\Omega$,
the latter instead captures the parts that are close to $\partial\Omega$.
Proposition \ref{stm:pointconv} in Section \ref{sec:PerK-Glimsup} makes the heuristics precise,
showing that the pointwise limit and the $\Gamma$-limit do not agree.
However, once the pointwise limit has been computed,
it is possible to recover the $\Gamma$-upper limit by a density result,
see Lemma \ref{stm:density}.

As for the $\Gamma$-lower limit, it is convenient to regard $J_0$
as an anisotropic perimeter functional.
We observe that we can write
	\begin{equation}\label{eq:J0}
	J_0(E;\Omega) =
		\begin{cases}
			\displaystyle{
				\int_{\partial^\ast E \cap \Omega} \sigma_K(\hat{n}(x)) \de \Hdmu(x)
			}			& \text{if $E$ is a finite perimeter set in $\Omega$,} \\
		+\infty 	& \text{otherwise},
	\end{cases}
	\end{equation}
where $\hat{n}\colon \partial^\ast E \to \mathbb{S}^{d-1}$ is
the measure-theoretic inner normal of $E$ defined by \eqref{eq:in-norm}
and $\sigma_K\colon \Rd \to [0,+\infty)$ is the anisotropy
\begin{equation}\label{eq:sigmaK}
\sigma_K(p) \coloneqq \frac{1}{2}\int_{\Rd} K(z)\va{z\cdot p}\de z,
\qquad\text{for } p\in\Rd.
\end{equation}
This function is evidently a norm on $\Rd$.

\begin{rmk}[The radial case \cite{BP}]
	When $K$ is radial, $J_0$ equals De Giorgi's perimeter
	times a constant that depends on the dimension of the space and on $K$.
	To see this, admit that $K(x) = \bar K (\va{x})$ for some $\bar K\colon [0,+\infty) \to [0,+\infty)$.
	Then, for any $\hat p\in\mathbb{S}^{d-1}$, we find
	\begin{align*}
	\sigma_K(\hat p) & = \frac{1}{2} \left(\int_{0}^{+\infty} \bar K(r) r^d \de r \right) 
	\int_{\Sdmu} \va{e\cdot \hat p} \de \call{H}^{d-1}(e) \\
	& = c_K	\fint_{\Sdmu} \va{e\cdot e_d} \de \call{H}^{d-1}(e),
	\end{align*}
	where $e_d \coloneqq (0,\dots,0,1)$ is the last element of the canonical basis.	
	
	Theorem \ref{stm:PerK-Gconv} was proved for radial and strictly decreasing kernels
	in \cite{BP}. The content of this chapter shows that
	the same arguments may be conveniently adapted to a more general setting.	
\end{rmk}

Following \cite{ADM,AB2}, we prove the lower limit inequality
\textit{via} the strategy introduced by I. Fonseca and S. M\"uller in \cite{FM},
which amounts to turn the proof of \eqref{eq:PerK-Gliminf}
into an inequality of Radon-Nikodym derivatives.
To accomplish the task,
the compactness criterion that we proved in the previous section comes in handy.

\begin{rmk}[Semicontinuity of the $\Gamma$-limit]\label{stm:J0-lsc}
	If Theorem \ref{stm:JK-Gconv} holds,
	then we obtain as a by-product that
	$J_0$ is $L^1_\loc(\Rd)$-lower semicontinuos. 
	Indeed, $\Gamma$-limits are always lower semicontinuous \cite{B}.
	More straightforwardly,
	the semicontinuity may be seen as a consequence of the following equality,
	which holds for all $u\in\BV(\Omega)$ \cite{Gr}:
		\[
			J_0(u;\Omega) = \sup
						\Set{-\int_{\Omega} u\div \zeta :
								\zeta\in C^1_{\mathrm{c}}(\Omega;\Rd), \sigma^\circ_K(\zeta) \leq 1
						}.
		\]
	Here, $\sigma^\circ_K$ is the dual norm of $\sigma_K$, that is
		\[
			\sigma^\circ_K(q) \coloneqq \sup\Set{ q \cdot p : p\in\Rd \text{ and } \sigma_K(p) \leq 1}
			\qquad \text{ for all } q\in\Rd.
		\]
	The previous formula may be seen as a more general version of \eqref{eq:TV}.
	
	Note further that, being $\sigma_K$ a norm $\Rd$,
	(semi)continuity w.r.t. other notions of convergence may be derived for the $\Gamma$-limit 
	as a consequence of Reshetnyak's Theorems \ref{stm:reshetnyak}.
\end{rmk}
	\section{Pontwise limit and $\Gamma$-upper limit}\label{sec:PerK-Glimsup}
		In the current section, we establish the $\Gamma$-upper limit inequality:
we show that, for any given measurable $E\subset \Rd$,
there exists a family $\Set{E_\eps}$
that converges to $E$ in $L^1_\loc(\Rd)$ as $\epsilon\to0^+$
and that has the property that
	\begin{equation}\label{eq:PerK-Glimsup}
		\limsup_{\epsilon\to 0^+} \frac{1}{\eps} J_\eps(E_\epsilon;\Omega) \leq 	J_0(E;\Omega).
	\end{equation}

We obtain the inequality above in two steps.
We firstly compute the pointwise limit of $\Set{\eps^{-1}J_\epsilon}$
by taking advantage of \eqref{eq:pointlim}.
We find that, as $\epsilon\to0^+$,
$\Set{J_\epsilon}$ approaches the presumed $\Gamma$-limit $J_0$ plus a boundary contribution.
Next, we prove that the latter is negligible in the $\Gamma$-limit
by reasoning on a suitable class of sets that is dense in energy for $J_0$
(recall Lemma \ref{stm:prop-Gconv}\ref{stm:density-Gconv}). 

We now study the pointwise convergence of $\Set{\eps^{-1}J_\epsilon}$.
An application of the Bourgain-Brezis-Mironescu formula
to the study of nonlocal perimeters already appeared in \cite{MRT},
where J. Maz\'on, J. Rossi, and J. Toledo exploited the approximation
established by J. D\'avila in \cite{Da}.
In \cite{MRT} the authors focus on compactly supported, $L^1$ kernels.
Here, we include a slightly more general statement:

	\begin{prop}[Pointwise limit of the rescaled nonocal perimeter]\label{stm:pointconv}
		Let $\Omega \subset \Rd$ be an open set
		whose boundary is compact and Lipschitz.
		Let also \eqref{eq:fastK} hold.
		Then, if $E$ is a finite perimeter set in $\Rd$,
		\begin{equation}\label{eq:MRT}
			\lim_{\epsilon\to0^+}\frac{1}{\eps}J_\epsilon(E;\Omega)
			= J_0(E;\Omega) + \int_{\partial^\ast E\cap\partial \Omega} \sigma_K(\hat{n}(x)) \de\Hdmu(x). 
		\end{equation}
	\end{prop}
	\begin{proof}
	By virtue of Theorem \ref{stm:ponce},
	\[
	\lim_{\epsilon\to 0^+} \frac{1} {\epsilon} J^1_\epsilon(E;\Omega) = J_0(E;\Omega),
	\]
	so, we only have to take care of $J^2_\eps$.
	We start from the $K_\epsilon$-perimeter of $E$ in the whole space
	and we rewrite it as the sum of three contributions:
	\[
	\int_E \int_{\co{E}} K_\epsilon(y-x) \de x \de y = 
	J^1_\epsilon( E ; \Omega ) +  J^2_\epsilon( E ; \Omega ) + J^1_\epsilon( E ; \co{\Omega} ).
	\]
	Since the topological boundary of $\Omega$ is negligible,
	we may appeal again to Theorem \ref{stm:ponce}
	and take the limit $\epsilon\to 0^+$.
	Let $\mathrm{int}\, \co{\Omega}$ be the interior of $\co{\Omega}$. We get
	\[\begin{split}
	\lim_{\epsilon\to 0^+} \frac{1}{\eps} J^2_\epsilon(E;\Omega)
	& = J_0(E) - J_0(E;\Omega) - J_0(E; \mathrm{int}\, \co{\Omega}) \\
	& = \int_{\partial^\ast E\cap\partial \Omega} \sigma_K(\hat{n}(x)) \de\Hdmu(x),
	\end{split}\]
	as desired.

\end{proof}

Let us now go back to \eqref{eq:PerK-Glimsup}.
The inequality holds trivially if the right-hand side is not finite,
therefore we may assume that $E$ is a Caccioppoli set in $\Omega$.
Note as well that if $E$ has finite perimeter in the whole space $\Rd$
and if 
	\begin{equation}\label{eq:transversality}
		\Hdmu(\partial^\ast E \cap \partial\Omega)=0,
	\end{equation}
then, in view of Proposition \ref{stm:pointconv}, 
the choice $E_\eps \coloneqq E$ for all $\epsilon>0$ defines a recovery family.
When \eqref{eq:transversality} hold,
we say that $E$ is \emph{transversal} to $\Omega$.

At this stage, by Lemma \ref{stm:prop-Gconv}\ref{stm:density-Gconv},
the proof of the upper limit inequality is accomplished
once we show that the class of finite perimeter sets in $\Rd$
that are transversal to $\Omega$ is dense in energy;
in particular, we consider approximations by polyhedra.
We say that a set $P\subset \Rd$ is a $d$-dimensional \emph{polyhedron}
if it is an open set whose boundary is a Lipschitz hypersurface
contained in the union of a finite number of affine hyperplanes.

\begin{lemma}[Density of polyhedra]\label{stm:density}
	Let $E$ be a finite perimeter set in $\Omega$.
	Then, there exists a family $\Set{P_\epsilon}$ of polyhedra with the following properties:
		\begin{enumerate}
			\item\label{stm:approx1} $\Hdmu(\partial P_\eps \cap \partial \Omega) = 0$;
			\item\label{stm:approx2} $P_\eps \to E$ in $L^1(\Omega)$
					and $\Per(P_\eps;\Omega)\to\Per(E;\Omega)$;
			\item\label{stm:approx3} $J_0(P_\epsilon;\Omega) \to J_0(E;\Omega)$.
		\end{enumerate}
\end{lemma}
\begin{proof}
	Assertion \ref{stm:approx2} is a standard one,
	see for instance \cite{Ma}.
	Also, a family that satisfies \ref{stm:approx2} may be refined
	so that \ref{stm:approx1} hold as well, see for instance \cite{ADM}*{Proposition 15}.	
	Lastly, \ref{stm:approx3} is a consequence
	of Reshetnyak's Continuity Theorem and \ref{stm:approx2}.
\end{proof}		
	\section{Characterisation of the anisotropy}\label{sec:sigmaK}
		This section is devoted to the proof
of a characterisation of the anisotropic norm $\sigma_K$
that appears in the definition of the limit functional $J_0$.
A key tool in the proof is the local minimality of halfspaces.
The characterisation proves to be useful
to establish the $\Gamma$-inferior limit inequality in Theorem \ref{stm:PerK-Gconv},
see the next section.

Recall that, for all $p\in\Rd$,
	\begin{equation*}
	\sigma_K(p) \coloneqq \frac{1}{2}\int_{\Rd} K(z)\va{z\cdot p}\de z.
	\end{equation*}
By homogeneity,
without loss of generality, we can think of $\sigma_K$ as function defined on $\Sdmu$.
For $p\in\Rdmz$, let us recall the notations 
	\[	\hat{p}\coloneqq \frac{p}{\va{p}}
		\quad\text{and }\quad
		H_{\hat{p}}\coloneqq \Set{x\in\Rd : x\cdot \hat{p} > 0}.
	\]
Also, in this section we use the symbol $B_r$ as a shorthand for $B(0,r)$, with $r>0$,
and we set $B\coloneqq B_1 = B(0,1)$.

Our goal is validating the following:

\begin{lemma}\label{stm:sigmaK}
	For any $\hat{p}\in\mathbb{S}^{d-1}$,
		\begin{equation}\label{eq:char-sigmaK}
			\sigma_K (\hat{p})
				= \inf\Set{
							\liminf_{\epsilon\to 0^+} \frac{1}{\omega_{d-1} \eps} J^1_\eps(E_\epsilon;B) :
							E_\eps \to H_{\hat{p}} \text{ in } L^1(B)
							},
		\end{equation}
	where $\omega_{d-1}$ is the $(d-1)$-dimensional Lebesgue measure
	of the unit ball in $\R^{d-1}$.
\end{lemma}

For $\hat{p}\in  \Sdmu$,
we label the quantity in the right-hand side of \eqref{eq:char-sigmaK}:
	\begin{equation}\label{eq:sigma'K}
		\sigma'_K(\hat{p}) \coloneqq \inf\Set{
			\liminf_{\epsilon\to 0^+} \frac{1}{\omega_{d-1} \eps} J^1_\eps(E_\epsilon;B) :
						E_\eps \to H_{\hat{p}} \text{ in } L^1(B)
		}.
	\end{equation}
We highlight that the previous equality corresponds
to the definition of the surface tension given in the phase-field model in \cite{AB2}.

Observe that, by virtue of Theorem \ref{stm:ponce}, we can compute $\sigma_K$ as a pointwise limit:
	\begin{equation}\label{eq:sigmaK-P}
		\sigma_K(\hat{p}) = \lim_{\epsilon\to 0^+} \frac{1}{\omega_{d-1}\eps} J^1_\eps(H_{\hat{p}};B).
	\end{equation}
It follows that $\sigma_K (\hat{p}) \geq \sigma'_K(\hat{p})$.

To the purpose of proving the reverse inequality,
we introduce a third function $\sigma''_K$:
for $\hat{p}\in \mathbb{S}^{d-1}$ and $\delta\in(0,1)$, we let
	\[
		\sigma''_K(\hat{p}) \coloneqq \inf\Set{
		\liminf_{\epsilon\to 0^+} \frac{1}{ \omega_{d-1} \eps} J^1_\eps(E_\epsilon;B) :
			E_\eps \to H_{\hat{p}} \text{ in } L^1(B)
				\text{ and }
			E_\eps \symdif H_{\hat{p}} \subset B_{1-\delta}
			}.
	\]
Our notation does not explicitly show the dependence of $\sigma''_K$ on the parameter $\delta$
since \textit{a posteriori} the values of $\sigma''_K$ are not influenced by it.

The thesis of Lemma \ref{stm:sigmaK} holds
as soon as we prove that
$\sigma_K(\hat{p}) \leq \sigma''_K(\hat{p}) \leq \sigma'_K(\hat{p})$ for all $\hat{p}\in  \Sdmu$.
We establish the former inequality by taking advantage
of the minimality of halfspaces, i.e. Theorem \ref{stm:piani},
while for the latter we need the following technical result,
which parallels a similar one proved in \cite{ADM} for fractional perimeters:
\begin{lemma}[Gluing]\label{stm:gluing}
	Suppose that $\delta_1,\delta_2\in\R$ satisfy $\delta_2 > \delta_1>0$,
	and that $E_1,E_2$ are measurable sets such that
	$J^1_K(E_i;\Omega)<+\infty$ for both $i=1,2$.
	Let us define
		\[
			\Omega_1 \coloneqq \Set{ x \in \Omega : \dist(x,\co{\Omega}) \leq \delta_1}
			\quad\text{and}\quad
			\Omega_2 \coloneqq \Set{ x \in \Omega : \dist(x,\co{\Omega}) > \delta_2}.
		\]
	Then, there exists a measurable set $F$ with the following properties: 
	\begin{enumerate}
		\item\label{stm:glue1}
			$F \cap \Omega_1 = E_1 \cap \Omega_1$
			and
			$F \cap \Omega_2 = E_2 \cap \Omega_2$;
		\item\label{stm:glue2}
			$\Ld \big( (F \symdif E_2)\cap\Omega\big) \leq \Ld \big((E_1\symdif E_2)\cap\Omega\big)$;
		\item for all $\eta >0$
			\begin{equation}\label{eq:gluing}
				\begin{split}
					J_\eps^1(F;\Omega) & \leq J^1_\eps(E_1;\Omega_\eta) + J_\eps^1(E_2;\Omega) \\
									& \quad
										+ \frac{ \eps}{\delta_2-\delta_1} \int_{\Rd} K(z)\va{z} \de z\,
										\Ld\big( (E_1 \symdif E_2) \cap \Omega \big) \\
									& \quad
										+ \frac{\eps}{\eta} \int_{\Omega\cap\co{\Omega}_\eta}
										\int_{\Rd} K(z)\va{z}\chi_{\Omega \cap \co{\Omega}_2}(x+\eps z)\de z\de x.					
			\end{split}
			\end{equation}
		where $\Omega_\eta\coloneqq \Set{x\in\Omega : \dist(x,\co{\Omega})\leq \delta_2+\eta}$.
		\end{enumerate}
	\end{lemma}
\begin{proof}
	We consider a function $w\colon \Omega \to [0,1]$,
	which is, loosely speaking, a convex combination
	of the data $u \coloneqq \chi_{E_1}$ and $v \coloneqq \chi_{E_2}$.
	Precisely, we set $w\coloneqq (1-\phi)u  + \phi v$,
	where $\phi\in C^{\infty}_\mathrm{c}(\Rd)$ is such that 
		\[
			0\leq \phi \leq 1 \text{ in } \Omega,
			\quad
			\phi=0 \text{ in } \Omega_1,
			\quad
			\phi =1 \text{ in } \Omega_2,
			\quad\text{and}\quad
			\va{\nabla \phi} \leq \frac{2}{\delta_2-\delta_1}.\]
			
	We claim that, for all $\eta>0$, $w$ satisfies the following inequality:
		\begin{equation}\label{eq:w}
		\begin{split}
		J_\eps^1(w;\Omega) & \leq J^1_\eps(E_1;\Omega_\eta) + J_\eps^1(E_2;\Omega) \\
					& \quad
						+ \frac{\eps}{\delta_2-\delta_1} \int_{\Rd} K(z)\va{z} \de z\,
							\Ld\big( (E_1 \symdif E_2) \cap \Omega \big) \\
					& \quad
						+ \frac{\eps}{\eta} \int_{\Omega\cap\co{\Omega}_\eta}
								\int_{\Rd} K(z)\va{z}\chi_{\Omega \cap \co{\Omega}_2}(x+\eps z)\de z\de x.					
		\end{split}
		\end{equation}
	
	Let us assume provisionally that the claim holds.
	Then, by virtue of the Coarea Formula,
	there exists $t^\ast\in(0,1)$ such that \eqref{eq:gluing} holds
	for the superlevel $F\coloneqq \Set{w>t^\ast}$.
	Furthermore, we see that such set fulfils \ref{stm:glue1} and \ref{stm:glue2} as well.
	Indeed, $\phi$ is supported in $\Omega\cap \co{\Omega}_1$
	and it equals $1$ in $\Omega_2$,
	whence $F \cap \Omega_1 = E_1 \cap \Omega_1$
	and $F \cap \Omega_2 = E_2 \cap \Omega_2$.
	As for \ref{stm:glue2}, we remark that
	$x\in E_2\cap \co{F}\cap \Omega$ and $x\in \co{E_2}\cap F\cap \Omega$ respectively get the equalities
	$w(x)=(1-\phi(x))\chi_{E_1}(x)+\phi(x)\leq t^\ast < 1$
	and $w(x)=(1-\phi(x))\chi_{E_1}(x)>t^\ast>0$,
	which in turn entail $x\in \co{E_1}\cap E_2\cap\Omega$ and $x\in E_1 \cap \co{E}_2\cap\Omega$.
	
	So, to achieve the conclusion, we only need to validate \eqref{eq:w}.		
	To this aim, we explicit the integrand appearing in $J^1_K(w;\Omega)$:
	when $x,y\in \Omega$, we have
		\[\begin{split}
			\va{w(y) - w(x)} & \leq (1-\phi(y)) \va{u(y)-u(x)} + \phi(y)\va{v(y)-v(x)} \\
									& \quad + \va{\phi(y)-\phi(x)}\va{v(x)-u(x)} \\
									& \leq \chi_{\Set{\phi \neq 1}}(y) \va{u(y)-u(x)}
											+ \chi_{\Set{\phi \neq 0}}(y) \va{v(y)-v(x)} \\
									& \quad + \va{\phi(y)-\phi(x)}\va{v(x)-u(x)}.
		\end{split}\]
	Since
	$\Set{\phi \neq 0} \subset \Omega \cap \co{\Omega}_1$ and
	$\Set{\phi \neq 1} \subset \Omega \cap \co{\Omega}_2$,
	we see that
		\[\begin{split}
			2J_\eps^1(w;\Omega) & \leq
					\int_\Omega\int_{\Omega \cap \co{\Omega}_2} K_\eps(y-x) \va{u(y)-u(x)}\de y\de x \\
				& \quad  + \int_\Omega\int_{\Omega \cap \co{\Omega}_1} K_\eps(y-x)\va{v(y)-v(x)}\de y\de x \\
				& \quad + \int_\Omega\int_\Omega K_\eps(y-x)\va{\phi(y)-\phi(x)}\va{v(x)-u(x)}\de y\de x.
	\end{split}\]
	We bound each of three double integrals above separately.
	Firstly, we divide $\Omega$ in the disjoint regions
	$\Omega_\eta$ and $\Omega\cap\co{\Omega}_\eta$.
	We obtain
		\[\begin{split}
			\int_\Omega\int_{\Omega \cap \co{\Omega}_2} & K_\eps(y-x)\va{u(y)-u(x)}\de y\de x \\
					& = \int_{\Omega_\eta}\int_{\Omega \cap \co{\Omega}_2}
									K_\eps(y-x)\va{u(y)-u(x)}\de y\de x \\
					& \quad +\int_{\Omega\cap\co{\Omega}_\eta}\int_{\Omega \cap \co{\Omega}_2}
									K_\eps(y-x)\va{u(y)-u(x)}\de y\de x.
		\end{split}\]
	Note that
		\[
			\int_{\Omega_\eta}\int_{\Omega \cap \co{\Omega}_2}
					K_\eps(y-x)\va{u(y)-u(x)}\de y\de x
				\leq 2J^1_\eps(u;\Omega_\eta)
		\]
	and that
		\[\begin{split}
			\int_{\Omega\cap\co{\Omega}_\eta}\int_{\Omega \cap \co{\Omega}_2}
					& K_\eps(y-x)\va{u(y)-u(x)}\de y\de x \\
			& \leq \frac{2\eps}{\eta} \int_{\Omega\cap\co{\Omega}_\eta}\int_{\Omega \cap \co{\Omega}_2}
						K_\eps(y-x)\frac{\va{y-x}}{\eps}\de y\de x \\
			& =  \frac{2\eps}{\eta}
						\int_{\Omega\cap\co{\Omega}_\eta}\int_{\Rd}
						K(z)\va{z}\chi_{\Omega \cap \co{\Omega}_2}(x+\eps z)\de z\de x,
	\end{split}\]
	therefore,
	\begin{multline}\label{eq:Om2}
	\int_\Omega\int_{\Omega \cap \co{\Omega}_2} K_\eps(y-x)\va{u(y)-u(x)}\de y\de x \\
			\leq  2J^1_\eps(u;\Omega_\eta) 
					+\frac{2\eps}{\eta}
					\int_{\Omega\cap\co{\Omega}_\eta}\int_{\Rd}
					K(z)\va{z}\chi_{\Omega \cap \co{\Omega}_2}(x+\eps z)\de z\de x.
	\end{multline}
	
	Next, as for the second integral, we have
		\begin{equation}\label{eq:Om1}
			\int_\Omega\int_{\Omega \cap \co{\Omega}_1} K_\eps(y-x) \va{v(y)-v(x)}\de y\de x
				\leq 2J_\eps^1(v;\Omega).
		\end{equation}
			
	Finally, we observe that
		\[
			\va{\phi(y)-\phi(x)}\leq \frac{2}{\delta_2-\delta_1}\va{y-x}
		\]
	and hence
		\begin{multline}\label{eq:OmOm}
			\int_\Omega\int_\Omega K_\eps(y-x)\va{\phi(y)-\phi(x)}\va{v(x)-u(x)}\de y\de x
				 \\ \leq
				\frac{2 \eps}{\delta_2-\delta_1} \int_{\Rd} K(z)\va{z} \de z \int_\Omega \va{v(x)-u(x)}\de x
		\end{multline}
	Combining \eqref{eq:Om2}, \eqref{eq:Om1}, and \eqref{eq:OmOm}
	we retrieve \eqref{eq:w}.
	\end{proof}

Now, we are in position to prove a chain of inequalities
which yields Lemma \ref{stm:sigmaK} as a corollary:

	\begin{lemma}\label{stm:sK-s''K}
		For all $\hat{p}\in\Sdmu$, $\sigma_K(\hat{p}) \leq \sigma''_K(\hat{p})\leq \sigma'_K(\hat{p})$.
	\end{lemma}
	\begin{proof}
	We begin with the inequality $\sigma_K \leq \sigma''_K$.
	Let $E_\eps$ be measurable subsets of $\Rd$ such that
	$E_\eps \cap \co{B} = H_{\hat{p}}\cap \co{B}$ for all $\epsilon>0$,
	and suppose that $E_\eps \to H_{\hat{p}}$ in $L^1(B)$.
	By Theorem \ref{stm:piani},
	we have that
	\begin{align*}
	0 & \leq J_\epsilon(E_\eps; B) - J_\epsilon(H_{\hat{p}}; B)  \\
	& =  J^1_\epsilon(E_\eps; B) - J^1_\epsilon(H_{\hat{p}}; B)
	+ J^2_\epsilon(E_\eps; B) - J^2_\epsilon(H_{\hat{p}}; B).
	\end{align*}
	If in addition we require that $E_\eps \symdif H_{\hat{p}} \subset B_{1-\delta}$,
	a direct computation shows
		\[\begin{split}
			J^2_\epsilon(E_\eps; B) & - J^2_\epsilon(H_{\hat{p}}; B) \\
				& = L_\eps( E_\eps \cap \co{H}_{\hat{p}} ; \co{H}_{\hat{p}} \cap \co{B} ) 
					+ L_\eps( \co{E}_\eps \cap H_{\hat{p}} ; H_{\hat{p}} \cap \co{B} ) \\
				& \quad		- L_\eps( \co{E}_\eps \cap H_{\hat{p}} ; \co{H}_{\hat{p}} \cap \co{B} )
					- L_\eps( E_\eps \cap \co{H}_{\hat{p}} ; H_{\hat{p}} \cap \co{B} ),
		\end{split}\]
	and so
		\[
			\va{J^2_\epsilon(E_\eps; B) - J^2_\epsilon(H_{\hat{p}}; B)}
				\leq L_\eps (E_\eps \symdif H_{\hat{p}} ; \co{B} \big).
		\]
	Since $E_\eps \symdif H_{\hat{p}} \subset B_{1-\delta}$,
	we notice that $\va{y-x}\geq\delta$ if $y\in \co{B}$ and $x\in E_\eps \symdif H_{\hat{p}}$,
	and thus
		\begin{align*}
			\frac{1}{\eps}\va{J^2_\epsilon(E_\eps; B) - J^2_\epsilon(H_{\hat{p}}; B)}
			& \leq \frac{1}{\delta} \int_{ E_\eps \symdif H_{\hat{p}} } \int_{ \co{B} }
					K_\eps (y-x) \frac{\va{y-x}}{\eps}\de y \de x  \\
			& \leq \frac{1}{\delta} \Ld( E_\eps \symdif H_{\hat{p}} )\int_{ \Rd } K (z) \va{z}\de z.
		\end{align*}
	We deduce that
		\[ 
		\lim_{\epsilon\to 0^+} \frac{1}{\eps} \va{ J^2_\epsilon(E_\eps; B) - J^2_\epsilon(H_{\hat{p}}; B) } = 0,
		\]
	whence 
		\begin{align*}
		0 & \leq \liminf_{\eps \to 0^+}   \frac{1}{\eps}\left[ J_\epsilon(E_\eps; B) - J_\epsilon(H_{\hat{p}}; B) \right]  \\
				& =  \liminf_{\eps \to 0^+}   \frac{1}{\eps}\left[ J^1_\epsilon(E_\eps; B) - J^1_\epsilon(H_{\hat{p}}; B) \right].
		\end{align*}
	Recalling \eqref{eq:sigmaK-P} and the definition of $\sigma''_K$,
	we conclude that $\sigma_K(\hat{p}) \leq \sigma''_K(\hat{p})$.

	We now focus on the inequality $\sigma''_K\leq \sigma'_K$.
	We pick a family $\Set{E_\epsilon}$ that converges to $H_{\hat{p}}$ in $L^1(B)$,
	and we assume, without loss of generality,
	that $J_{\epsilon}^1(E_{\epsilon};B)$ is finite.   
	We exploit Lemma~\ref{stm:gluing}
	taking as data, for any $\epsilon$, $E_\epsilon$ and $H_{\hat{p}}$.
	We obtain a family $\Set{F_\epsilon}$
	that converges to $H_{\hat{p}}$ in $L^1(B)$
	and that fulfils $F_\eps \symdif H_{\hat{p}} \subset B_{1-\delta}$.
	Moreover, for any $\eta\in\big( 0, (1-\delta)/2 \big)$
	and for $\delta_1 = \delta$ and $\delta_2 = \delta+\eta$,
	formula \eqref{eq:gluing} yields
		\[\begin{split}
			\frac{1}{\eps}J_\eps^1(F_\eps;B)
				& \leq \frac{1}{\eps} J^1_\eps(H_{\hat{p}} ; B\cap \co{B}_{1-\delta-2\eta} )
					+ \frac{1}{\eps} J_\eps^1(E_\eps;B) \\
				& \quad	+ \frac{1}{\eta} \int_{\Rd} K(z)\va{z} \de z\,
										\Ld\big( (H_{\hat{p}}  \symdif E_\eps) \cap B \big) \\
				& \quad + \frac{1}{\eta} \int_{B_{1-\delta-2\eta}}
									\int_{\Rd} K(z)\va{z}\chi_{B \cap \co{B}_{1-\delta-\eta}}(x+\eps z)\de z\de x.					
		\end{split}\]
	By utilizing \eqref{eq:pointlim} we find
		\[
			\lim_{\epsilon\to0}\frac{1}{\eps}J_\eps^1(H_{\hat{p}} ; B\cap \co{B}_{1-\delta-2\eta})
			= J_0(H_{\hat{p}} ; B\cap \co{B}_{1-\delta-2\eta}),\]
	and thus, in the limit $\epsilon\to0^+$, we infer
		\[\begin{split}
			\sigma''_K(\hat{p}) & \leq
						\liminf_{\epsilon\to0^+}\frac{1}{\eps\omega_{d-1}}J_\eps^1(F_\eps;B) \\
						& \leq \frac{1}{\omega_{d-1}}J_0(H_{\hat{p}} ; B\cap \co{B}_{1-\delta-2\eta})
							+\liminf_{\epsilon\to0^+}\frac{1}{\eps\omega_{d-1}}J_\eps^1(E_\eps;B),
		\end{split}\]
	In view of the arbitrariness of the family $\Set{E_\eps}$,
	we achieve the conclusion by letting $\eta$ and $\delta$ vanish.
\end{proof}
	\section{$\Gamma$-lower limit}\label{sec:PerK-Gliminf}
		We conclude the proof of Theorem \ref{stm:PerK-Gconv}
by validating the $\Gamma$-lower limit inequality.
We remind that the method we apply was introduced in \cite{FM}
and that it was applied in \cite{ADM,AB2}
to establish $\Gamma$-convergence results that are close to ours.	

Our task is showing that,
for all measurable $E$ and for all families $\Set{E_\eps}$
such $E_\epsilon \to E$ in $L^1_\loc(\Rd)$ as $\epsilon\to0^+$,
we have
	\begin{equation}\label{eq:PerK-Gliminf-bis}
	J_0(E;\Omega) \leq \liminf_{\epsilon\to 0^+} \frac{1}{\eps} J_\epsilon^1(E_\epsilon;\Omega).
	\end{equation}

Firstly, we notice that \eqref{eq:PerK-Gliminf-bis} holds trivially
when the right-hand side is not finite,
thus we may suppose that
	\[
		\liminf_{\epsilon\to 0^+} \frac{1}{\eps} J_\epsilon^1(E_\epsilon;\Omega) <+\infty.
	\]
Besides, up to extraction of a subsequence,
we may assume that the the lower limit is a limit,
and hence the ratios $\Set{\eps^{-1}J_\eps^1(E_\eps;\Omega)}$ are uniformly bounded in $\epsilon$.
In view of the compactness criterion on page \pageref{stm:JK1-cpt},
the bound yields regularity for the limit set $E$,
which turns out to be a finite perimeter set in $\Omega$.

We introduce the densities
\[
f_\eps(x)  \coloneqq
\begin{cases}
\displaystyle{\frac{1}{2\epsilon}\int_{{\co{E}_\eps}\cap\Omega}K_\eps(y-x)\de y}
& 	\text{if } x\in \Omega\cap E_\eps, \\
\\
\displaystyle{\frac{1}{2\epsilon}\int_{E_\eps\cap\Omega}K_\eps(y-x)\de y}
& \text{if } x\in \Omega\cap\co{E}_\eps,
\end{cases}
\]
and the measures
\[
\mu_\eps \coloneqq  f_\eps\,\call{L}^d \llcorner \Omega.
\]

For any $\epsilon>0$,
the total variation  of $\mu_\epsilon$ in $\Omega$ equals $\eps^{-1} J_\eps^1(E_\eps;\Omega)$,
thus $\Set{\mu_\eps}$ is a family of positive measures
with uniformly bounded total variation in $\Omega$.
It follows from Theorem \ref{stm:riesz} and Banach-Alaoglu's Theorem
that there exists a finite positive measure $\mu$ on $\Omega$
such that (up to subsequences)
$\mu_\eps \weakstar \mu$ as $\epsilon \to 0^+$, and,
by the lower semicontinuity of the total variation,
\[
\norm{\mu} \leq \liminf_{\epsilon\to0} \frac{1}{\eps}J_\epsilon^1(E_\eps;\Omega).
\]
We now see that \eqref{eq:PerK-Gliminf-bis} holds true if
\begin{equation}\label{eq:tvineq}
\norm{\mu}\geq J_0(E;\Omega).
\end{equation}

In order to establish the previous inequality,
we recall that we know that $E$ is a finite perimeter set in $\Omega$,
so for a point $x\in\partial^\ast E \cap \Omega$
we can consider the Radon-Nikodym derivative
\[
\der{\mu}{\nu} (x) \coloneqq  \lim_{r\to 0} \frac{\mu(B(x,r))}{\omega_{d-1}r^{d-1}},
\]
where $\nu \coloneqq \call{H}^{d-1} \llcorner \partial^\ast E$.
Since the sequence $\mu_\epsilon$ weakly-$\ast$ converges to $\mu$,
by Proposition \ref{stm:foliations},
for all $r>0$ but at most a countable set $N$,
we have $\nu(B(x,r))=\lim_{\epsilon\to0}\nu_\eps(B(x,r))$.
Thus,
\begin{align*}
\der{\mu}{\nu}(x)
& = \lim_{r\to 0, r\notin N}\left[\lim_{\epsilon\to0}\frac{\mu_\eps(B(x,r))}{\omega_{d-1}r^{d-1}}\right] \\
& = \lim_{\ell\to+\infty}\frac{\mu_{\epsilon_\ell}(B(x,r_\ell))}{\omega_{d-1}r^{d-1}_\ell}
\end{align*}
for suitable subsequences $\Set{\eps_\ell}$ and $\Set{r_\ell}$ satisfying
\[
\lim_{\ell\to+\infty}r_\ell = \lim_{\ell\to+\infty}\frac{\eps_\ell}{r_\ell}=0.
\]
Rewriting the last equality explicitly, we have
\[\begin{split}
\der{\nu}{\mu}(x) = \limseq{\ell}\frac{1}{2 \omega_{d-1} \eps_\ell r^{d-1}_\ell}
&	\left[\int_{E_{\epsilon_\ell} \cap B(x,r_\ell) \cap \Omega}
\int_{\co{E_{\epsilon_\ell}}\cap\Omega}K_{\eps_\ell}(y-x)\de y\de x \right.\\ 
&	\left. +\int_{\co{E_{\epsilon_\ell}} \cap B(x,r_\ell)\cap \Omega}
\int_{E_{\eps_\ell}\cap\Omega}K_{\eps_\ell}(y-x)\de y \de x\right],
\end{split}\]
and hence
\[\begin{split}
\der{\mu}{\nu}(x)
\geq & \limsup_{\ell\to+\infty}\frac{1}{\omega_{d-1} \eps_\ell r^{d-1}_\ell}
J^1_{\epsilon_\ell}(E_{\epsilon_\ell};B(x,r_\ell)\cap \Omega) \\
= & \limsup_{\ell\to+\infty}\frac{1}{\omega_{d-1} \eps_\ell r^{d-1}_\ell}
J^1_{\epsilon_\ell}(E_{\epsilon_\ell};B(x,r_\ell)),
\end{split}\]
because $B(x,r_\ell)\subset \Omega$ when $\ell$ is large enough.
A change of variables yields
\[
\der{\mu}{\nu}(x)
\geq \limsup_{\ell\to+\infty}\frac{r_\ell}{\omega_{d-1} \eps_\ell}
J^1_{\frac{\epsilon_\ell}{r_\ell}}\left( \frac{E_{\epsilon_\ell}-x}{r_\ell}; B \right),
\]

In view of position \eqref{eq:sigma'K}, it is clear that
\[
\der{\mu}{\nu}(x) \geq \sigma'_K(\hat{n}(x)),
\]
which gets
\[
\norm{\mu} \geq 
\int_{\partial^\ast E \cap \Omega} \sigma'_K (\hat{n}(x)) \de \call{H}^{d-1}(x).
\]
The characterisation of $\sigma_K$ provided by Lemma \ref{stm:sigmaK} yields the conclusion.

\chapter{Convergence of evolutions by~rescaled~nonlocal~curvatures}\label{ch:nlc}
		In Chapter \ref{ch:judith} we proved that,
by a suitable rescaling, we can recover a local perimeter functional
as the $\Gamma$-limit of nonlocal perimeters.
We also saw in Section \ref{sec:nlc} that
the functional
	\begin{equation*}
	H_K(E,x) = -\lim_{r\to 0^+}\int_{\co{B(x,r)}}K(y-x)\left(\chi_E(y)-\chi_{\co{E}}(y)\right)\de y
	\end{equation*}
is the geometric first variation of the nonlocal perimeter $\PerK$.
Relying on the paper \cite{CP},
in the current chapter we show that
we can produce a localisation effect for the nonlocal curvature
by means of the same procedure that we applied to the nonlocal perimeter:
for any $\epsilon>0$ we set
	\begin{equation*}
		K_\epsilon(x) \coloneqq \frac{1}{\eps^d}K\left(\frac{x}{\eps}\right),
	\end{equation*}
we define
	\begin{equation}\label{eq:resccurv}
		H_\epsilon(E,x) \coloneqq H_{K_\eps}(E,x)
	\end{equation}
for a set $E\subset\Rd$ and $x\in\partial E$,
and then we let $\eps$ tend to $0$.
As we discuss in Section \ref{sec:K},
the limit of the family $\Set{\epsilon^{-1} H_\eps}$ agrees with the analysis in Chapter \ref{ch:judith},
in the sense that the limit curvature, which is local,
is the first variation of the surface energy
	\[
	J_0(E) = \int_{\partial^\ast E} \sigma_K(\hat{n}(x)) \de \Hdmu(x)
	\]
obtained as the $\Gamma$-limit of the rescaled nonlocal perimeters.

The study of the asymptotics of rescaled nonlocal curvatures is expounded
in Section \ref{sec:convcurv}.
Then, we tackle a second issue:
is the relationship between curvatures carried over to the solutions of the related geometric motions?
More precisely, we investigate whether the solutions of
	\[
		\partial_t x(t) \cdot \hat{n} = - \frac{1}{\eps}H_\epsilon(E(t), x(t)),
	\]
converge, as $\epsilon$ vanishes, to the solution of
	\[
		\partial_t x(t) \cdot \hat{n} = - H_0(E(t), x(t)),
	\]
where $t\mapsto E(t)$ is an evolution of sets in $\Rd$,
$\hat{n}$ is the outer unit normal
to $\partial E(t)$ at the point $x(t)$,
and $H_0$ is the first variation of $J_0$.

To answer, we firstly explain in Section \ref{sec:geomevol}
what kind of solutions we take into consideration,
and we recall the well-posedness results for local and nonlocal curvature flows
that are available in the framework of level set formulations.
By this method, the geometric evolutions are cast into degenerate parabolic equations,
to which viscosity theory provides an adequate definition of solution.
Then, we prove that solutions in this class to the motions by rescaled nonlocal curvatures
locally uniformly converge to the viscosity solution of the limit curvature flow.
This is the main result of the chapter,
which we establish in Section \ref{sec:convflows}
by resorting to the set-theoretic approach of geometric barriers,
formerly introduced by De Giorgi \cite{D2}.
We use barriers to transfer the geometric information
given by the convergence result of Section \ref{sec:convcurv}
to viscosity solutions.
In this respect, comparison results between barriers and level set flows \cites{BNo,CDNV} are crucial.
	\section{Preliminaries}\label{sec:K}
		Since we want to mimic the approach of Chapter \ref{ch:judith},
as an initial step we compute the candidate limit curvature,
that is, the first variation of the functional \eqref{eq:J0} on page \pageref{eq:J0}.
We premise the assumptions on the kernel $K$
that we use throughout the chapter.

As before, $K\colon \Rd \to [0,+\infty)$ is a measurable function
such that
	\begin{equation}\label{eq:K-even}
		K(y) = K(-y) \quad\text{for a.e. } y
	\end{equation}
and
	\begin{equation}\label{eq:W11}
		K\in W^{1,1}(\co{B(0,r)}) \quad\text{for all } r>0.
	\end{equation}
This allows both $K$ and $\nabla K$ to be singular in $0$ 
and, at the same time, implies convergence of their integrals at infinity.
Moreover, we assume  that 		
	\begin{equation}\label{eq:imbert2}
		K(y),\ \va{y} \va{\nabla K(y)} 
		\in L^1(Q_\lambda(e))\quad \text{for all } e \in \Sdmu \text{ and } \lambda>0,
	\end{equation}
where, for any  $e \in \Sdmu$ and $\lambda>0$,
	\[
		Q_\lambda(e) \coloneqq
			\Set{y\in \Rd : \va{y\cdot e}\leq \frac{\lambda}{2} \va{\pi_{e^\perp}(y)}^2 }.
	\]
(see again Figure \ref{fig:quadriche} on page \pageref{fig:quadriche}).
These requirements are in line with the ones in Section \ref{sec:nlc},
and, in particular, they ensure that
sets with $C^{1,1}$ boundary have finite nonlocal curvature.
 
To the purpose of proving a nonlocal-to-local converge result
that is uniform w.r.t. the point in which the curvatures are evaluated,
we need to describe the behaviour of $K$ in quantitative terms.
So, we suppose that
	\begin{equation}\label{eq:sing-orig}
	\lim_{r\to 0^+} r\int_{\co{B(0,r)}} K(y)\de y = 0.
	\end{equation}
Besides, we want our analysis to encompass anisotropic kernels.
Still, we have to admit some control
on the mass of $K$ in $Q_\lambda(e)$, uniformly in $e$.
We therefore assume that for all $\lambda>0$ there exists $a_\lambda>0$
such that for all $e \in \Sdmu$
	\begin{equation}\label{eq:massa-parabole}
	\int_{Q_\lambda(e)} K(y)\de y \leq a_\lambda.
	\end{equation}
In addition, we require that 
there exist $a_0,b_0>0$ such that for all $e \in \Sdmu$
	\begin{eqnarray}
	\limsup_{\lambda\to 0^+}
		\frac{1}{\lambda}\int_{Q_\lambda(e)} K(y)\de y \leq a_0, \label{eq:imbert3} \\
	\limsup_{\lambda\to 0^+}
		\frac{1}{\lambda}\int_{Q_\lambda(e)} \va{\nabla K(y)}\va{y}\de y \leq b_0 \label{eq:imbert4}.
\end{eqnarray}
We assume as well that for all $e \in \Sdmu$
	\begin{equation}\label{eq:imbert5}
		\lim_{\lambda\to+\infty} \frac{1}{\lambda}\int_{Q_\lambda(e)} K(y)\de y = 0.
	\end{equation}	

Lastly, we let $K$ be bounded above by a fractional  kernel away from the origin,
that is, 
	\begin{equation}\label{eq:frac-decay}
		K(y)\leq \frac{m}{\va{y}^{d+1+s}} \quad \text{if } y\in \co{B(0,1)}
	\end{equation}
for some $m>0$ and $s\in(0,1)$.
This hypothesis entails that, for all $\alpha < s$,
	\begin{equation}\label{eq:imbert1}
	\lim_{r\to +\infty} r^{1+\alpha}\int_{\co{B(0,r)}} K(y)\de y = 0.
	\end{equation}
Actually, most of the results in the sequel hold
even if the weaker assumption \eqref{eq:imbert1} replaces  \eqref{eq:frac-decay}.

Let us describe two classes of singular kernels
that satisfy \eqref{eq:K-even} -- \eqref{eq:frac-decay}.

\begin{exa}
If $K\colon \Rd \to [0,+\infty)$ is even and
there exist $m,\mu>0$ and $s,\sigma\in(0,1)$ such that 
	\[K(y), \va{y}\!\va{\nabla K(y)} \leq \frac{\mu}{\va{y}^{d+\sigma}}
		\quad \text{for all } y\in B(0,1)
	\]
and that
	\[K(y), \va{y}\!\va{\nabla K(y)} \leq \frac{m}{\va{y}^{d+1+s}}
		\quad \text{for all } y\in \co{B(0,1)},\]
then, $K$ fulfils all the previous requirements. 
 
Fractional kernels with exponential decay at infinity, i.e.
even functions $K\colon \Rd \to [0,+\infty)$  
for which there exist $m, \mu>0$ and $s\in(0,1)$ such that 
	\[
		K(y),  \va{y}\!\va{\nabla K(y)} \leq \frac{\mu e^{-m \va{y}}}{\va{y}^{d+s}}\quad  \forall y\in\Rd,
	\]
also fit in our setting.
\end{exa}

We now derive the candidate limit curvature starting
from the anisotropic surface energy $J_0$.
In this case, we suppose that $K$ fits in the framework
of the $\Gamma$-convergence statement
concerning perimeters, i.e. Theorem \ref{stm:JK-Gconv}.
As above, $p^\perp$ is the hyperplane of vectors  orthogonal to $p\in\Rdmz$
and $H_{\hat{p}} \coloneqq \Set{ x\in\Rd : x\cdot \hat{p} > 0 }$.

	\begin{lemma}[Derivatives of the anisotropy]\label{stm:deriv-sigmaK}
		Let $K\colon \Rd \to [0,+\infty)$ be a function
		such the assumptions \eqref{eq:K-even} -- \eqref{eq:imbert3} hold,
		and suppose that
			\[
				c_K \coloneqq \frac{1}{2} \int_{\Rd} K(z)\va{z} \de y <+\infty.
			\]
		 Let $\sigma_K$ be the norm in \eqref{eq:sigmaK}, i.e.
			\[
				\sigma_K(p) \coloneqq \frac{1}{2}\int_{\Rd} K(z)\va{z\cdot p}\de z,
				\qquad\text{for } p\in\Rd.
			\]
		Then, $\sigma_K\in C^2(\Rdmz)$ and
			\[
				\nabla \sigma_K (p) = \int_{H_{\hat{p}}} K(z) z \de z,
				\qquad
				\nabla^2 \sigma_K(p) = \frac{1}{\va{p}} \int_{\hat{p}^\perp} K(z) z \otimes z \de \Hdmu(z).
			\]
	\end{lemma}
	
	The calculation of the gradient is immediate;
	as for the Hessian, we preliminarily recall a result about absolute continuity on lines
	of Sobolev functions.
	We refer to the monograph \cite{HKS} for the details.
	
	\begin{thm}\label{stm:acl}
		Let $\Omega\subset\Rd$ be an open set.
		For any $p\in[1,+\infty)$,
		$u: \Omega\to \R$ belongs to the Sobolev space $W^{1,p}(\Omega)$
		if and only if it belongs to $L^p(\Omega)$
		and the following property holds:
		for any $e\in\Sdmu$
		there exists $N_e\subset e^\perp$
		such that $\Hdmu(N_e)=0$
		and for all $z\in e^\perp \cap \co{N_e}$
		the function $I\ni t\mapsto u(z+te)$ is absolutely continuous
		on any compact interval $I$
		such that $z+te\in \Omega$ when $t\in I$.
		
		Further, for any $e\in\Sdmu$
		and a.e. $y\in\Omega$,
		the classical directional derivative $\partial_e u$
		exists and it coincides with $\nabla u (y)\cdot e$.
	\end{thm}

In the light of the previous characterisation,
when \eqref{eq:W11} holds, the kernel $K$ is absolutely continuous on lines
in $\co{B(0,r)}$ for all $r>0$. 
We use this property to establish the following: 

\begin{lemma}\label{stm:MK}
	For all $e\in\Sdmu$,
		\begin{equation}\label{eq:Kz2}
			\int_{e^\perp}K(z)\va{z}^2\de\Hdmu(z)\leq a_0,
		\end{equation}
	$a_0$ being as in \eqref{eq:imbert3},
	and there holds
		\begin{equation}\label{eq:decad-Kz2}
			\lim_{r\to+\infty} r^{\beta}\int_{e^\perp \cap \co{B(0,r)}} K(z)\va{z}^2\de\Hdmu(z)=0
			\quad \text{for all } \beta<s.
		\end{equation}
	\end{lemma}
	\begin{proof}
	By Theorem \ref{stm:acl},	
	for any $e\in\Sdmu$ and any $\ell\in\N$,
	there exists a $\Hdmu$-negligible
	$N_\ell \subset \Set{z\in e^\perp : \ell \va{z}\geq1}$
	such that, for all $z\in e^\perp\cap \co{N_\ell}$ with $\ell\va{z}\geq1$,
	the function $t\mapsto K(z+te)$ is absolutely continuous
	when $t$  belongs to closed, bounded intervals. 
	By the arbitrariness of $\ell\in\N$,
	we conclude that for $\Hdmu$-a.e. $z\in e^\perp$,
	$[a,b] \ni t\mapsto K(z+te)$ is absolutely continuous for any $a,b\in\R$.
	
	Hence, by the Mean Value Theorem,
	for $\Hdmu$-almost every $z\in e^\perp$
	we find
		\begin{equation}\label{eq:mean-val}
			\lim_{\lambda\to0^+} \frac{1}{\lambda}
				\int_{-\frac{\lambda}{2}\va{z}^2}^{\frac{\lambda}{2}\va{z}^2}
					K(z+te)\de t = K(z)\va{z}^2.
		\end{equation}
	By definition of $Q_\lambda(e)$,
		\[
		\frac{1}{\lambda}\int_{e^\perp}
			\int_{-\frac{\lambda}{2}\va{z}^2}^{\frac{\lambda}{2}\va{z}^2}
				K(z+te)\de t\de\Hdmu(z) =
			\frac{1}{\lambda}\int_{Q_\lambda(e)} K(y) \de y,			
		\]
	and the right hand side is finite for any $\lambda>0$,
	thanks to \eqref{eq:imbert2}.
	By appealing to \eqref{eq:imbert3} and \eqref{eq:mean-val} ,
	we can take the limit $\lambda\to 0^+$
	on both sides of the equality, and this yields \eqref{eq:Kz2}.
	
	Estimate \eqref{eq:decad-Kz2} is an easy consequence of assumption \eqref{eq:frac-decay}.
	\end{proof}
	
Now we are ready to prove Lemma \ref{stm:deriv-sigmaK}.
	\begin{proof}
		By Lebesgue's Theorem, we see that $\sigma_K \in C^1(\Rdmz)$, with
			\[
				\nabla \sigma_K (p) = \frac{1}{2} \int_{\Rd} K(z) z \frac{z\cdot p}{\va{z\cdot p}} \de z
												= \int_{H_{\hat{p}}} K(z) z \de z.
			\]
		
		Let now $i,j\in \Set{1,\dots, d}$ and fix $p\in\Rdmz$.
		We have
			\[
			 \partial_{i,j}^2 \sigma_K(p) \coloneqq
			 		\lim_{h\to 0^+} \frac{1}{h}
			 			\left[ \int_{H_{\widehat{p+h e_j}}} K(z)z_i \de z - \int_{H_{\hat{p}}} K(z)z_i \de z\right],			
			\]
		where $z_i \coloneqq z \cdot e_i$ for all $z\in \Rd$ and all $i$.
		We rewrite the term between square brackets as
			\[
			 	\int_{\hat{p}^\perp} \int_{\frac{h z_j}{\va{p} + h \hat{p}_j}}^{0}
			 		K(z + t\hat{p})(z_i + t \hat{p}_i) \de t \de \Hdmu(z),
			\]
		and, reasoning as in the proof of the previous Lemma,
		we apply the Mean Value Theorem to get
			\[
				\partial_{i,j}^2 \sigma_K(p) \coloneqq
				\lim_{h\to 0^+} \frac{1}{\va{p} + h \hat{p}_j} \int_{\hat{p}^\perp} 
						K(z + t_h\hat{p})(z_i + t_h \hat{p}_i) z_j \de \Hdmu(z),
			\]
		with $t_h$ between $0$ and $h z_j/(\va{p} + h \hat{p}_j)$.
		Thanks to \eqref{eq:Kz2},
		we can take the limit $h \to 0^+$ and
		we conclude the computation of the Hessian.
		Next, we show that the latter is continuous.
		
		For all $e\in\Sdmu$, set 
			\[
				M_K(e) \coloneqq \int_{e^\perp} K(z) z \otimes z \de \Hdmu(z).
			\]
		Fix a unit vector $e$ and consider a sequence of rotations $\Set{R_\ell}$
		such that $R_\ell\to \id$. We have
			\[\begin{split}
				\va{M_K(R_\ell e)-M_K(e)} & =
					\va{\int_{e^\perp} K(R_\ell z)\, R_\ell z\otimes R_\ell z\de\Hdmu
						- \int_{e^\perp} K(z) z\otimes z\de\Hdmu} \\
				& \leq \va{\int_{e^\perp} K(R_\ell z)\, [R_\ell z\otimes R_\ell z - z\otimes z]\de\Hdmu} \\
				& \quad +\va{\int_{e^\perp} [K(R_\ell z)-K(z)]z\otimes z\de\Hdmu}.	
			\end{split}\]
		Since $K\in L^1(\co{B(0,r)})$ for all $r>0$, it holds
				\[
				\limseq{\ell}\norm{K\circ R_\ell - K}_{L^1(\co{B(0,r)})}=0,
				\]
		whence $K(R_\ell z)\to K(z)$ for $\Hdmu$-a.e. $z\in e^\perp$.
		This, together with \eqref{eq:Kz2},
		gets that the upper bound on $\va{M_K(R_\ell e)-M_K(e)}$
		vanishes as $\ell \to +\infty$.
		Therefore, $\nabla^2 \sigma_K(p) = \va{p}^{-1} M_K(\hat{p})$ is continuous on $\Rdmz$.
	\end{proof}

Let $\Sigma$ a $C^2$ hypersurface in $\Rd$.
For $x\in \Rd$,
we define the following \emph{anisotropic mean curvature}\index{curvature!anisotropic mean --}
functional:
	\[
		H_0(\Sigma,x) \coloneqq - \div \big(\nabla \sigma_K (\hat{n}(x))\big),
	\]
with $\sigma_K$ as above and
$\hat{n}(x)$ is the outer unit normal to $\Sigma$ at $x$.
Note that this is the geometric first variation of the surface energy $J_0$ in \eqref{eq:J0}.
Since $\Sigma$ is smooth,
we can find an open neighbourhood of $x$ and $\phi\in C^2(\Rd)$ such that
$\Sigma \cap U = \Set{y\in U :\phi(y) = 0}$,
$\nabla\phi(x)\neq 0$.
Let us admit that $\nabla\phi(x) / \va{\nabla\phi(x)}$ is the outer normal at $x$.
Then, we find
	\begin{equation}
		H_0(\Sigma,x) =
			-\frac{1}{\va{\nabla \phi(x)}} \mathrm{tr}\left(M_K (\widehat{\nabla \phi(x)}) \nabla^2 \phi(x)\right),
	\end{equation}
where $\mathrm{tr}$ is the trace operator and
	\begin{equation}
	M_K(e) \coloneqq \int_{e^\perp}K(z)z\otimes z\de \Hdmu(z) \qquad \text{for all } e \in \Sdmu 
	\end{equation}
because, if $p\in \Rdmz$, by Lemma \ref{stm:deriv-sigmaK},
$M_K(\hat{p}) = \nabla^2\sigma_K (\hat{p})$.

A third, useful expression of the local curvature functional is
	\begin{equation}\label{eq:H0}
	H_0(\Sigma,x) 
		=  - \frac{1}{\va{\nabla\phi(x)}}\int_{\nabla\phi(x)^\perp}
							K(z)\nabla^2 \phi(x)z\cdot 	z\de\Hdmu(z). 
	\end{equation}

\begin{rmk}		
	Let $\Sigma$ be a smooth hypersurface 
	whose outer unit normal at a given point $x$ is $\hat n$,
	and consider the map $T(y)\coloneqq Ry + z$,
	where $R$ is an orthogonal matrix and $z\in\Rd$.
	Then, using  \eqref{eq:H0} it is easy to check that it  holds
	\begin{equation}\label{eq:l-cng-var}
	H_0(\Sigma,x)=\tilde H_0(T(\Sigma),T(x)),
	\end{equation}
	where $\tilde H_0$ is the anisotropic mean curvature functional
	associated with the kernel $\tilde K\coloneqq K \circ R^\mathrm{t}$.
	
	To prove our claim, we observe that
	if $\Sigma = \Set{y\in\Rd : \phi(y)=0}$
	for some smooth $\phi\colon \Rd\to\R$,
	then $T(\Sigma)=\Set{y\in\Rd : \psi(y)=0}$
	with $\psi(y)\coloneqq \phi( R^\trasp(y-x))$.
	We have
	\[\nabla\psi(T(y))=R\nabla\phi(y)
	\quad\text{and}\quad
	\nabla^2\psi(T(y))=R\nabla^2\phi(y)R^\trasp,
	\]
	and, therefore,
	\[\begin{split}
	\tilde H_0(T(\Sigma),T(x)) = &
	-\frac{1}{\va{R\nabla\phi(x)}}\int_{R(\hat n^\perp)}
	\tilde K(z)\left(R\nabla^2\phi(x)R^\trasp\right)z\cdot z\de\Hdmu(z) \\
	= & -\frac{1}{\va{\nabla\phi(x)}}\int_{\hat n^\perp}
	K(z)\nabla^2\phi(x)z\cdot z\de\Hdmu(z).
	\end{split}\]
\end{rmk}

\begin{rmk}[Connection with standard mean curvature] \label{remarkradial} 
	When $K$ is  radial, that is $K(x)=\bar K(\va{x})$
	for some $\bar K\colon [0,+\infty)\to [0,+\infty)$, then 
	$H_0$ coincides with a multiple of the standard mean curvature.	
	
	To see this, let $\Sigma$ be a $C^2$ hypersurface
	such that $0\in\Sigma$ and
	$\Sigma\cap U = \Set{y\in U : \phi(y)=0}$
	for some  neighbourhood $U$ of $0$ and
	some smooth function $\phi\colon U\to \R$.
	We suppose also that $\nabla\phi(0)\neq 0$ and
	that the outer unit normal to $\Sigma$ at $0$ is $e_d$.
	We recall the expression
	of the  mean curvature $H$ of $\Sigma$ at $0$:
	\[ H(\Sigma,0) =
	-\frac{1}{\omega_{d-1}\va{\nabla\phi(0)}}
	\int_{\mathbb{S}^{d-2}} \nabla^2\phi(0)e\cdot e\, \de\call{H}^{d-2}(e),
	\]
	with $\omega_{d-1}$ the $(d-1)$-dimensional Lebesgue measure
	of the unit ball in $\R^{d-1}$.
	
	If  $K(x)=\bar K(\va{x})$, formula \eqref{eq:Kz2} gets
		\[ c_K = \int_0^{+\infty} r^d \bar K(r)\de r <+\infty,\]
	and, consequently, we have
		\[
			H_0(\Sigma,0) = -
			\frac{c_K}{\va{\nabla\phi(0)}}
					\int_{\mathbb{S}^{d-2}}\nabla\phi(0)e\cdot e\,\de\call{H}^{d-2}(e) 
				=  \omega_{d-1}\,c_K H(\Sigma,0).
	\]
\end{rmk}
	\section{Convergence of the rescaled nonlocal curvatures}\label{sec:convcurv}
		This section is devoted to the proof
of the uniform convergence of the rescaled nonlocal curvatures in \eqref{eq:resccurv}
to the local functional $H_0$.
This is accomplished in two steps:
first, in Proposition \ref{stm:conv-rate},
we establish pointwise convergence of the curvatures,
with a precise estimate of the error, 
and then, in Theorem \ref{stm:main1},
we show that it is possible to make the error estimate uniform
when the hypersurfaces at stake are compact.

Let us introduce some notation.
For $e\in\Sdmu$, $x\in\Rd$ and $\delta>0$,
we let
	\begin{equation*}\label{cilindro}
		C_e(x,\delta)\coloneqq
			\Set{y\in\Rd : y = x+z+te,\text{ with } z\in e^\perp\cap B(0,\delta),t\in(-\delta,\delta)}.
	\end{equation*}
When $E\subset \Rd$ is a set with boundary $\Sigma \coloneqq \partial E$ of class $C^2$,
then for all $x\in \Sigma$, there exist a open neighbourhood $U$ of $x$
and  $\phi\in C^2(\Rd)$ such that
	\begin{equation*}
	\Sigma\cap U = \{y \in U : \phi(y)=0\},
	\quad
	E \cap U =\{y \in U : \phi(y)>0\},
	\end{equation*}
and $\nabla\phi(y)\neq 0$ for all $y\in\Sigma\cap U$.
We use the symbol $\hat{n}\coloneqq\hat{n}(x)$
for the \emph{outer} unit normal to $\Sigma$ at $x$.
By the Implicit Function Theorem,
there exist $\bar\delta\coloneqq\bar{\delta}(x)>0$
and $f\colon \hat{n}^\perp\cap B(0,\bar\delta) \to (-\bar{\delta},\bar{\delta})$
such that
	\begin{gather}
		\Sigma\cap C_{\hat n}(x,\bar\delta)
		= \Set{y=x+z-f(z)\hat n: z\in \hat n^\perp \cap B(0,\bar\delta)},
			\label{eq:graph}  \\ 
		E\cap C_{\hat n}(x,\bar\delta) 
			= \Set{ y=x+z-t\hat n: z\in \hat n^\perp \cap B(0,\bar\delta), t\in(f(z),\bar{\delta})}, 
			\label{eq:epigraph}
	\end{gather}
and $\inf_{y\in C_{\hat{n}}(x,\bar\delta)}\va{\nabla\phi(y)}>0$.

\begin{prop}[Rate of convergence and pointwise limit]\label{stm:conv-rate}
Let $E\subset \Rd$ be a set whose topological boundary $\Sigma$ is of class  $C^2$.
Let us fix $x\in \Sigma$ and let $\bar{\delta}$ and $f$ be as above.
Let also $s\in (0,1)$ be the exponent in \eqref{eq:frac-decay}.
Then, for all $\alpha,\beta\in(0,s)$, there exist $q>1$ and $\bar{\eps}\in(0,1)$
such that $q\bar{\eps}\leq\bar{\delta}$ and that,
for all $\epsilon\in(0,\bar{\eps})$ and all $\delta \in(q\eps,\bar{\delta})$,
it holds
		\begin{equation}\label{eq:conv-rate} 
			\va{ \frac{1}{\eps}H_\eps(E,x) - H_0(\Sigma,x)} \leq \mathrm{Err}(\epsilon,\delta),
		\end{equation}
where		
		\begin{equation}
		\label{eq:error} 
			\mathrm{Err}(\epsilon,\delta)\coloneqq \,
				\frac{1}{\delta}\left(\frac{\eps}{\delta}\right)^{\alpha}
				+ (b_0+1)\norm{\nabla^2f}^2_{L^\infty(D)}\delta
				+ a_0\, O_f(\delta)
				+ \va{\nabla^2 f(0)} \left(\frac{\eps}{\delta}\right)^\beta,
		\end{equation}
$D \coloneqq \hat n^\perp\cap B(0,\bar\delta)$, and
		\begin{equation}\label{eq:omegaf}
		O_f(\delta)\coloneqq \sup_{z\in B(0,\delta)}\va{\nabla^2 f(z)-\nabla^2 f(0)}.
		\end{equation}
	
Furthermore, for all $\gamma\in \big(0, \alpha/(\alpha+1)\big)$,
it is possible to choose $\delta = q\epsilon^\gamma$ in \eqref{eq:conv-rate},
so that
	\[
		\lim_{\epsilon\to 0^+}\frac{1}{\eps}H_\eps(E,x) = H_0(\Sigma,x).
	\]
\end{prop}
\begin{proof}
We firstly focus on the bound on the rate of convergence.
We observe that to validate \eqref{eq:conv-rate}
it suffices to reason on the case
when $x=0$ and $\hat{n}=e_d$.
Indeed, if $x\neq0$ or $\hat{n}\neq e_d$,
by formulas \eqref{eq:nl-cng-var} and \eqref{eq:l-cng-var},
we have
	\[
		H_\epsilon(E,x) = \tilde H_\eps(T(E),0)
			\quad\text{and}\quad
		H_0(\Sigma,x) = \tilde H_0(T(\Sigma),0),
	\]
where $T(y)\coloneqq R(y - x)$, 
$R$ is the matrix associated with the canonical basis of $\Rd$
that represents a  rotation such that $R\hat{n}=e_d$,
and $\tilde H_\eps$ and $\tilde H_0$ are defined
as in \eqref{eq:dfncurv} and \eqref{eq:H0},
the kernels $K_\epsilon$  and $K$ being replaced respectively
by $\tilde K_\epsilon\coloneqq K_\epsilon\circ R^\trasp$
and $\tilde K\coloneqq K\circ R^\trasp$.
		
We argue as in the proof of estimate \eqref{eq:est-reg-hyp}:
we consider $f\colon D \to (-\bar{\delta},\bar{\delta})$ of class $C^2$
such that \eqref{eq:graph} and \eqref{eq:epigraph} hold.
Moreover, $f(0)=0$, $\nabla f(0)=0$, and
		\begin{gather}
		\partial_i f = \dfrac{\partial_i\phi}{\partial_d \phi} \label{eq:implicit1} \\ 
		\partial^2_{i,j}f= \dfrac{1}{\partial_d \phi}
			\left(\partial^2_{i.j}\phi + \partial_i f\,\partial^2_{j,d}\phi
			+ \partial_j f\,\partial^2_{i,d}\phi + \partial_i f\,\partial_j f\,\partial^2_{d,d}\phi\right) 
			\label{eq:implicit2},
		\end{gather}
where, for $i,j=1,\dots,d-1$,
$\partial_i$ and $\partial^2_{i,j}$ denote respectively
the $i$-th partial derivative and the second order partial derivative w.r.t. the $i$-th and $j$-th variable.

Let us now fix  $\eps$ and $\delta$ such that
$0<\epsilon<\delta<\bar{\delta}$.
We write $\eps^{-1}H_\eps$ as
	\[
		\frac{1}{\eps}H_\eps(E, 0) = I_{\epsilon}^0 + I_\epsilon^\infty,
	\]
where
	\begin{gather*}
		I_\eps^0 \coloneqq -\frac{1}{\eps}\int_{C} K_\eps(y)\tilde\chi_E(y)\de y,
		\qquad
		I_\eps^\infty \coloneqq -\frac{1}{\eps}\int_{C^c} K_\eps(y)\tilde\chi_E(y)\de y,
	\end{gather*}
and $C\coloneqq C_{e_d}(0,\delta)$ for short.
The former of the two integrals takes into account the interactions
with points that are close to $0$, 
and, when $\eps$ is small, we expect it
to approximate the anisotropic mean curvature.
The second terms, instead, encodes the contribution given by a region far away from $0$.
Note that the integral defining $I_\epsilon^0$ has to be understood
in the principal value sense.
As suggested by this heuristic, we write
	\begin{equation}\label{eq:Heps-H0}
		\va{\frac{1}{\eps}H_\epsilon(E,0)-H_0(\Sigma,0)}
				\leq \va{I_\eps^0 - H_0(\Sigma, 0)} + \va{I_\epsilon^\infty}.
	\end{equation}

We may easily bound the second summand on the right-hand side
by means of \eqref{eq:imbert1}:
for all  $\alpha<s$ there exists $q_\infty>1$ such that
\begin{equation}\label{eq:Iepsinfty}
	\va{I_\epsilon^\infty} \leq
			\frac{1}{\eps}\int_{B\left(0,\frac{\delta}{\eps}\right)^c}K(y)\de y
			\leq \frac{1}{\delta}\left(\frac{\eps}{\delta}\right)^{\alpha}
	\quad\text{whenever } q_\infty\eps <\delta.
\end{equation}

To deal with the difference between $I_\epsilon^0$ and $H_0(\Sigma,0)$,
we rewrite both the terms in a more convenient way.
We introduce  the function $f_\eps(z) \coloneqq f(\eps z)/\eps$,
and we notice that
	\begin{equation*}
		\begin{split}
		I^0_\eps & = \frac{1}{\eps} \int_{e_d^\perp\cap B\left(0,\delta\right)}
								\left[
										\int_{-\delta}^{f(z)} K_\eps(z + t e_d)\de t
										- \int_{f(z)}^{\delta} K_\eps(z + t e_d)\de t
								\right] \de\Hdmu(z) \\
						& = \frac{1}{\eps}\int_{e_d^\perp\cap B\left(0,\frac{\delta}{\eps}\right)}
								\int_{-f_\eps(-z)}^{f_\eps(z)} K(z+te_d)\de t\de\Hdmu(z).
		\end{split}	
	\end{equation*}
Thanks to the regularity of $f$,
for all $z\in D$ we have
	\[
	f_\eps(z)=\frac{\eps}{2}\nabla^2 f(\eps \theta z)z\cdot z
	\quad \text{for some  $\theta\in(0,1)$ },
	\]
so that, when $t$ ranges between $-f_\eps(-z)$ and $f_\eps(z)$,
it holds
	\begin{equation}\label{eq:stima-quadr}
	\va{t} \leq \frac{\eps}{2}\norm{\nabla^2 f}_{L^\infty(D)}\va{z}^2.
	\end{equation}

Further, we remark that
	\[
		H_0(\Sigma,0) =  \int_{e_d^\perp} K(z)\nabla^2 f(0)z\cdot z\de\Hdmu(z).
	\]	
Indeed,	$H_0(\Sigma, 0)$ may be cast as in \eqref{eq:H0},
and by \eqref{eq:implicit1}, \eqref{eq:implicit2} we have
		\[
			-\frac{1}{\va{\partial_d\phi(0)}}\nabla^2 \phi(0)z\cdot z
			=	\nabla^2 f(0)z\cdot z
			\quad\text{when } z\in e_d^\perp.
		\]

We introduce the quantity
		\[ 
			I_\eps^1 \coloneqq  \frac{1}{\eps}\int_{e_d^\perp\cap B\left(0,\frac{\delta}{\eps}\right)}
						K(z) \left[f_\eps(z)+f_\eps(-z)\right]\de\Hdmu(z)
		\]
to help us quantify the difference between $I_\epsilon^0$ and $H_0(\Sigma,0)$.
By the triangular inequality,
	\begin{equation}\label{eq:Ieps0-H0}
		\va{I_\epsilon^0 - H_0(\Sigma,0)}
			\leq \va{I_\epsilon^0 - I_\epsilon^1} + \va{I_\epsilon^1 - H_0(\Sigma,0)},
	\end{equation}
and we estimate the summands in the right-hand side separately. 

We observe that 
		\begin{equation}\label{eq:Ieps0-Ieps1}
		\va{I^0_\eps-I_\eps^1} \leq
			\frac{1}{\eps}\int_{e_d^\perp\cap B\left(0,\frac{\delta}{\eps}\right)}
				\va{\int_{-f_\eps(-z)}^{f_\eps(z)} \left[K(z+te_d)-K(z)\right]\de t}\de\Hdmu(z).	
		\end{equation}
Since $K\in W^{1,1}(\co{B(0,r)})$ for all $r>0$,
by Theorem \ref{stm:acl},
	for $\Hdmu$-a.e. $z\in e_d^\perp$ it holds
		\[K(z+te_d)-K(z)=\int_0^t \partial_d K(z+se_d) \de s,\]
	and this, combined with \eqref{eq:stima-quadr}, implies that 		\begin{equation*}
			\va{K(z+te_d)-K(z)} \leq
				\int_{-\frac{\eps}{2}\norm{\nabla^2f}_{L^\infty(D)}\va{z}^2}^{\frac{\eps}{2}\norm{\nabla^2f}_{L^\infty(D)}\va{z}^2}
					\va{\nabla K(z+se_d)}\de s.
		\end{equation*}
By plugging this inequality in \eqref{eq:Ieps0-Ieps1}, we infer
		\begin{equation*}
		\begin{split}
		| I^0_\eps - & I_\eps^1 | \\
		\qquad  & \leq  \norm{\nabla^2f}_{L^\infty(D)}
				\int_{e_d^\perp\cap B\left(0,\frac{\delta}{\eps}\right)} \va{z}^2
				\int_{-\frac{\eps}{2}\norm{\nabla^2f}_{L^\infty(D)}\va{z}^2}^{\frac{\eps}{2}\norm{\nabla^2f}_{L^\infty(D)}\va{z}^2}
			\va{\nabla K(z+se_d)}\de s \de\Hdmu(z)\\
		\qquad & \leq \frac{\delta}{\eps} \norm{\nabla^2f}_{L^\infty(D)}
				\int_{Q(\eps)}\va{y}\va{\nabla K(y)}\de y,
		\end{split}
		\end{equation*}
	where  $Q(\eps)\coloneqq Q_{\eps\norm{\nabla^2f}_{L^\infty(D)}}(e_d)$. 		
	In view of \eqref{eq:imbert4}, there exists $\eta\in(0,\bar{\delta})$ such that  
		\begin{equation}\label{eq:Ieps0Jeps2}
		\va{I^0_\eps-I_\eps^1}\leq (b_0+1)\norm{\nabla^2f}^2_{L^\infty(D)}\delta
		\quad\text{whenever } \eps <\eta.
		\end{equation}
	
Then, we have		
		\begin{equation*}
			\begin{split}
			\va{I_\eps^1-H_0(\Sigma, 0)} & \leq 
					O_f(\delta) \int_{e_d^\perp\cap B\left(0,\frac{\delta}{\eps}\right)}
										K(z)\va{z}^2\de\Hdmu(z) \\
					&\quad + \va{\nabla^2 f(0)} \int_{e_d^\perp\cap \co{B\left(0,\frac{\delta}{\eps}\right)}}
								K(z)\va{z}^2\de\Hdmu(z),
			\end{split}
		\end{equation*}
	with $O_f$ defined in \eqref{eq:omegaf}.
	As a consequence of \eqref{eq:decad-Kz2},
	for all $\beta<s$, there exists $q_1>0$
	such that, if $q_1\eps <\delta$, then
		\[
			\int_{e_d^\perp\cap \co{B\left(0,\frac{\delta}{\eps}\right)}}
				K(z)\va{z}^2\de\Hdmu(z) \leq \left(\frac{\eps}{\delta}\right)^\beta;
		\]
	Hence, bearing in mind \eqref{eq:Kz2},
	we infer
		\begin{equation}\label{eq:Ieps1-H0}
			\va{I_\eps^1-H_0(\Sigma,0)} \leq a_0\, O_f(\delta)
				+ \va{\nabla^2 f(0)} \left(\frac{\eps}{\delta}\right)^\beta
			\quad\text{whenever } q_1\eps <\delta.
		\end{equation}
	
	Summing up, we can now accomplish the proof of \eqref{eq:conv-rate}
	by choosing the parameters $q$ and $\bar{\eps}$
	in such a way that the estimates above are fulfilled at once.
	Inequalities \eqref{eq:Heps-H0} and  \eqref{eq:Ieps0-H0} get
		\[
			\va{\frac{1}{\eps}H_\epsilon(E,0)-H_0(\Sigma,0)}
				\leq \va{I_\epsilon^0 - I_\epsilon^1} + \va{I_\epsilon^1 - H_0(\Sigma,0)} + \va{I_\epsilon^\infty},
		\]
	therefore, if we set $q\coloneqq \max\{q_\infty,q_1\}>1$
	with $q_\infty$ and $q_1$ as above,
	both \eqref{eq:Iepsinfty} and \eqref{eq:Ieps1-H0}
	hold for all $\epsilon,\delta>0$ such that
	$q\eps < \delta<\bar{\delta}$.
	Besides, if we pick $\bar{\eps}\coloneqq\min\Set{\eta,\bar{\delta}/q}$,
	\eqref{eq:Ieps0Jeps2} is satisfied too
	whenever $\epsilon<\bar{\eps}$.
	
	Once rate estimate \eqref{eq:conv-rate} is on hand,
	the pointwise convergence follows straightforwardly.
	Indeed, for $\gamma\in \big(0, \alpha/(\alpha+1)\big)$,
	$q\epsilon^\gamma > q \eps$ and
	$\mathrm{Err}(\epsilon,q\epsilon^\gamma)$ vanishes as $\epsilon$ tends to $0$. 
\end{proof}

	\begin{rmk}
		Assumptions \eqref{eq:sing-orig} and \eqref{eq:imbert5} are not exploited
		in the proof of Lemma \ref{stm:conv-rate}.
		We shall utilize them to establish Proposition \ref{aprioriest}.
	\end{rmk}
	
When $\Sigma$ is a smooth, compact hypersurface,
uniform convergence of the rescaled nonlocal curvatures stems
from a refinement of the previous error estimate.

\begin{thm}\label{stm:main1}
	Let $K$ satisfy all  the assumptions in Section \ref{sec:K}
	and let $E\subset\Rd$ be a set
	whose topological boundary $\Sigma$ is compact and of class $C^2$.
	Then,
		\[
			\lim_{\epsilon\to 0^+}\frac{1}{\eps}H_\epsilon(E,x)= H_0(\Sigma,x)
					\quad\text{uniformly in $x\in\Sigma$}.
		\]
\end{thm}
\begin{proof}
By definition of $C^2$ hypersurface,
for any $x\in\Sigma$ there exist an open set $U_x \subset \Rd$ and
a $C^2$ function $\phi_x\colon U_x \to \R$ such that
	\[
		\Sigma \cap U_x = \Set{ y \in U_x : \phi_x(y) = 0}.
	\]
For all $x\in \Sigma$, we also choose another open set $V_x$,
in such a way that it is compactly contained in $U_x$
and that $\inf_{y\in V_x}\va{\nabla\phi(y)}>0$.
Observe that, by compactness,
$\Sigma$ admits a finite cover $\Set{V_1,\dots,V_L}$
extracted from the collection $\Set{V_x}_{x\in\Sigma}$.

If $x\in \Sigma$ belongs to $V_\ell$ for some $\ell$,
we choose $\bar\delta_x$ so small that
the closure of the cylinder $C_{\hat{n}_x}(x,\bar\delta_x)$ is contained in $V_\ell$,
where $\hat{n}_x$ is the outer unit normal to $\Sigma$ at $x$.
We now apply Proposition \ref{stm:conv-rate} at the point $x$
selecting $\bar \delta = \bar \delta_x$.
We get, for all $\alpha,\beta\in(0,s)$,
the existence of $q>1$ and $\bar{\eps}\in(0,1)$
such that $q \bar{\eps} \leq \bar{\delta}_x$ and that
for all $\gamma \in  \big(0, \alpha/(\alpha+1)\big)$ and all $\epsilon\in(0,\bar{\eps})$ it holds
	\begin{multline*}
			\va{\frac{1}{\eps}H_\eps(E,x) - H_0(\Sigma,x)}
				\leq \mathrm{Err}(\epsilon,q\epsilon^\gamma) \\
				 \leq	c\Bigl(\eps^{\alpha-\gamma(1+\alpha)}
				+ \norm{\nabla^2f_x}_{L^\infty(\hat{n}_x^\perp\cap B(0,\bar{\delta}_x))}\eps^\gamma 
				+O_{f_x}(q\eps^\gamma)
				+ \va{\nabla^2 f_x(0)} \eps^{(1-\gamma)\beta}\Bigr),
	\end{multline*}
where $c\coloneqq c(\alpha,\beta,a_0,b_0)>0$ and
$f_x$ is the implicit function such that
\eqref{eq:graph} holds in $C_{\hat{n}_x}(x,\bar\delta_x)$.

Recall that $x\in V_\ell$ and
let us denote by $\phi_\ell$ a $C^2$ function $\phi$
such that $\Sigma \cap V_\ell = \Set{ y \in V_\ell : \phi(y) = 0}$.
By formulas \eqref{eq:implicit1} and \eqref{eq:implicit2},
one may show that $\va{\nabla^2 f_x(0)}$ and
$\norm{\nabla^2 f_x}_{L^\infty(\hat{n}_x^\perp\cap B(0,\bar{\delta}_x))}$
are bounded above by quantities that
depend only on the infimum of $\va{\nabla\phi}$ on $V_\ell$,
which is strictly positive,
and on the values of $\nabla \phi_\ell$ and $\nabla^2 \phi_\ell$ in $V_\ell$.
In turn, the latter are bounded in $V_\ell$, because $\phi_\ell$ is $C^2$.
Similarly, $O_{f_x}(\delta)$ is smaller than some $O_\ell(\delta)$
that is defined in terms of $\phi_\ell$ and of its derivatives,
and that vanishes as $\delta\to 0^+$.

In conclusion, since there is a finite number of $\phi_\ell$,
it suffices to maximise w.r.t. the parameter $\ell$
to obtain a bound on $\va{\eps^{-1}H_\eps(E,x) - H_0(\Sigma,x)}$
that is uniform in $x \in \Sigma$ and that vanishes when $\epsilon\to 0^+$.
\end{proof}


	\section{Level set formulations of motions by curvature}\label{sec:geomevol}
		We now revise the approach to mean curvature motions
{\it via } level set formulations.
We include the existence and uniqueness results for the related viscosity solutions
and recall their connections with the notion of geometric barriers. 

We consider the geometric evolutions
driven by the curvatures $\epsilon^{-1}H_\epsilon$ and $H_0$
defined respectively in \eqref{eq:resccurv} and \eqref{eq:H0}:
if $ t \mapsto E(t)$ is an evolution of sets,
we formally prescribe that
	\begin{equation}\label{gf}
		\partial_t x(t) \cdot \hat{n} = - \frac{1}{\eps}H_\epsilon(E(t), x(t)),
		\qquad
		\partial_t x(t) \cdot \hat{n} = - H_0(E(t), x(t)),
	\end{equation}
where $\hat{n}$ is the outer unit normal 
to $\partial E(t)$ at the point $x(t)$.
We accompany these equations with an initial datum $E_0$,
which we assume to be a bounded set. 	

To cast the geometric laws \eqref{gf} in the level set form,
we interpret the initial datum $E_0$
as the $0$ superlevel set of a certain function $u_0\colon \Rd\to\R$,
that is, $E_0=\Set{x : u_0(x)\geq 0}$ and $\partial E_0=\Set{x : u_0(x)=0}$.
Hereafter, we assume that 
	\begin{equation}\label{initialdatum} 
		u_0\colon \Rd \to\R
		\text{ is Lipschitz continuous and constant outside a compact set $C\subset\Rd$.} 
	\end{equation}
Then, we consider the Cauchy's problems: 
	\begin{eqnarray}\label{eq:eps_pb}
		\begin{cases}
		\partial_t  u(t,x) + \frac{\va{\nabla u(t,x)}}{\eps}H_\epsilon(\Set{y:u(t,y)\geq u (t,x)},x) = 0
						& (t,x)\in[0,+\infty)\times\Rd \\
		u(0,x) = u_0(x) 	&	x\in\Rd
	\end{cases}.
	\\  \label{eq:limit_pb}
	\begin{cases}
		\partial_t u(t,x)
		- \mathrm{tr}\left(M_K\left(\widehat{\nabla u(t,x)}\right)\nabla^2 u(t,x)\right)
		= 0
		& (t,x)\in[0,\infty)\times\Rd \\
		u(0,x) = u_0(x) 	& x\in\Rd
	\end{cases}.
\end{eqnarray}

Observe that
	\[
		\va{\nabla u(x)}H_0(\Set{y:u(y)= u (x)},x)
			=	-\mathrm{tr}\left(M_K\left(\widehat{\nabla u(x)}\right)\nabla^2 u(x)\right).
	\] 
Parabolic equations defined by operators of this form are well studied.
We shall consider solutions in the following viscosity sense
(see \cite{DFM} and references therein):

\begin{dfn}[Solution to the limit problem]
	A function $u\colon [0,\infty)\times \Rd\to\R$ is
	a \emph{viscosity subsolution}(resp. \emph{superso\-lu\-tion})
	to Cauchy's problem \eqref{eq:limit_pb} if
	it is locally bounded, upper semicontinuous (resp. lower semicon\-tin\-u\-ous func\-tion),
	and satisfies the following:
	\begin{enumerate}
		\item $u(0,x)\leq u_0(x)$ for all $x\in\Rd$, (resp. $u(0,x)\geq u_0(x)$);
		\item for all $(t,x)\in(0,+\infty)\times \Rd$
		and for all $\phi\in C^2([0,+\infty)\times \Rd)$ such that
		$u -\phi$ has a maximum at $(t,x)$ (resp. a minimum at $(t,x)$)
		it holds
		\[\partial_t\phi(t,x)\leq 0\quad(\text{resp. }\partial_t\phi(t,x)\geq 0)
		\qquad \text{when } \nabla\phi(t,x)=0 \text{ and } \nabla^2\phi(t,x)=0
		\]
		or 
		\[\partial_t \phi(t,x)
		- \mathrm{tr}\left(M_K\left(\widehat{\nabla \phi(t,x)}\right)\nabla^2 \phi(t,x)\right) \leq 0 \   (\text{resp. }\geq 0)
		\quad\text{otherwise.}
		\]
	\end{enumerate}
	A continuous function $u\colon [0,+\infty)\times \Rd \to\R$
	is a \emph{viscosity solution} \index{viscosity solution!to the anisotropic mean curvature motion}
	to \eqref{eq:limit_pb}
	if it is both a viscosity sub- and supersolution.
\end{dfn}
As for existence, note that the function
	\begin{equation*}
		\begin{matrix}
			F_0\colon & \Rdmz\times \mathrm{Sym}(d)&  \longrightarrow & \R \\
							& (p,X) & \longmapsto &
				\displaystyle{-\mathrm{tr}\left(M_K\left(\hat{p}\right)X\right)}
		\end{matrix}	
	\end{equation*}
that defines the problem \eqref{eq:limit_pb} satisfies the following:
\begin{enumerate}
	\item it is continuous;
	\item it is geometric, that is
			for all $\lambda >0$, $\sigma\in\R$, $p\in\Rd\setminus\Set{0}$
			and $X\in\mathrm{Sym}(d)$ it holds
			$F_0(\lambda p, \lambda X + \sigma p\otimes p)=\lambda F_0(p,X)$.
	\item it is degenerate elliptic, that is
			$F_0(p,X)\geq F_0(p,Y)$
			for all $p\in\Rdmz$
			and $X,Y\in\mathrm{Sym}(d)$ such that $X\leq Y$.
\end{enumerate}
Since the classic by Y. Chen, Y. Giga, and S. Goto \cite{CGG},
it is well known that
that these three conditions grant
existence and uniqueness of a viscosity solution:
\begin{thm}\label{esistenza0}
	Assume that \eqref{initialdatum} holds.
	Then, Cauchy's problem \eqref{eq:limit_pb} admits
	a unique bounded Lipschitz viscosity solution in $[0,+\infty)\times \Rd$,
	which is constant in $\Rd\setminus C$, for some compact set $C\subset \Rd$.
\end{thm}

In second place,
we recall the notion of viscosity solution for the nonlocal problems.
which goes back to the work \cite{S},
see also \cites{I,DFM,CMP}.

\begin{dfn}[Solution to the rescaled problems]
	A locally bounded, upper semicontinuous function  (resp. lower semicontinuous) 
	$u_\eps\colon [0,+\infty)\times \Rd \to\R$  is a \emph{viscosity subsolution} (resp. \emph{supersolution})
	to the problem \eqref{eq:eps_pb} if
	\begin{enumerate}
		\item $u_\eps(0,x)\leq u_0(x)$ for all $x\in\Rd$ (resp. $u_\eps(0,x)\geq u_0(x)$);
		\item for all $(t,x)\in(0,+\infty)\times \Rd$
		and for all $\phi\in C^2([0,+\infty)\times \Rd)$ such that
		$u_\eps -\phi$ has a maximum at $(t,x)$ (resp. has a minimum at $(t,x)$), 
		it holds
		\[\partial_t\phi(t,x)\leq 0 \quad  (\text{resp. }\partial_t\phi(t,x)\geq 0)\qquad
		\text{when } \nabla\phi(t,x)=0, 
		\]
		or 
		\begin{gather*}
			\partial_t \phi(t,x)
				+  \frac{\va{\phi(t,x)}}{\eps}H_\epsilon(\Set{y:\phi(t,y)\geq \phi (t,x)},x) \leq 0 \\
			\big(
				\text{resp. } \partial_t \phi(t,x)
				+ \frac{\va{\phi(t,x)}}{\eps} H_\epsilon(\Set{y:\phi(t,y) >\phi (t,x)},x) \geq 0
				\big)
		\quad\text{otherwise.}
		\end{gather*}
	\end{enumerate}
	
	A continuous function $u_\eps\colon [0,+\infty)\times \Rd \to\R$
	is a \emph{viscosity solution}\index{viscosity solution!to the nonlocal curvature motion}
	to \eqref{eq:eps_pb}
	if it is both a viscosity sub- and supersolution. 
\end{dfn}

Using this concept of solution,
comparison results and well posedness of \eqref{eq:eps_pb} are available,
as proved in a very general setting
by A. Chambolle, M. Morini, and M. Ponsiglione in \cite{CMP}
(see also the paper by C. Imbert \cite{I}).

\begin{thm}[Comparison principle and existence of solutions to the nonlocal problem]\label{esistenzaeps}
	If the kernel $K$ is as in Section \ref{sec:K} and
	\eqref{initialdatum} holds, then,
	for all $\epsilon>0$,
	if $v_\eps, w_\eps\colon [0,+\infty)\times\Rd \to \R$
	are respectively a sub- and a supersolution to \eqref{eq:eps_pb},
	then $v_\eps(t,x)\leq w_\eps(t,x)$ for all $(t,x)\in [0,+\infty)\times\Rd$.
	
	Moreover, \eqref{eq:eps_pb} admits a unique bounded, Lipschitz viscosity solution
	in $[0,+\infty)\times\Rd$, which is constant in $\Rd\setminus C$,
	for some compact set $C\subset \Rd$.  
\end{thm}  

Summing up, we know that,
for every initial datum $u_0$ fulfilling \eqref{initialdatum},
there exists a unique viscosity solution $u_\eps$ to \eqref{eq:limit_pb}
and a unique viscosity solution $u$ to \eqref{eq:limit_pb}. 
We define the related \emph{level set flows}\index{level set flow} as follows:
for every $\lambda\in \R$, we put
	\begin{gather}\label{levelseteps}
		E_{\eps, \lambda}^>(t)=\Set{x\in\Rd : u_\eps(t,x)>\lambda},
		\quad
		E_{\eps, \lambda}^\geq(t) = \Set{x\in\Rd : u_\eps(t,x)\geq \lambda}, \\
		E^{>}_{\lambda}(t)=  \Set{x\in\Rd : u(t,x)> \lambda},
		\quad		
		E^{\geq}_{\lambda}(t) = \Set{x\in\Rd : u(t,x)\geq \lambda}.
	\label{levelset0}  
	\end{gather}
It is well known that, as long as they are smooth, 
these families are solutions to the geometric flows \eqref{gf} with initial data
respectively $\Set{u_0 > \lambda}$ and $\Set{u_0 \geq \lambda}$. 

Geometric evolutions can be formulated
as PDEs involving the distance function from the moving front,
see for instance the survey \cite{A} by L. Ambrosio.
In the following definitions,
we use it to express a regularity property
both in time and space for sets (see \ref{stm:ii} below) 
evolving according to a generic geometric law.

\begin{dfn}\label{defistrict}
	Let $0\leq t_0<t_1<+\infty$. 
	We say that the evolution of sets
	$[t_0,t_1]\ni t\mapsto D(t)$ is a
	\emph{geometric subsolution (resp. supersolution)}\index{geometric sub- and supersolutions}
	to the flow associated with the curvature functional $H$ if
	\begin{enumerate}
		\item\label{stm:i} $D(t)$ is closed and  $\partial D(t)$ is compact
		for all $t\in[t_0,t_1]$;
		\item\label{stm:ii} there exists an open set $U\subset \Rd$
		such that the distance function $(t,x)\mapsto \de_{D(t)}(x)$
		is of class $C^\infty$ in $[t_0,t_1]\times U$
		and $\partial D(t)\subset U$ for all $t\in[t_0,t_1]$;
		\item for all $t\in(t_0,t_1)$ and $x(t)\in\partial D(t)$,
		it holds
		\begin{equation}\label{eq:smooth-sub}
			\partial_t x(t)\cdot \hat{n} \leq -H( D(t),x(t))\qquad(\text{resp. }\partial_t x(t)\cdot \hat{n} \geq -H( D(t),x(t)),
		\end{equation}
		where $\hat{n} $ is the outer unit normal to $D(t)$ at $x$.
	\end{enumerate}
	
	When strict inequalities hold,
	$D(t)$  is called \emph{strict geometric subsolution (resp. supersolution)}.
\end{dfn}

\begin{rmk}
	For any $p\in \Rd\setminus\Set{0}$ and
	$X\in\mathrm{Sym}(d)$, by \eqref{eq:Kz2} we have 
	\[\va{\mathrm{tr}(M_K(\hat{p})X)} = 
	\frac{1}{2}\va{\int_{\hat{p}^\perp}
		K(z)\, Xz\cdot z\de\Hdmu(z)}
	\leq \frac{a_0}{2}\va{X},
	\]
	This ensures that
	geometric sub- and supersolution for the flow driven by $H_0$ exist.
\end{rmk}

Next, we remind the notion of geometric barriers
w.r.t. these smooth evolutions:

\begin{dfn}\label{defbarrier}\index{barrier}
	Let $T>0$ and  $\mathcal{F}^-$ and $\mathcal{F}^+$
	be, respectively, the families of
	strict geometric sub- and supersolution 
	to the flow  associated with some curvature functional $H$,
	as in Definition \ref{defistrict}. 
	\begin{enumerate}
		\item We say that the evolution of sets
		$[0,T]\ni t\mapsto E(t)$ is an \emph{outer barrier}
		w.r.t. $\mathcal{F}^-$ (resp. $\mathcal{F}^+$)
		if whenever $[t_0,t_1]\subset[0,T]$
		and $[t_0,t_1]\ni t\mapsto D(t)$ is a
		smooth strict subsolution
		(resp. $F(t)$ is a smooth strict  supersolution) such that
		$D(t_0)\subset E(t_0)$,
		then  we get  $D(t_1)\subset E(t_1)$ (resp. such that $F(t_0)\subset E(t_0)$,
		then  we get  $F(t_1)\subset E(t_1)$).
		\item Analogously,
		$[0,T]\ni t\mapsto E(t)$ is an \emph{inner barrier}
		w.r.t. the family $\mathcal{F}^-$ (resp. $\mathcal{F}^+$)
		if whenever $[t_0,t_1]\subset[0,T]$
		and $[t_0,t_1]\ni t\mapsto D(t)$ is a
		smooth strict subsolution			
		(resp. supersolution) such that
		$E(t_0) \subset \mathrm{int}(D(t_0))$,
		then $E(t_1) \subset \mathrm{int}(D(t_1))$ (resp. $E(t_0) \subset \mathrm{int}(F(t_0))$,
		then $E(t_1) \subset \mathrm{int}(F(t_1))$).			
	\end{enumerate}
\end{dfn}

We are interested in barriers is motivated by the next results,
which show that they are comparable with level sets flows.
The case of the local curvatures was studied in \cite{BNo}*{Theorem 3.2}
by G. Bellettini and M. Novaga.
For further reading about barriers for more general local motions,
we refer also to \cite{BNo2}.
In the nonlocal setting, recently
an analogous comparison principle has been established
in \cite{CDNV}*{Proposition A.10} by A. Cesaroni, S. Dipierro, M. Novaga, and E. Valdinoci.

\begin{thm}\label{stm:BN}
	Let  $u$ be the unique solution to \eqref{eq:limit_pb} with initial datum $u_0$ as in \eqref{initialdatum}.
	Let $E^{>}_{\lambda}$, $E^{\geq}_{\lambda}$ the sets defined in \eqref{levelset0}.  
	\begin{enumerate}
		\item The map 
		$[0,T]\ni t\mapsto E^>_\lambda(t)$
		is the minimal outer barrier for
		the family of strict geometric subsolutions
		associated with $H_0$,
		that is $E^>_\lambda(t)$ is an outer barrier
		and $E^>_\lambda(t)\subset E(t)$
		for any other outer barrier $E(t)$.
		\item The map $[0,T]\ni t\mapsto E^\geq_\lambda(t)$ 
		is the maximal inner barrier for
		the family of  geometric  strict supersolutions
		associated with $H_0$,
		that is $E^\geq_\lambda(t)$ is an inner barrier
		and $E(t) \subset E^\geq_\lambda(t)$
		for any other inner barrier $E(t)$.
	\end{enumerate}
\end{thm}

\begin{prop}\label{stm:comp-CDNV}
	Let $u_\epsilon\colon [0,+\infty)\times \Rd\to \R$
	be the viscosity solution to \eqref{eq:eps_pb} with initial datum $u_0$ as in \eqref{initialdatum}. 
	Let $E_{\epsilon,\lambda}^>$, $E_{\epsilon,\lambda}^\geq$
	be as in \eqref{levelseteps}.
	Then, the evolutions
	$t\mapsto E_{\epsilon,\lambda}^>(t)$
	and $t\mapsto E_{\epsilon,\lambda}^\geq(t)$
	are, respectively,
	an outer barrier w.r.t  geometric  strict subsolutions to \eqref{eq:eps_pb}
	and an inner barrier w.r.t geometric  strict supersolutions to \eqref{eq:eps_pb}.
\end{prop}	
	\section{Convergence of the solutions to the nonlocal flows}\label{sec:convflows}
		Now that we clarified
what class of solutions for geometric motions we address to,
we can formulate our main result:
	
	\begin{thm}\label{stm:main2}\index{convergence of the solutions of the rescaled nonlocal flows}
		Let $K$ satisfy all  the assumptions in Section \ref{sec:K}. 
		Let $u_0\colon \Rd \to \R$ be a Lipschitz continuous function
		that is  constant outside a compact set.
		Let  $u_\eps, u\colon [0,+\infty)\times\Rd \to \R$ be respectively
		the unique continuous viscosity solution to
		\eqref{eq:eps_pb} and \eqref{eq:limit_pb}	with  initial datum $u_0$.
		Then
			\[
				\lim_{\eps\to 0^+} u_\eps(t,x)=u(t,x) \quad\text{locally uniformly in } [0,+\infty)\times\Rd.\] 
	\end{thm}

An analogous result was proved by F. Da Lio, N. Forcadel, and R. Monneau \cite{DFM}*{Theorem 1.4}.
Though the authors were concerned with a problem which resembles much ours,
their assumptions on the interaction kernel were dictated
by a model of dislocation dynamics and are different from the ones we consider.
Indeed, they assumed $K$ to be bounded in a neighbourhood of $0$, hence nonsingular,
and to decay as $\va{x}^{-d-1}$ at infinity. 
Perforce, their scaling of the curvature,
	\[\frac{1}{\eps\log \eps} H_{K_\eps},\] 
does not coincide with ours.
However, the proofs are partly similar.
Indeed, as \cite{DFM}, we firstly prove a compactness result for  
the family of solutions to Cauchy's problems \eqref{eq:eps_pb},
but, in contrast with the work of Da Lio, Forcadel, and Monneau,
we show that the cluster point must be unique
and coincide with the viscosity solution to \eqref{eq:limit_pb}
by means of De Giorgi's barriers.

Let us first prove the compactness result,
which is not completely novel.
We include the proof because it has not been stated explicitly yet for our setting. 

	\begin{prop}\label{aprioriest}\index{compactness!for the solutions of the rescaled nonlocal flows}
		Assume that $u_0\colon \Rd \to\R$ is as in \eqref{initialdatum} 
		and let $u_\eps$ be the unique continuous viscosity solution to \eqref{eq:eps_pb}. 
		Then, 
			\begin{equation}\label{equilip}
				\va{u_\eps(t,x)-u_\eps(t,y)}\leq
					\norm{\nabla u_0}_{L^\infty(\Rd)} \va{x-y}
				\quad \text{for all } t\in [0,T]
					\text{ and } x,y\in \Rd,
			\end{equation}
		and there exists a constant $c>0$ independent of $\eps$ such that
			\begin{equation}\label{equihold}
				\va{u_\eps(t,x)-u_\eps(s,x)}
					\leq \norm{\nabla u_0}_{L^\infty(\Rd)}\sqrt{c\va{t-s}}
				\quad \text{for all } t,s\in [0,T]
					\text{ and } x\in \Rd.
		\end{equation}
	\end{prop}		
	\begin{proof}
	By standard arguments \cite{DFM,CDNV}, one can see that
	the Lipschitz continuity of the datum and the comparison principle yield
	the equi-Lipschitz estimate \eqref{equilip}.
	
	As for the equi-H\"older continuity,
	we adapt the strategy of Section 5 in \cite{DFM}.
	Some care is needed because of the possible singularity of our interaction kernel.
	
	For $\eta>0$ and $x\in\Rd$ fixed, we consider
		\begin{equation}\label{test}
			\phi(t,y)= Lt + A\sqrt{\va{y-x}^2+\eta^2}+u_0(x),
		\end{equation}
	where we set $A\coloneqq \norm{\nabla u_0}_{L^\infty(\Rd)}$. 
	We claim that, for $L>0$ sufficiently large,
	$\phi$ is a supersolution to \eqref{eq:eps_pb} for any $\epsilon\in(0,1)$.
	
	We firstly note that $\phi(0,y)\geq u_0(y)$
	thanks to the Lipschitz continuity of $u_0$.
	Also, we observe that, for any $y\in \Rd$,
		\[\Set{z\in\Rd : \phi(t,z)\geq \phi(t,y)}
			= \co{B(x,\va{y-x})},\]
	whence, to show that $\phi$ is a supersolution,
	it suffices to find $L$ so large that 
		\[
			\frac{L}{A} \geq
			\frac{ \va{y-x}}{\eps \sqrt{\va{y-x}^2+\eta^2}}H_\eps(B(x,\va{y-x}),y)
				\quad\text{for all }y\in\Rd \text{ and } \epsilon\in(0,1).\] 
	Recalling that the nonlocal curvature is invariant under translations,
	if we set $e \coloneqq \widehat{y-x}$ and $r\coloneqq \va{y-x}$,
	we have that the last inequality holds if and only if
		\begin{equation}\label{eq:unif-supersol}
		\frac{L}{A} \geq \frac{r}{\eps \sqrt{r^2+\eta^2}}H_\eps(B(-re,r),0)
		\quad\text{for all }r>0,e\in\Sdmu, \text{ and } \epsilon\in(0,1).
		\end{equation}
	Hence, we want to prove that there exists 
	$L_0\coloneqq L_0(\eta)>0$ such that
		\begin{equation}\label{sup}
			\sup_{ r>0,\,e\in\mathbb{S}^{d-1}}\sup_{\epsilon\in(0,1)}
			\frac{r}{ \eps \sqrt{r^2+\eta^2}}H_\eps(B(-re,r),0)
			\leq L_0.
		\end{equation}
	This clearly gets \eqref{eq:unif-supersol} for $L=AL_0$.
	
	To recover \eqref{sup},
	we use inequality \eqref{eq:est-reg-hyp}.
	Retracing its proof, we see that, when $E=B(-re,r)$ and $x=0$,
	we can choose $\lambda = \lambda_0/r$ for some $\lambda_0 >0$ and $\bar{\delta} = r/2$,
	so that we get
		\[0 \leq H_\eps(B(-re,r),0)
			\leq \int_{Q_{\frac{\lambda_0\eps}{r}}(e)} K(y)\de y
				+ \int_{\co{B\left(0,\frac{r}{2\eps}\right)}}K(y)\de y.\]
	It follows that
		\begin{equation*}
			\frac{r}{\eps \sqrt{r^2+\eta^2}}H_\eps(B(-re,r),0)
				\leq \frac{r}{\eps\eta}
				\left[\int_{Q_{\frac{\lambda_0\eps}{r}}(e)} K(y)\de y
				+ \int_{\co{B\left(0,\frac{r}{2\eps}\right)}}K(y)\de y\right].
		\end{equation*}
	We exploit the hypotheses on $K$:
	by \eqref{eq:sing-orig}, \eqref{eq:imbert5},
	\eqref{eq:imbert3}, and \eqref{eq:imbert1},
	there exist $\lambda,\Lambda>0$ with the following properties:
		\begin{enumerate}
			\item $\lambda<\Lambda$;
			\item if $r<\lambda \epsilon$, then
					\[\frac{r}{\eps}\int_{Q_{\frac{\lambda_0\eps}{r}}(e)} K(y)\de y\leq \frac{1}{2}
					\quad\text{and}\quad
					\frac{r}{\eps}\int_{\co{B\left(0,\frac{r}{2\eps}\right)}}K(y)\de y \leq\frac{1}{2}\]
				and, consequently,
					\begin{equation}\label{eq:regime1}
					\frac{r}{\eps\sqrt{r^2+\eta^2}}H_\eps(B(-re,r),0) \leq \frac{1}{\eta};
					\end{equation}
			\item if $r>\Lambda \epsilon$, then
					\[\frac{r}{\eps}\int_{Q_{\frac{\lambda_0\eps}{r}}(e)} K(y)\de y\leq a_0 + \frac{1}{2}
					\quad\text{and}\quad
					\frac{r}{\eps}\int_{\co{B\left(0,\frac{r}{2\eps}\right)}}K(y)\de y \leq\frac{1}{2}\]
				and, accordingly,
					\begin{equation}\label{eq:regime2}
					\frac{r}{\eps\sqrt{r^2+\eta^2}}H_\eps(B(-re,r),0) \leq \frac{a_0+1}{\eta}.
					\end{equation}
		\end{enumerate}
	Now, only the case
	$\lambda\eps \leq r\leq \Lambda\epsilon$ is left to discuss.
	In this intermediate regime, 
	recalling \eqref{eq:massa-parabole},
	we easily obtain
		\begin{equation}\label{eq:regime3}
		\frac{r}{\eps\sqrt{r^2+\eta^2}}H_\eps(B(-re,r),0)
			\leq \frac{\Lambda}{\eta}
				\left(c + \int_{\co{B\left(0,\frac{\lambda}{2}\right)}}K(y)\de y\right),
		\end{equation}
	with $c>0$ depending only on $\lambda$.
	
	By the whole of	\eqref{eq:regime1}, \eqref{eq:regime2}, and \eqref{eq:regime3},
	we infer that there exists a constant
	$c\coloneqq c(a_0,\lambda,\Lambda)>0$ such that
		\[\sup_{ r>0,\,e\in\mathbb{S}^{d-1}}\sup_{\epsilon\in(0,1)}
		\frac{r}{\eps\sqrt{r^2+\eta^2}}H_\eps(B(-re,r),0)
		\leq \frac{c}{\eta}\]
	and, thus, \eqref{eq:unif-supersol} holds
	if we pick $ L=cA/\eta$.
	
Summing up, we proved  that, for any fixed $x\in\Rd$,
	the function
		\[\phi(t,y)=A\left(\frac{c}{\eta}t + \sqrt{\va{y-x}^2+\eta^2} \right) + u_0(x)\]
	is a supersolution to \eqref{eq:eps_pb} for any $\epsilon>0$.
	A similar reasoning shows that there exists $c=c(a_0,\lambda,\Lambda)$
	such that the function
		\[
			\psi(t,y) \coloneqq -A\left(\frac{c}{\eta}t + \sqrt{\va{y-x}^2+\eta^2} \right) + u_0(x),
		\]
	is a subsolution to \eqref{eq:eps_pb} for all $x\in\Rd$ and $\epsilon>0$.		 
	
	We apply the comparison principle in Theorem \ref{esistenzaeps}
	to $u_\epsilon$, $\phi$, and $\psi$.
	This yields	that, for all $(t,x)\in[0,T]\times\Rd$ and all $\eta>0$, 
		\[
			\va{u_\eps(t,x)-u_0(x)}
			\leq \norm{\nabla u_0}_{L^\infty(\Rd)}\left(\frac{c}{\eta}t + \eta\right).
		\]
	Being $\eta$ arbitrary, we may set $\eta=\sqrt{ct}$ and obtain
		\begin{equation*}\label{eq:q-equihold}
		\va{u_\eps(t,x)-u_0(x)}\leq  2\norm{\nabla u_0}_{L^\infty(\Rd)}\sqrt{ct}.
		\end{equation*}
	In turn, this entails \eqref{equihold},
	because the problem \eqref{eq:eps_pb} is invariant w.r.t. translations in time,
	it admits a unique solution, and
	$\norm{\nabla u_\eps(t,\,\cdot\,)}_{L^\infty(\Rd)}\leq \norm{\nabla u_0}_{L^\infty(\Rd)}$
	for all $t\in[0,T]$.
\end{proof}

Now, we deal with the proof of  Theorem \ref{stm:main2}.
Thanks to the previous proposition,
we know that the family $\Set{u_\eps}$ of the solutions to \eqref{eq:eps_pb}
is precompact in $C([0,T]\times \Rd)$.
It follows that, for some sequence $\Set{\epsilon_\ell}$ and some $v\in C([0,T]\times \Rd)$,
$u_{\eps_\ell}$ converges to $v$ locally uniformly as $\ell\to +\infty$
The desired conclusion is achieved
as soon as we prove $v=u$, $u$ being the solution to \eqref{eq:limit_pb}.
In view of the following lemma,
we see that the equality may be inferred
by certain inclusions of the level sets of $u$ and $v$:
	\begin{lemma}\label{stm:incl-suplev}
		Let $f,g\colon\Rd \to \R$ be  two continuous  functions such that  for all $\lambda \in \R$ there  hold 
		\begin{gather*}
		\Set{x\in\Rd : f(x)>\lambda}\subseteq \Set{x\in\Rd : g(x)\geq \lambda}, \\
		\Set{x\in\Rd : g(x)>\lambda}\subseteq \Set{x\in\Rd : f(x)\geq \lambda}.
		\end{gather*}
		Then, $f(x)=g(x)$ for all $x\in\Rd$.	
	\end{lemma}
	\begin{proof}
		Let $\bar x\in\Rd$ and
		assume that $g(\bar x)=\lambda$. Then, for all $\mu>0$,
		we have
			\[
			\bar x\in \Set{x : g(x)>\lambda-\mu}\subseteq  \Set{x :f(x)\geq \lambda-\mu},
			\]
		which yields $f(\bar x)\geq \lambda$.
		If $f(\bar x)>\lambda$, then
			\[
			\bar x\in \Set{x : f(x)>\lambda+\mu_0}\subseteq  \Set{x : g(x)\geq \lambda+\mu_0>\lambda}
			\]
		some $\mu_0>0$, which contradicts the equality  $g(\bar x)=\lambda$. 
		Thus, it necessarily holds $f(\bar x)=\lambda$.
	\end{proof}	

So, to show that any cluster point $v$ of the family $\Set{u_\epsilon}$ coincides with $u$,
we compare its superlevel sets with the ones of $u$.
To this purpose, we introduce the \emph{set-theoretic upper limits}\index{set-theoretic upper limit}
of the level set flows $t\mapsto E^{>}_{\eps,\lambda}(t)$, $t\mapsto E^{\geq}_{\eps,\lambda}(t)$
associated with $u_\epsilon$ (recall \eqref{levelseteps}).
We define
	\begin{equation}\label{eq:limsup}
		\tilde E^>_\lambda(t) \coloneqq \bigcap_{\eps<1} \bigcup_{\eta<\eps} E^{>}_{\eta, \lambda}(t)
		\quad\text{and}\quad
		\tilde E^\geq_\lambda(t) \coloneqq \bigcap_{\eps<1} \bigcup_{\eta<\eps} E^{\geq}_{\eta, \lambda}(t). 	
	\end{equation} 

	\begin{rmk}\label{ind}
	It is readily shown that, for any $\bar\eps<1$, 
	\begin{equation*}
		\tilde E^>_\lambda(t) = \bigcap_{\eps<\bar\eps}
												\bigcup_{\eta<\eps} E^>_{\eta, \lambda}(t)
	\quad\text{and}\quad
		\tilde E^\geq_\lambda(t) = \bigcap_{\eps<\bar\eps}
													\bigcup_{\eta<\eps} E^\geq_{\eta, \lambda}(t). 	
	\end{equation*}
	\end{rmk}

The proof of the main result of this chapter follows.
In our argument,
the uniform convergence of curvatures established in Theorem \ref{stm:main1} plays a major role.

\begin{proof}[Proof of Theorem \ref{stm:main2}]
Let $v\in C([0,T]\times \Rd)$ be the limit point
 w.r.t. the locally uniform convergence of some subsequence of $\Set{u_\epsilon}$,
which we do not relabel.
Such $v$ exists as a consequence of Proposition \ref{aprioriest}.

We now show that for every $\lambda\in\R$, 
	\begin{equation}\label{inclusion1}
	\Set{x\in\Rd : v(t,x)>\lambda} \subseteq \tilde E^>_\lambda(t)
			\subseteq \tilde E^\geq_\lambda(t)
			\subseteq \Set{x\in\Rd : v(t,x)\geq \lambda},
	\end{equation}
with $\tilde E^>_\lambda$ and $\tilde E^\geq_\lambda$
as in \eqref{eq:limsup}.
Without loss of generality, we may discuss just the case $\lambda = 0$.
We only utilise the pointwise convergence of $u_\eps$ to $v$:
if $\bar x\in \Rd$ is such that $v(t,\bar x)=\mu$ for some $\mu>0$,
then there exists $\bar\eps>0$ such that
	\[u_\eps(t,\bar x)\geq \frac{\mu}{2}>0 \quad \text{for all $\eps<\bar\eps$},\]
and, hence, $\bar x\in  \tilde E^>_0(t)$.
Thus, we deduce $\Set{x\in\Rd : v(t,x)>0 }\subseteq \tilde E^>_0(t)$.

On the other hand, if $\bar x\in \tilde E^\geq_0(t)$,
then for all $\eps<1$ there exists $\eta_\eps<\eps$
such that $u_{\eta_\eps}(t,\bar x)\geq 0$. 
Taking the limit $\epsilon\to 0^+$, we get
			$$v(t, \bar x) = \lim_{\eps\to 0^+} u_{\eta_\eps}(t,\bar x) \geq 0.$$
		
Secondly, we prove that, for all $\lambda\in \R$, 
	\begin{equation}\label{inclusion2}
		\Set{x\in\Rd : u(t,x)>\lambda}\subseteq
			\tilde E^>_\lambda(t)\subseteq
			\tilde E^\geq_\lambda(t)\subseteq 
			\Set{x\in\Rd : u(t,x)\geq \lambda},
		\end{equation}
where $u$ is the viscosity solution to \eqref{eq:limit_pb}. 

We begin by showing that
$ \tilde E_{\lambda}^>$ and $ \tilde E_{\lambda}^\geq$ are
an outer and an inner barrier, respectively,
for the family of strict geometric subsolutions and supersolutions, respectively,
associated with the flow driven by $H_0$. 
If this assertion hold, then Theorem \ref{stm:BN} immediately gets \eqref{inclusion2},
because it states that $\Set{x\in\Rd : u(t,x)>\lambda}$ is the minimal outer  barrier
for the family of strict geometric  subsolutions,
and $\Set{x\in\Rd : u(t,x)\geq \lambda}$ is the maximal inner barrier
for the family of strict geometric supersolutions. 

We consider only the case of $ \tilde E_{0}^>(t)$,
since the proofs for $\lambda\neq 0$ and for  $ \tilde E_{0}^\geq$ are the same. 	

For some $0\leq t_0<t_1\leq T$, 
let $t\mapsto D(t)$  be an evolution of sets 	
which is a strict geometric subsolution to the anisotropic mean curvature motion
when $t\in [t_0, t_1]$.
Explicitly, we assume that there exists $\gamma>0$ such that  
			\begin{equation}\label{eq:strict-geom}
			\partial_t x(t)\cdot \hat{n}_D(t,x(t)) \leq -H_0(\partial D(t),x(t)) -\gamma
			\quad \text{for all } t\in (t_0,t_1]\text{ and }x(t)\in \partial D(t),
			\end{equation}
where $\hat{n}_D$ is the outer unit normal to $D(t)$.
We suppose as well that 
			\begin{equation}\label{eq:incl-tempo0}
			D(t_0) \subset  	\tilde E_{0}^>(t_0).	\end{equation}
Our goal is verifying that it holds $D(t_1) \subset \tilde E_{0}^>(t_1)$. 

Bearing in mind definition \eqref{eq:limsup},
we get from \eqref{eq:incl-tempo0}
that for all $\eps<1$ there exists $\eta_\eps\leq \eps$ such that
	\begin{equation}\label{ic}
		D(t_0)\subseteq E_{\eta_\eps, 0}^-(t_0).
	\end{equation} 
Since the second fundamental forms of $\partial D(t)$ are uniformly bounded for $t\in [t_0, t_1]$,
we may appeal to Theorem \ref{stm:main1}, which yields 
			\[
				\lim_{\epsilon\to 0^+} \frac{1}{\eps}H_\epsilon(D(t),x)=H_0(D(t),x)
					\quad \text{uniformly in } t\in[t_0,t_1]\text{ and }x\in\partial D(t).
			\]
Accordingly,
there exists $\bar\eps\coloneqq \bar\epsilon(\gamma)$ such that,
for all  $\eps<\bar\eps$,
			\begin{equation*}
			\partial_t x(t)\cdot \hat{n}_D(t,x(t)) \leq -\frac{1}{\eps}H_\eps(D(t),x(t)) -\frac{\gamma}{2}
			\quad \text{for all } t\in (t_0,t_1]\text{ and }x(t)\in \partial D(t),
			\end{equation*}
		or, in other words,  
		$t\mapsto D(t)$
		is a strict geometric subsolution to
		all the rescaled problems of parameter $\epsilon\in(0,\bar\epsilon)$.
		By \eqref{ic} and  Proposition \ref{stm:comp-CDNV},
		we infer that for all $\eps<\bar\eps$ there exists $\eta_\eps\leq \eps$ such that 
			\[
				D(t)\subset E_{\eta_\epsilon,0}^-(t)\quad \text{for all } t\in[t_0, t_1].
			\]
 		In view of Remark \ref{ind}, it holds as well that
 			\[
 				D(t)\subseteq \tilde E_{0}^>(t)\quad \text{for all } t\in[t_0,t_1],
			\]		
		and we get that $D(t_1)\subseteq \tilde E_{0}^>(t_1)$, as desired. 			
	 
		We can finally draw the conclusion.
		Indeed, by 	\eqref{inclusion1} and \eqref{inclusion2},
		for every $\lambda\in\R$ and $t\in [0, T]$, we have
			\[\begin{split}
			\Set{x\in\Rd : v(t,x)>\lambda} \subseteq \Set{x\in\Rd : u(t,x)\geq\lambda}, \\
			\Set{x\in\Rd : u(t,x)>\lambda} \subseteq \Set{x\in\Rd : v(t,x)\geq\lambda},
			\end{split}\]
		and, from this, Lemma \ref{stm:incl-suplev} yields $u=v$.
		\end{proof} 

\backmatter
\cleardoublepage
	
	\cleardoublepage
	\printindex
\end{document}